\documentclass[twoside,centertags                  
]{amsart}

\usepackage[cmtip,color,all,curve,
arc,poly
]{xy}

\usepackage{amssymb}
\usepackage{lscape}
\usepackage{graphicx}
\usepackage{mathrsfs}

\usepackage{float}

\usepackage[usenames,dvipsnames]{pstricks}
\usepackage{epsfig}
\usepackage{pst-grad} 
\usepackage{pst-plot} 

\newtheorem{thm}[equation]{Theorem}
\newtheorem{cor}[equation]{Corollary}
\newtheorem{lem}[equation]{Lemma}
\newtheorem{prop}[equation]{Proposition}
\theoremstyle{definition}
\newtheorem{defn}[equation]{Definition}
\theoremstyle{remark}
\newtheorem{rem}[equation]{Remark}
\newtheorem{exm}[equation]{Example}

\newcommand{\C}[1]{\mathscr{#1}}

\def\op{\mathrm{op}}
\newcommand{\operad}[1]{\mathrm{Op}(#1)}
\newcommand{\operadpc}[1]{\mathrm{Op}^{pc}(#1)}

\def\r{\rightarrow} 
\def\l{\leftarrow} 

\def\into{\rightarrowtail}
\def\onto{\twoheadrightarrow}

\newcommand{\id}[1]{\mathrm{id}_{#1}}

\newcommand{\Mod}[1]{\mathrm{Mod}(#1)}

\def\rmap{\operatorname{Map}}

\def\ner{\operatorname{Ner}}

\def\ho{\operatorname{Ho}}

\def\aut{\operatorname{Aut}}

\def\st{\stackrel} 

\def\unit{\mathbb{I}} 

\def\To{\longrightarrow}

\renewcommand{\ker}{\operatorname{Ker}}

\numberwithin{equation}{section}

\newdir{ >}{{}*!/-3pt/@{>}}
\newdir{> }{{}*!/10pt/@{>}}
\newdir{>> }{{}*!/10pt/@{>>}}

\begin{document}

\title
{Homotopy units in $A$-infinity algebras}%
\author{Fernando Muro}%
\address{Universidad de Sevilla,
Facultad de Matem\'aticas,
Departamento de \'Algebra,
Avda. Reina Mercedes s/n,
41012 Sevilla, Spain}
\email{fmuro@us.es}
\urladdr{http://personal.us.es/fmuro}


\begin{abstract}
We show that the canonical map from the associative operad to the unital associative operad is a homotopy epimorphism for a wide class of symmetric monoidal model categories. As a consequence, the space of unital associative algebra structures on a given object  is up to homotopy a subset of connected components of the space of non-unital associative algebra structures.
\end{abstract}

\maketitle


\section{Introduction}

It is well known that monoids in a monoidal category, a.k.a.~algebras, may have at most one unit. Hence, being unital can be regarded as a property, rather than a structure. In other words, the set of unital monoid structures on a given object embeds as a subset of the set of non-unital monoid structures. This fact can be deduced from the following stronger and fancier statement. 

\begin{prop}\label{epicureo}
Given a closed symmetric monoidal category $\C V$ with an initial object, the canonical morphism $\phi^{\C V}\colon\mathtt{Ass}^{\C V}\r\mathtt{uAss}^{\C V}$ from the associative operad to the unital associative operad is an epimorphism in the category $\operad{\C V}$ of non-symmetric operads in $\C V$.
\end{prop}

The canonical morphism $\phi^{\C V}$ models the forgetful functor from unital monoids to non-unital monoids.

If $\C V$ is also a model category, one is often more interested in homotopy algebra structures rather than strict algebra structures. This is because, given a monoid $M$ and a weak equivalence $\varphi\colon X\st{\sim}\r M$ in $\C V$, there need not be a monoid structure on $X$ compatible with $\varphi$, but there is always a compatible homotopy monoid structure on $X$, at least if $X$ is fibrant and cofibrant.

Homotopy  (unital) associative algebras are known as (unital) $A$-infinity algebras. They are formally defined as algebras over cofibrant resolutions of the operads $\mathtt{Ass}^{\C V}$ and $\mathtt{uAss}^{\C V}$. If $\C V=\operatorname{Top}$ is the category of topological spaces, there are nice resolutions of these operads given by associahedra \cite{hahs} and unital associahedra \cite{uass}. The cellular homology of (unital) associahedra yield resolutions for $\C V=\operatorname{Ch}(\Bbbk)$   the category of chain complexes over a commutative ring $\Bbbk$.

The strongest possible homotopical generalization of Proposition \ref{epicureo} is the following result, which is the main theorem of this paper. 

\begin{thm}\label{hepi}
Let $\C V$ be a simplicial or complicial closed symmetric monoidal model category. Assume that $\C V$ satisfies the monoid axiom and the strong unit axiom. Suppose further that $\C V$ is cofibrantly generated and has sets of generating (trivial) cofibrations with presentable sources. Then the morphism $\phi^{\C{V}}\colon \mathtt{Ass}^{\C{V}}\r \mathtt{uAss}^{\C{V}}$ is a homotopy epimorphism in $\operad{\C V}$.
\end{thm}

This means that taking derived mapping spaces in the model category $\operad{\C V}$ \cite{htnso} out of $\phi^{\C{V}}$,
\begin{equation*}
(\phi^{\C V})^{*}\colon \rmap_{\operad{\C V}}( \mathtt{uAss}^{\C{V}}, \mathcal O)
\To
\rmap_{\operad{\C V}}( \mathtt{Ass}^{\C{V}}, \mathcal O),
\end{equation*}
is essentially an inclusion of connected components for any operad $\mathcal O$, i.e.~
an injection on $\pi_0$ and an isomorphism in all  homotopy groups $\pi_n$, $n>0$, with all possible base points.
Putting $\mathcal O=\mathtt{End}_{\C V}(X)$, the endomorphism operad of an object $X$ in $\C V$, we deduce that the homotopical moduli space \cite{rezkphd} of unital $A$-infinity algebra structures on $X$ embeds as a subset of connected components of the homotopical moduli space of all $A$-infinity algebra structures on $X$. In \cite{manso} we go beyond, showing that if $X$ is perfect then $\phi^{\C{V}}$ induces  an affine Zariski open immersion of geometric moduli spaces in many homotopical algebraic geometry contexts, including derived, complicial, and brave new algebraic geometry.

A friendly characterization of the image of the injective map
$\pi_0(\phi^{\C V})^{*}$ when $\mathcal O=\mathtt{End}_{\C V}(X)$ is an endomorphism operad is possible for many $\C{V}$'s thanks to results of Lyubashenko--Manzyuk and Lurie, see Remarks \ref{laim1} and \ref{laim2}. 

Let us comment on the hypotheses of Theorem \ref{hepi}. A symmetric monoidal model category $\C V$ \cite{ammmc} is \emph{simplicial} if it is equipped with a symmetric monoidal Quillen adjunction from the category of simplicial sets,
$$\xymatrix{{\operatorname{Set}^{\Delta^{\op}}}\ar@<.5ex>[r]^-{F}&{\C{W}}.\ar@<.5ex>[l]^-{G}}$$
The upper arrow will always be the left adjoint in this kind of diagram. Similarly, 
$\C V$ is \emph{complicial} if it is equipped with a symmetric monoidal Quillen adjunction
$$\xymatrix{\operatorname{Ch}(\Bbbk)\ar@<.5ex>[r]^-{F}&{\C{W}}.\ar@<.5ex>[l]^-{G}}$$
The strong unit axiom, introduced in \cite[Definition A.9]{htnso2}, says that tensoring with a cofibrant replacement $\tilde\unit$ of the tensor unit $\unit$ preserves all weak equivalences. This obviously holds if $\unit$ is cofibrant, but it is also true in many other cases of interest, such as diagram spectra with the positive stable model structure \cite{mcds}. The rest of the hypotheses are needed to have a model structure on $\operad{\C V}$ with fibrations and weak equivalences defined as in $\C V$, see \cite[Theorem 1.1]{htnso}.

The paper is structured as follows. Section \ref{hee} studies homotopy epimorphisms in arbitrary model categories. In Section \ref{operads} we recall what we need about operads and their homotopy theory. Sections \ref{grd} and \ref{DG} contain the proof of Theorem \ref{hepi} for two special categories $\C V$: groupoids and unbounded complexes over a commutative ring. In the last section, Section \ref{transavia}, we deduce the main theorem from these two specific cases. 

We assume the reader familiarity with category theory and abstract homotopy theory. Some standard references are \cite{cwm2,hmc, hirschhorn}. For monoidal categories, functors, and adjunctions, we refer to \cite[Chapter 3]{mfsha}.

\subsection*{Acknowledgements}

A previous version of this paper only contained Theorem~\ref{hepiDG}, with a substantially more complicated proof. Lecturing about this result at the Homotopical Algebra
Summer Day in Barcelona 2012, I realised of the possibility of simplifying the proof, as it is given in Section \ref{DG}. The simplification needs the results in \cite{htnso2}, which are of independent interest. I'm grateful to the organizers of that Summer Day, Imma G\'alvez and Javier Guti\'errez, for providing such an inspiring environment. The results in \cite{htnso2} also alowed me to extend Theorem \ref{hepiDG} to a wide class of model categories, see Theorem \ref{hepi}. I wished to do this since I obtained the first proof of Theorem \ref{hepiDG}, and I'm grateful to Joe Hirsh for encouraging me to do so during the Summer Day.

I was partially supported
by the Andalusian Ministry of Economy, Innovation and Science under the grant FQM-5713, by the Spanish Ministry of Education and
Science under the MEC-FEDER grant  MTM2010-15831, and by the Government of Catalonia under the grant SGR-119-2009.

\section{Homotopy epimorphisms}\label{hee}

Recall that a morphism $f\colon X\r Y$ in a category $\C{C}$ is an epimorphism if $\C{C}(f,Z)\colon \C{C}(Y,Z)\r \C{C}(X,Z)$ is an injective map for any object $Z$ in $\C{C}$. The following characterization of epimorphisms is well known and easy to check.  

\begin{prop}\label{epi}
Let $f\colon X\r Y$ be a morphism in a category $\C{C}$. Assume the push-out
$$\xymatrix{X\ar[r]^f\ar[d]_f\ar@{}[rd]|{\text{push}}&Y\ar[d]^{i_2}\\Y\ar[r]_-{i_1}&Y\cup_XY}$$
exists. The following statements are equivalent:
\begin{enumerate}
 \item $f$ is an epimorphism.
\item $i_1$ is an isomorphism.
\item $i_2$ is an isomorphism.
\item The \emph{codiagonal} $\nabla=(1_Y,1_Y)\colon Y\cup_XY\r Y$ is an isomorphism.
\end{enumerate}
If they hold, then $i_1=i_2=\nabla^{-1}$.
\end{prop}

The strongest homotopy invariant property which generalizes the notion of injective map is the following one.

\begin{defn}\label{hess}
A map $g\colon K\r L$ between simplicial sets is a \emph{homotopy monomorphism} if it gives rise to an injection on connected components,
$$\pi_0(g)\colon\pi_0(K)\hookrightarrow\pi_0 (L),$$
and isomorphisms on homotopy groups for all possible base points $x\in K_{0}$,
$$\pi_n(g)\colon\pi_n(K,x)\st{\cong}\To\pi_n (L,g(x)),\qquad n\geq1.$$
\end{defn}

There are other obvious characterizations of homotopy monomorphisms of simplicial sets.

\begin{lem}
Given a map $g\colon K\r L$ between simplicial sets, the following statements are equivalent:
\begin{enumerate}
\item $g$ is a homotopy monomorphism.
\item $g$ corestricts to a weak equivalence between $K$ and a subset of connected components of $L$.
\item For any $x\in K_{0}$, the homotopy fiber of $g$ at $g(x)$ is contractible.
\item The homotopy fibers of $g$ are empty or contractible.
\end{enumerate}
\end{lem}

When we say that a simplicial set is \emph{contractible} we mean that it is weakly equivalent to a point. The usual terminology is `weakly contractible' but we prefer to shorten it.

There are also less obvious characterizations along the lines of the dual of Proposition \ref{epi}.

\begin{prop}\label{toenspace}
Let $g\colon K\onto L$ be a Kan fibration between Kan complexes. Consider the pull-back square
$$\xymatrix{K\times_{L}K\ar@{->>}[r]^-{p_{2}}\ar@{->>}[d]_-{p_{1}}\ar@{}[rd]|{\text{pull}}&K\ar@{->>}[d]^-{g}\\K\ar@{->>}[r]_-{g}&L}$$
The following statements are equivalent:
\begin{enumerate}
 \item $g$ is a homotopy monomorphism.
\item $p_1$ is a weak equivalence.
\item $p_2$ is a  weak equivalence.
\item The diagonal $\Delta=\binom{1_{K}}{1_{K}}\colon K\r K\times_{L}K$ is a weak equivalence.
\end{enumerate}
If they hold, then $p_1=p_2=\Delta^{-1}$  in the homotopy category of simplicial sets.
\end{prop}

\begin{proof}
Since $p_{j}\Delta=1_{K}$, $j=1,2$, the equivalences $(2)\Leftrightarrow(3)\Leftrightarrow(4)$ and the final statement are clear.

In order to show $(1)\Leftrightarrow(2)$, notice that parallel arrows in the square of the statement have essentially the same fibers. More precisely, if $F_x$ denotes the fiber of $p_1$ over a base point $x\in K_{0}$, then $F_x$ is isomorphic to the fiber of $g$ over $g(x)$.  The map $\pi_{0}(p_{1})$ is surjective since $p_{1}\Delta=1_{K}$. Therefore, by the long exact sequence in homotopy groups, $F_x$ is contractible for any $x\in K_{0}$ if and only if $p_1$ is a weak equivalence.
\end{proof}

This proposition is actually useful to characterize when an arbitrary morphism $g$ of simplicial sets is a homotopy monomorphism, since this property is homotopy invariant, so we can replace $g$ by a weakly equivalent morphism which is a Kan fibration between Kan complexes.

We now define homotopy epimorphisms in model categories via mapping spaces and homotopy monomorphisms of simplicial sets.  This definition is dual to the notion of homotopy monomorphism in \cite{htdgcat}.

\begin{defn}\label{he}
A morphism $f\colon X\rightarrow Y$ in a model category $\C{M}$ is said to be a \emph{homotopy epimorphism} if for any object $Z$ in $\C{M}$, the induced morphism on derived mapping spaces,
$$f^*=\rmap_{\C{M}}(f,Z)\colon\rmap_{\C{M}}(Y,Z)\To \rmap_{\C{M}}(X,Z),$$
is a homotopy monomorphism of simplicial sets.
\end{defn}

\begin{rem}
This definition is compatible with Definition \ref{hess}, i.e.~a morphism $g\colon K\r L$ of simplicial sets is a homotopy monomorphism in the sense of Definition \ref{hess} if and only if it is a homotopy epimorphism in the opposite of the model category of simplicial sets in the sense of Definition \ref{he}. This follows from Proposition \ref{toenspace}.
\end{rem}

The construction of derived mapping spaces we have in mind is the simplicial set $$\rmap_{\C{M}}(X,Z)=\C{M}(\tilde X, Z_\bullet),$$ where $\tilde X$ is a cofibrant resolution of $X$ and $Z_\bullet$ is a simplicial resolution of $Z$. In particular, $f^*=\rmap_{\C{M}}(f,Z)={\C{M}}(\tilde f,Z_\bullet)$, where $\tilde f\colon \tilde X\r\tilde Y$ is a lifting of $f$ to cofibrant resolutions of $X$ and $Y$, 
$$\xymatrix{\tilde{X}\ar[r]^{\tilde{f}}\ar[d]_\sim&\tilde{Y}\ar[d]^\sim\\X\ar[r]_f&Y}$$
It is usual to require cofibrant resolutions $\tilde X\st{\sim}\r X$ to be trivial fibrations from a cofibrant object. However, for us it is enough to have a weak equivalence with cofibrant source.

Notice that $f\colon X\r Y$ being a homotopy epimorphism only depends on the image of $f$ in the homotopy category $\ho\C M$. Actually, it only depends on the isomorphism class of $f$ in $\ho\C M$.

The following result characterizes homotopy epimorphisms along the lines of Proposition \ref{epi}. The dual caracterization of homotopy monomorphisms was noticed in \cite{htdgcat}. 

\begin{prop}\label{toen}
Let $f\colon X\into Y$ be a cofibration between cofibrant objects in a model category $\C{M}$. Consider the push-out square
$$\xymatrix{X\ar@{ >->}[r]^f\ar@{ >->}[d]_f\ar@{}[rd]|{\text{push}}&Y\ar@{ >->}[d]^{i_2}\\Y\ar@{ >->}[r]_-{i_1}&Y\cup_XY}$$
The following statements are equivalent:
\begin{enumerate}
 \item $f$ is a homotopy epimorphism.
\item $i_1$ is a weak equivalence.
\item $i_2$ is a  weak equivalence.
\item The codiagonal $\nabla$ is a weak equivalence.
\end{enumerate}
If they hold, then $i_1=i_2=\nabla^{-1}$  in $\ho\C{M}$.
\end{prop}

\begin{proof}
Since $\nabla i_j=1_Y$, $j=1,2$, the equivalences $(2)\Leftrightarrow(3)\Leftrightarrow(4)$ and the final statement are clear.

If we apply $\rmap_{\C{M}}(-,Z)=\C{M}(-, Z_\bullet)$ to the  push-out in the statement, we obtain a pull-back of simplicial sets consisting of Kan fibrations between Kan complexes,
$$\xymatrix{{\C{M}}( X,Z_{\bullet})\ar@{<<-}[r]^-{f^*}\ar@{<<-}[d]_-{f^*}\ar@{}[rd]|{\text{pull}}&{\C{M}}( Y,Z_{\bullet})\ar@{<<-}[d]^-{i_2^*}&\\{\C{M}}( Y,Z_{\bullet})\ar@{<<-}[r]_-{i_1^*}&{\C{M}}(Y\cup_XY,Z_{\bullet})}$$
Hence  $(1)\Leftrightarrow(2)$ follows from Proposition \ref{toenspace} and the fact that $i_1^*$ is a weak equivalence of simplicial sets  for all objects $Z$ in $\C M$ if and only if $i_1$ is a weak equivalence in~$\C{M}$.
\end{proof}

\begin{rem}\label{posttoen}
Proposition \ref{toen} is actually useful to check whether any morphism in $\C M$ is a homotopy epimorphism. A morphism $f\colon X\r Y$ in $\C M$ is a homotopy epimorphism if and only if a cofibrant resolution $\tilde f\colon\tilde  X\into\tilde  Y$ of $f$ is. Such a cofibrant resolution is a cofibration between cofibrant objects fitting into a commutative diagram
$$\xymatrix{\tilde{X}\ar@{ >->}[r]^{\tilde{f}}\ar[d]_\sim&\tilde{Y}\ar[d]^\sim\\X\ar[r]_f&Y}$$

The statements (2), (3) and (4) in Proposition \ref{toen} only depend on
the isomorphism class of the commutative square
$$\xymatrix{\tilde X\ar@{ >->}[r]^{\tilde f}\ar@{ >->}[d]_{\tilde f}\ar@{}[rd]|{\text{push}}&\tilde Y\ar@{ >->}[d]^{i_2}\\\tilde Y\ar@{ >->}[r]_-{i_1}&\tilde Y\cup_{\tilde X}\tilde Y}$$
in the homotopy category $\ho(\C M^{\square})$ of commutative squares in $\C M$. The isomorphism class of this square 
only depends on the isomorphism class of $f\colon X\r Y$ in $\ho\C M$. Actually, it can be constructed using  the derivator $\mathbb{D}\C M$ of $\C M$, which consists of all homotopy categories of diagrams in $\C M$ with the shape of a finite direct category, such as $\square$, $\bullet\r \bullet$, or $\bullet\l\bullet\r \bullet$, see \cite{ciscd}. A category is finite and direct if its nerve has finitely-many non-degenerate simplices. Let $\operatorname{Cat}$ be the category of categories and functors and  $\operatorname{Dirf}\subset \operatorname{Cat}$ the full  subcategory of finite direct categories. The derivator $\mathbb{D}\C M$ is the $2$-functor,
\begin{align*}
\mathbb{D}\C M\colon\operatorname{Dirf}^{\op}&\To\operatorname{Cat},\\
I&\;\mapsto\;\ho(\C M^{I^{\op}}).
\end{align*}
Apparently, there is a problem here with the size of $\operatorname{Cat}$. Morphism `sets' in $\operatorname{Cat}$ may be proper classes. Nevertheless, there is really no trouble, since morphism sets in $\operatorname{Dirf}$ are honest sets, so the derivator $\mathbb{D}\C M$ is insensitive to the problems of $\operatorname{Cat}$.

Let $\mathbb D\colon \operatorname{Dirf}^{\op}\r\operatorname{Cat}$ be an abstract derivator, more precisely, a right derivator  satisfying \cite[Der 5]{ciscd}, i.e.~\cite[Der 5]{ktdt}. If $e$ denotes the category with only one object and one morphism (the identity), one can give a definition of homotopy epimorphism in $\mathbb D(e)$ along the lines of (2), (3) and (4) above, extending the notion of homotopy epimorphism in $\mathbb D\C M(e)=\ho\C M$. Homotopy epimorphisms are preserved by cocontinuous morphisms of right derivators, in particular by equivalences of derivators. This observation yields a quick justification for the following corollary. The first part also follows easily from the elementary properties of mapping spaces. 
\end{rem}

\begin{cor}\label{trans1}
Let $F\colon \C M\rightleftarrows\C N\colon G$ be a Quillen adjunction between model categories and let $\mathbb L F\colon\ho \C M\rightleftarrows\ho\C N\colon \mathbb R G$ be the derived adjoint pair between homotopy categories. The functor $\mathbb L F$ preserves homotopy epimorphisms. Moreover, if $F\dashv G$ is a Quillen equivalence then $\mathbb R G$ also preserves homotopy epimorphisms. Furthermore, if $\mathbb L F$ reflects isomorphisms then it also reflects homotopy epimorphisms.
\end{cor}

\begin{rem}\label{casilp}
This is a continuation of the previous remark. Let $X\st{\bar f}\into\tilde Y\st{\sim}\r Y$ be a factorization of $f$ into a cofibration followed by a weak equivalence. If $\C M$ is left proper the gluing lemma holds, see \cite[Proposition 13.5.4]{hirschhorn}. Therefore,  the previous push-out square is isomorphic to
$$\xymatrix{X\ar@{ >->}[r]^{\bar f}\ar[d]_{f}\ar@{}[rd]|{\text{push}}&\tilde Y\ar[d]^{i_2}\\Y\ar@{ >->}[r]_-{i_1}&Y\cup_{X}\tilde Y}$$
in $\ho(\C M^{\square})$. In particular, $f$ is a homotopy epimorphism if and only if this $i_1$ is a weak equivalence. 

The gluing lemma also holds in cofibration categories \cite[II.1.2 (b)]{ah}. Hence, the same is true if $X$, $Y$ and $\tilde Y$ belong to a full subcategory of $\C M$ which is a cofibration category with cofibrations and weak equvialences defined as in $\C M$. 
\end{rem}

\section{Operads}\label{operads}

All operads considered in this paper are non-symmetric.

\begin{defn}\label{defop}
Let $\C{V}$ be a  symmetric monoidal category with tensor product $\otimes$ and tensor unit $\unit$. An \emph{operad} $\mathcal O$ in $\C V$ is a sequence $\mathcal O=\{\mathcal O(n)\}_{n\geq 0}$
of objects in $\C V$ equipped with an \emph{identity},
$$\id{\mathcal{O}}\colon\unit\r \mathcal{O}(1),$$ 
and  \emph{composition laws}, $1\leq i\leq p$, $q\geq 0$,
$$\circ_i\colon \mathcal{O}(p)\otimes \mathcal{O}(q)\To \mathcal{O}(p+q-1),$$
satisfying certain associativity and unit equations, see \cite[Remark 2.6]{htnso}. We refer to $\mathcal O(n)$ as the \emph{arity} $n$ component of $\mathcal O$.

A \emph{morphism of operads} $f\colon \mathcal O\r\mathcal P$ is a sequence of morphisms $f(n)\colon \mathcal O(n)\r\mathcal P(n)$ in $\C V$, $n\geq 0$, compatible with the identities and composition laws in the obvious way. We usually drop the arity from the notation $f(n)$ in order to simplify. We denote by $\operad{\C V}$ the category of operads in $\C V$. 
\end{defn}

\begin{rem}\label{examples}
If $\C{V}=\operatorname{Set}$ is the category of sets, the identity is simply an element $\id{\mathcal{O}}\in\mathcal{O}(1)$ and the associativity and unit equations are:
\begin{enumerate}
\item $(a\circ_i b)\circ_j c=(a\circ_j c)\circ_{i+q-1}b$ if $1\leq j<i$ and $c\in\mathcal{O}(q)$.
\item $(a\circ_i b)\circ_j c=a\circ_i( b\circ_{j-i+1}c)$ if $b\in\mathcal{O}(p)$ and $i\leq j <p+i$.
\item $\id{\mathcal{O}}\circ_1 a=a$.
\item $a\circ_i \id{\mathcal{O}}=a$.
\end{enumerate}
The same happens if $\C{V}=\operatorname{Top}$ is the category of topological spaces or the category $\Mod{\Bbbk}$ of modules over a commutative ring $\Bbbk$.

If $\C{V}=\Mod{\Bbbk}^{\mathbb Z}$ is the category of $\mathbb{Z}$-graded $\Bbbk$-modules then $\id{\mathcal{O}}$ must be in degree $0$, $\id{\mathcal{O}}\in\mathcal{O}(1)_{0}$, and  (1) must be replaced with
\begin{itemize}
\item[$(1')$] $(a\circ_i b)\circ_j c=(-1)^{|b||c|}(a\circ_j c)\circ_{i+q-1}b$ if $1\leq j<i$ and $c\in\mathcal{O}(q)$.
\end{itemize}
This reflects the use of the Koszul sign rule in the definition of the symmetry constraint for the tensor product in $\Mod{\Bbbk}^{\mathbb Z}$.

Furthermore, if $\C{V}=\operatorname{Ch}(\Bbbk)$ is the category of differential graded $\Bbbk$-modules the identity must be a cycle, $d(\id{\mathcal{O}})=0$, and the differential must behave as a derivation with respect to all composition laws,
$$d(a\circ_{i}b)=d(a)\circ_{i} b+(-1)^{|a|}a\circ_{i}d(b).$$
In this paper differentials have degree $|d|=-1$, i.e.~we consider chain complexes. 

The category $\C V=\operatorname{Grd}$ of groupoids with the cartesian symmetric monoidal structure behaves essentially as $\operatorname{Set}$. The identity $\id{\mathcal O}$ is an object of the groupoid $\mathcal O(1)$. 
\end{rem}

\begin{rem}\label{alternative}
Operads can be alternatively (and are usually) described in terms of \emph{multiplication morphisms}, $n\geq 1$, $p_{1},\dots p_{n}\geq 0$,
\begin{align*}
\mathcal O(n)\otimes \mathcal O(p_{1})\otimes\cdots\otimes \mathcal O(p_{n})&\To \mathcal O(p_{1}+\cdots+p_{n}),
\end{align*}
defined by iterating composition laws, e.g.~if $\C V$ is any of the categories in the previous remark, this morphism is given by
\begin{align*}
(a,b_{1},\dots,b_{n})\mapsto{}& a(b_{1},\dots,b_{n})\\&=(\cdots((a\circ_{1}b_{1})\circ_{p_{1}+1}b_{2})\circ_{p_{1}+p_{2}+1}\cdots)\circ_{p_{1}+\cdots+p_{n-1}+1}b_{n}. 
\end{align*}
This iterated composition can be expressed in many different ways, for instance, if $\C{V}=\operatorname{Set}$ or $\operatorname{Mod}(\Bbbk)$,
$$a(b_{1},\dots,b_{n})=(\cdots((a\circ_{n}b_{n})\circ_{n-1}b_{n-1})\circ_{n-2}\cdots)\circ_{1}b_{1}.$$
If $\C{V}=\operatorname{Mod}(\Bbbk)^{\mathbb Z}$ or $\operatorname{Ch}(\Bbbk)$ this formula would be true up to a sign determined by the Koszul rule.

The multiplication morphisms together with the identity and certain associativity and unit equations yield an equivalent definition of operad, see \cite[Remark 2.5]{htnso}. If $\C V$ is the category of sets or $\Bbbk$-modules these equations are:
\begin{align*}
&\hspace{-28pt}a(b_{1}(c_{11},\dots, c_{1p_{1}}),\dots\dots,b_{n}(c_{n1},\dots, c_{np_{n}}))\\
&=a(b_{1},\dots,b_{n})(c_{11},\dots, c_{1p_{1}},\dots\dots,c_{n1},\dots, c_{np_{n}}),\\
\id{\mathcal O}(a)&=a,\\
a(\id{\mathcal O},\dots, \id{\mathcal O})&=a.
\end{align*}
If $\C V$ is the category of graded modules we must alter the first equation with a sign, according to the Koszul rule. In the differential graded case, in addition, the differential  must behave  like a derivation with respect to the multiplication morphisms, i.e.
\begin{align*}
d(a(b_{1},\dots,b_{n}))=d(a)(b_{1},\dots,b_{n})+\sum_{i=1}^{n}(-1)^{|a|+\sum\limits_{j=1}^{i-1}|b_{j}|}a(b_{1},\dots,d(b_{i}),\dots,b_{n}).
\end{align*}
\end{rem}

\begin{exm}
The \emph{unital associative operad} $\mathtt{uAss}^{\C{V}}$ in $\C{V}$, whose algebras are unital mon\-oids, is given by $\mathtt{uAss}^{\C{V}}(n)=\unit{}$ for all $n\geq 0$. The identity of this operad $\id{\mathtt{uAss}^{\C{V}}}\colon\unit\r\mathtt{uAss}^{\C{V}}(1)$ is simply the identity morphism in $\unit{}$, and all composition laws are given by the unit isomorphism $\unit{}\otimes \unit{}\cong \unit{}$, wich is part of the symmetric monoidal structure of $\C{V}$.
\end{exm}

\begin{exm}\label{comparison}
Suppose $\C{V}$ is closed and has an initial object $\varnothing$. The \emph{associative operad} $\mathtt{Ass}^{\C{V}}$ in $\C{V}$, whose algebras are non-unital mon\-oids, is given by $\mathtt{Ass}^{\C{V}}(n)=\unit{}$ for all $n\geq 1$ and $\mathtt{Ass}^{\C{V}}(0)=\varnothing$. The operad structure is determined by the fact that the sequence of morphisms $$\phi^{\C V}\colon \mathtt{Ass}^{\C{V}}\To \mathtt{uAss}^{\C{V}}$$
given by the identity  in $\unit{}$ in all positive arities, $\phi^{\C V}(n)=\id{\unit{}}$, $n\geq 1$, is a morphism of operads.
\end{exm}

\begin{rem}\label{final}
Suppose $\C{V}$ is cartesian closed and has an initial object $\varnothing$. If $\otimes=\times$ is the cartesian product then $\unit$ is the final object in $\C V$ and $\mathtt{uAss}^{\C{V}}$ is the final operad, i.e.~the final object in $\operad{\C V}$. Hence, $\phi^{\C V}$ is the only possible map. Moreover, $\mathtt{Ass}^{\C{V}}$ is the final object of the full subcategory of operads which are $\varnothing$ in arity $0$. This happens when $\C V$ is $\operatorname{Set}$, $\operatorname{Top}$, or $\operatorname{Grd}$.
\end{rem}

\begin{rem}\label{free}
If $\C V$ is cocomplete, then so is $\operad{\C V}$. In this case, the forgetful functor from  $\operad{\C V}$ to the category $\C V^{\mathbb N}$ of sequences  $V=\{V(n)\}_{n\geq 0}$ of objects in $\C V$ has a left adjoint, the \emph{free operad} functor,
$$\xymatrix{\C V^{\mathbb N}\ar@<.5ex>[r]^-{\mathcal F}&\operad{\C V}.\ar@<.5ex>[l]^-{\text{forget}}}$$
We sometimes write $\mathcal F=\mathcal F_{\C V}$. 
This adjunction is monadic, i.e.~$\operad{\C V}$ is the category of algebras over the free operad monad. These facts allow the construction of operads by presentations. A presentatation of an operad $\mathcal O$ consists of describing $\mathcal O$ as the coequalizer of two parallel arrows between free operads,
$$\xymatrix{\mathcal F(U)\ar@<.5ex>[r]\ar@<-.5ex>[r]&\mathcal F(V)\ar[r]&\mathcal{O}.}$$

The free operad functor $\mathcal F$ can be explicitly described in terms of \emph{planted planar trees with leaves}, see \cite[\S3 and \S5]{htnso}. \emph{Planting} a tree consists of choosing a degree $1$ vertex, called \emph{root}. The \emph{degree} of a vertex is the number of adjacent edges. The \emph{planar} structure is given by an order in the set of vertices which indicates how to draw them from left to right. The \emph{leaves} are specified degree $1$ vertices different from the root. They can be distinguished in pictures since we do not draw them. We do not draw the root either, but there is no confusion since the root is placed at the bottom. Vertices are distrubuted in ascending layers according to the distance to the root. We call \emph{inner vertices} those which are drawn, i.e.~the vertices which are neither leaves nor the root. A \emph{cork} is an inner vertex of degree $1$. An \emph{inner edge} is an edge which is not adjacent to the root or to a leaf. These notions are better  
illustrated with 
a 
picture,
$$\scalebox{.7} 
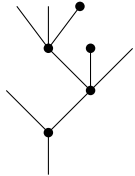
{
\begin{pspicture}(0,-1.67)(2.42,1.67)
\psdots[dotsize=0.18](0.8,-0.85)
\psdots[dotsize=0.18](1.6,-0.05)
\psdots[dotsize=0.18](1.6,0.75)
\psdots[dotsize=0.18](1.4,1.55)
\psline[linewidth=0.02cm](0.8,-0.85)(0.8,-1.65)
\psline[linewidth=0.02cm](0.8,-0.85)(1.6,-0.05)
\psline[linewidth=0.02cm](1.6,-0.05)(1.6,0.75)
\psline[linewidth=0.02cm](1.6,-0.05)(2.4,0.75)
\psline[linewidth=0.02cm](1.6,-0.05)(0.8,0.75)
\psline[linewidth=0.02cm](0.8,0.75)(1.4,1.55)
\psline[linewidth=0.02cm](0.8,1.55)(0.8,0.75)
\psline[linewidth=0.02cm](0.8,0.75)(0.2,1.55)
\psline[linewidth=0.02cm](0.8,-0.85)(0.0,-0.05)
\psdots[dotsize=0.18](0.8,0.75)
\end{pspicture} 
}$$
This is a planted planar tree with four leaves and five inner vertices, including two corks. There are four inner edges. Sometimes, abusing language, we also call leaf or root to the adjacent edge, which is what we really depict. From now on, in the whole paper,  whenever we talk about trees we mean planted planar trees with leaves. 

Let us explicitly describe the free operad construction in the category $\C V=\operatorname{Set}$ of sets. Given a sequence of sets $V=\{V(n)\}_{n\geq 0}$ the free operad $\mathcal F(V)$ is formed in arity $n$ by labelled trees with $n$ leaves. A labelling consists of assigning an element of $V$ to each inner vertex. More precisely, the \emph{arity} of a vertex is the degree minus one, 
$$\text{arity of }v=(\text{degree of }v)-1,$$
and an inner vertex of arity $n$ is labelled with an element of $V(n)$, $n\geq 0$. For instance, the previous tree can be labelled as follows,
$$\scalebox{.7} 
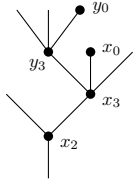
{
\begin{pspicture}(0,-1.67)(2.42,1.67)
\psdots[dotsize=0.18](0.8,-0.85)
\psdots[dotsize=0.18](1.6,-0.05)
\psdots[dotsize=0.18](1.6,0.75)
\psdots[dotsize=0.18](1.4,1.55)
\psline[linewidth=0.02cm](0.8,-0.85)(0.8,-1.65)
\psline[linewidth=0.02cm](0.8,-0.85)(1.6,-0.05)
\psline[linewidth=0.02cm](1.6,-0.05)(1.6,0.75)
\psline[linewidth=0.02cm](1.6,-0.05)(2.4,0.75)
\psline[linewidth=0.02cm](1.6,-0.05)(0.8,0.75)
\psline[linewidth=0.02cm](0.8,0.75)(1.4,1.55)
\psline[linewidth=0.02cm](0.8,1.55)(0.8,0.75)
\psline[linewidth=0.02cm](0.8,0.75)(0.2,1.55)
\psline[linewidth=0.02cm](0.8,-0.85)(0.0,-0.05)
\psdots[dotsize=0.18](0.8,0.75)
\rput(1.2,-1){$x_{2}$}
\rput(.6,.5){$y_{3}$}
\rput(2,-.2){$x_{3}$}
\rput(2,.8){$x_{0}$}
\rput(1.8,1.6){$y_{0}$}
\end{pspicture} 
}$$
Here, $x_{n},y_{n}\in V(n)$. Two labelled trees are identified if there is a simplicial isomorphism between them preserving the root, the leaves, the planar structure, and the labels. The composition law $\circ_{i}$ is defined by \emph{grafting}, i.e.~$T\circ_{i}T'$ is the labelled tree obtained by grafting the root of $T'$ onto the $i^{\text{th}}$ leaf of $T$, e.g.
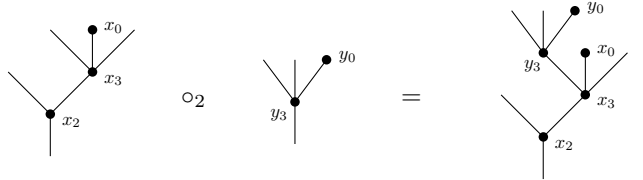
$$
\begin{array}{c}
\scalebox{.7}{
\begin{pspicture}(0,-1.67)(2.42,0.8)
\psdots[dotsize=0.18](0.8,-0.85)
\psdots[dotsize=0.18](1.6,-0.05)
\psdots[dotsize=0.18](1.6,0.75)
\psline[linewidth=0.02cm](0.8,-0.85)(0.8,-1.65)
\psline[linewidth=0.02cm](0.8,-0.85)(1.6,-0.05)
\psline[linewidth=0.02cm](1.6,-0.05)(1.6,0.75)
\psline[linewidth=0.02cm](1.6,-0.05)(2.4,0.75)
\psline[linewidth=0.02cm](1.6,-0.05)(0.8,0.75)
\psline[linewidth=0.02cm](0.8,-0.85)(0.0,-0.05)
\rput(1.2,-1){$x_{2}$}
\rput(2,-.2){$x_{3}$}
\rput(2,.8){$x_{0}$}
\end{pspicture}}
\end{array}\quad
\circ_{2}\qquad
\begin{array}{c}
\scalebox{.7}{
\begin{pspicture}(.6,.2)(1.42,1.67)
\psdots[dotsize=0.18](1.4,1.55)
\psline[linewidth=0.02cm](.8,-0.05)(0.8,0.75)
\psline[linewidth=0.02cm](0.8,0.75)(1.4,1.55)
\psline[linewidth=0.02cm](0.8,1.55)(0.8,0.75)
\psline[linewidth=0.02cm](0.8,0.75)(0.2,1.55)
\psdots[dotsize=0.18](0.8,0.75)
\rput(.5,.5){$y_{3}$}
\rput(1.8,1.6){$y_{0}$}
\end{pspicture}} 
\end{array}
\qquad=\qquad
\begin{array}{c}
\scalebox{.7}{
\begin{pspicture}(0,-1.67)(2.42,1.67)
\psdots[dotsize=0.18](0.8,-0.85)
\psdots[dotsize=0.18](1.6,-0.05)
\psdots[dotsize=0.18](1.6,0.75)
\psdots[dotsize=0.18](1.4,1.55)
\psline[linewidth=0.02cm](0.8,-0.85)(0.8,-1.65)
\psline[linewidth=0.02cm](0.8,-0.85)(1.6,-0.05)
\psline[linewidth=0.02cm](1.6,-0.05)(1.6,0.75)
\psline[linewidth=0.02cm](1.6,-0.05)(2.4,0.75)
\psline[linewidth=0.02cm](1.6,-0.05)(0.8,0.75)
\psline[linewidth=0.02cm](0.8,0.75)(1.4,1.55)
\psline[linewidth=0.02cm](0.8,1.55)(0.8,0.75)
\psline[linewidth=0.02cm](0.8,0.75)(0.2,1.55)
\psline[linewidth=0.02cm](0.8,-0.85)(0.0,-0.05)
\psdots[dotsize=0.18](0.8,0.75)
\rput(1.2,-1){$x_{2}$}
\rput(.6,.5){$y_{3}$}
\rput(2,-.2){$x_{3}$}
\rput(2,.8){$x_{0}$}
\rput(1.8,1.6){$y_{0}$}
\end{pspicture}}
\end{array}.
$$
The identity is the smallest possible tree, the only tree with no inner vertices, where the root edge is a leaf,
$$\id{}=\begin{array}{c}\begin{pspicture}(0.6,-0.6)(.6,0)
\psline[linewidth=0.02cm](0.6,0)(0.6,-0.61)
\end{pspicture}\end{array}.$$
The unit of the adjunction $V\rightarrow\mathcal F(V)$ is the sequence of maps sending $x\in V(n)$, $n>0$, to the \emph{corolla} with $n$ leaves and no corks whose only inner vertex, of arity $n$, is labelled with $x$,
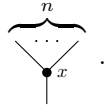
$$\begin{array}{c}
\overbrace{\scalebox{.7} 
{
\begin{pspicture}(0,-0.63)(1.22,0.7)
\psdots[dotsize=0.18](0.6,-0.01)
\psline[linewidth=0.02cm](0.6,0)(0.6,-0.61)
\psline[linewidth=0.02cm](0.6,0)(0.0,0.59)
\psline[linewidth=0.02cm](0.6,0)(1.2,0.59)
\psdots[dotsize=0.04](0.4,0.59)
\psdots[dotsize=0.04](0.6,0.59)
\psdots[dotsize=0.04](0.8,0.59)
\rput(0.9,-0.01){$x$}
\end{pspicture}
}}^{n}\end{array}.$$
For $n=0$, $x\in V(0)$ is sent to the corolla with no leaves and one cork, the \emph{lollipop}, labelled with $x$, 
$$\begin{array}{c}
\scalebox{.7} 
{
\begin{pspicture}(0,-0.43)(1.22,0.2)
\psdots[dotsize=0.18](0.6,-0.01)
\psline[linewidth=0.02cm](0.6,0)(0.6,-0.61)
\rput(0.9,-0.01){$x$}
\end{pspicture}
}\end{array}.$$

Over an arbitrary category $\C V$, the free operad $\mathcal F(V)$ generated by a sequence $V$ is given by 
$$\mathcal F(V)(n)=\coprod_TV(T).$$
Here $T$ runs over the (isomorphism classes of) trees with $n$ leaves, and 
$V(T)$ is a tensor product whose factors are objects of the sequence $V$, one for each inner vertex of $T$. More precisely, the tensor factor associated to an inner vertex $v\in T$ of arity $n$ is $V(n)$, e.g.
$$T\quad=\quad\begin{array}{c}
\scalebox{.7} 
{
\begin{pspicture}(0,-1.67)(2.42,1.67)
\psdots[dotsize=0.18](0.8,-0.85)
\psdots[dotsize=0.18](1.6,-0.05)
\psdots[dotsize=0.18](1.6,0.75)
\psdots[dotsize=0.18](1.4,1.55)
\psline[linewidth=0.02cm](0.8,-0.85)(0.8,-1.65)
\psline[linewidth=0.02cm](0.8,-0.85)(1.6,-0.05)
\psline[linewidth=0.02cm](1.6,-0.05)(1.6,0.75)
\psline[linewidth=0.02cm](1.6,-0.05)(2.4,0.75)
\psline[linewidth=0.02cm](1.6,-0.05)(0.8,0.75)
\psline[linewidth=0.02cm](0.8,0.75)(1.4,1.55)
\psline[linewidth=0.02cm](0.8,1.55)(0.8,0.75)
\psline[linewidth=0.02cm](0.8,0.75)(0.2,1.55)
\psline[linewidth=0.02cm](0.8,-0.85)(0.0,-0.05)
\psdots[dotsize=0.18](0.8,0.75)
\end{pspicture} 
}\end{array}
\leadsto
\begin{array}{c}
\scalebox{1} 
{
\begin{pspicture}(0,-1.67)(2.42,1.67)
\rput(0.8,-0.85){$V(2)$}
\rput(1.2,-0.45){$\otimes$}
\rput(1.6,-0.05){$V(3)$}
\rput(1.6,0.35){$\otimes$}
\rput(1.,0.35){$\otimes$}
\rput(1.6,0.75){$V(0)$}
\rput(1.4,1.55){$V(0)$}
\rput(1.,1.15){$\otimes$}
\rput(0.5,0.75){$V(3)$}
\end{pspicture} 
}\end{array}=\quad  V(T).$$
The composition laws can be described as formal graftings, as above.  The identity is the inclusion of the factor $V(|)=\unit$ of the coproduct $\mathcal F(V)(1)$, which is the tensor unit since $|$ has no inner vertices. The unit $V\rightarrow\mathcal F(V)$ is the sequence of morphisms $V(n)\rightarrow\mathcal F(V)(n)$  given by the inclusion of the factor of the coproduct corresponding to corolla with $n$ leaves and  no corks for $n>0$, and to the lollipop for $n=0$.

If $\C V=\operatorname{Mod}(\Bbbk)$ is the category of $\Bbbk$-modules, labelled trees as above denote tensors, e.g.
$$\begin{array}{c}
\scalebox{.7} 
{
\begin{pspicture}(0,-1.67)(2.42,1.67)
\psdots[dotsize=0.18](0.8,-0.85)
\psdots[dotsize=0.18](1.6,-0.05)
\psdots[dotsize=0.18](1.6,0.75)
\psdots[dotsize=0.18](1.4,1.55)
\psline[linewidth=0.02cm](0.8,-0.85)(0.8,-1.65)
\psline[linewidth=0.02cm](0.8,-0.85)(1.6,-0.05)
\psline[linewidth=0.02cm](1.6,-0.05)(1.6,0.75)
\psline[linewidth=0.02cm](1.6,-0.05)(2.4,0.75)
\psline[linewidth=0.02cm](1.6,-0.05)(0.8,0.75)
\psline[linewidth=0.02cm](0.8,0.75)(1.4,1.55)
\psline[linewidth=0.02cm](0.8,1.55)(0.8,0.75)
\psline[linewidth=0.02cm](0.8,0.75)(0.2,1.55)
\psline[linewidth=0.02cm](0.8,-0.85)(0.0,-0.05)
\psdots[dotsize=0.18](0.8,0.75)
\rput(1.2,-1){$x_{2}$}
\rput(.6,.5){$y_{3}$}
\rput(2,-.2){$x_{3}$}
\rput(2,.8){$x_{0}$}
\rput(1.8,1.6){$y_{0}$}
\end{pspicture} 
}   
  \end{array}=x_2\otimes x_3\otimes y_3\otimes y_0\otimes x_0.
$$
Similarly if $\C V=\operatorname{Mod}(\Bbbk)^{\mathbb Z},\operatorname{Ch}(\Bbbk),$ etc.
\end{rem}

\begin{rem}\label{treecomp}
Labelled trees can also be used to represent an iterated composition in an operad $\mathcal O$. Labels, as above, are placed in inner vertices, and the arity of the label must coincide with the arity of the vertex. The way of composing elements is determied by the geometry of the tree, e.g.
$$
\begin{array}{c}
\scalebox{.7} 
{
\begin{pspicture}(0,-1.67)(2.42,1.67)
\psdots[dotsize=0.18](0.8,-0.85)
\psdots[dotsize=0.18](1.6,-0.05)
\psdots[dotsize=0.18](1.6,0.75)
\psdots[dotsize=0.18](1.4,1.55)
\psline[linewidth=0.02cm](0.8,-0.85)(0.8,-1.65)
\psline[linewidth=0.02cm](0.8,-0.85)(1.6,-0.05)
\psline[linewidth=0.02cm](1.6,-0.05)(1.6,0.75)
\psline[linewidth=0.02cm](1.6,-0.05)(2.4,0.75)
\psline[linewidth=0.02cm](1.6,-0.05)(0.8,0.75)
\psline[linewidth=0.02cm](0.8,0.75)(1.4,1.55)
\psline[linewidth=0.02cm](0.8,1.55)(0.8,0.75)
\psline[linewidth=0.02cm](0.8,0.75)(0.2,1.55)
\psline[linewidth=0.02cm](0.8,-0.85)(0.0,-0.05)
\psdots[dotsize=0.18](0.8,0.75)
\rput(1.2,-1){$x_{2}$}
\rput(.6,.5){$y_{3}$}
\rput(2,-.2){$x_{3}$}
\rput(1.9,.9){$x_{0}$}
\rput(1.8,1.6){$y_{0}$}
\end{pspicture} 
}
\end{array}\quad=
\quad 
x_2(-,x_3(y_3(-,-,y_0),x_0,-)).$$
The labelled trees of a free operad are also iterated compositions in this sense. 
Over an arbitrary category $\C V$, whose objects may not be sets with structure, composition of labelled trees is really a morphism $\mathcal O(T)\r\mathcal O(n)$ built from composition laws, where $n$ is the number of leaves of $T$.
One can actually give yet another characterization of  the category of operads in terms of the objects $\mathcal O(T)$ and morphisms between them induced by maps of trees, see \cite[\S3]{htnso}.
\end{rem}

\begin{rem}\label{binary}
Let $\C V$ be a any (co)complete closed symmetric monoidal category. 
In \cite[\S 5]{htnso}, an explicit construction of the push-out in $\operad{\C V}$ of a diagram of the form $$\mathcal O\longleftarrow\mathcal F(U)\st{\mathcal F(f)}\To \mathcal F(V)$$
is given.  In the following sections we make an extensive use of coproducts of the form $\mathcal O\amalg\mathcal F(V)$ in $\operad{\C V}$. Such a coproduct is a push-out as above with $U$ the initial sequence, which consists of the initial object $U(n)=\varnothing$ in $\C V$ in each arity $n\geq 0$. In this remark we explain in a detailed way how the general construction in \cite{htnso} looks like in this specific case. 

In general, a push-out in $\operad{\C V}$ as above can be decomposed in $\C V^{\mathbb N}$ as sequential colimit 
$$\mathcal O=P_0\st{\varphi_1}\To P_1\r\cdots\r P_{t-1}\st{\varphi_t}\To P_t\r\cdots.$$
Here, each $\varphi_t(n)\colon P_{t-1}(n)\r P_{t}(n)$ is a push-out of a coproduct  of maps in $\C V$. 
The source of each of these maps is a tensor product in $\C V$ containing at least one component of $U$, in particular it is $\varnothing$ if $U$ is the initial sequence. Moreover, if $U$ is the initial sequence $P_t(n)$ is obtained from $P_{t-1}(n)$ by adding new coproduct factors, $t\geq 1$. This coproduct is indexed by the set of trees with $t$ inner vertices and $n$ leaves and such that all leaves are vertices of even level. 
The \emph{level} of a vertex $v\in T$  is the distance to the root, i.e.~the number of edges in the shortest path from $v$ to the root. Therefore, the operad $\mathcal O\amalg\mathcal F(V)$  in arity $n$ can be decomposed as a coproduct in $\C V$ indexed by the set of all trees with exactly $n$ leaves all of which have even level. We now proceed with this explicit description. 

For any tree $T$, we define an object $(\mathcal O, V)(T)$ in $\C V$ which is a tensor product of components of $\mathcal O$ and $V$ indexed by the inner vertices of $T$. Suppose $v\in T$ is an inner vertex of arity $m$. The corresponding tensor factor is $V(m)$ if $v$ has even level and
$\mathcal O(m)$ if $v$ has odd level, e.g.~
$$T\quad=\quad\begin{array}{c}
\scalebox{.7} 
{
\begin{pspicture}(0,-1.67)(2.42,1.67)
\psdots[dotsize=0.18](0.8,-0.85)
\psdots[dotsize=0.18](1.6,-0.05)
\psdots[dotsize=0.18](1.6,0.75)
\psdots[dotsize=0.18](1.4,1.55)
\psline[linewidth=0.02cm](0.8,-0.85)(0.8,-1.65)
\psline[linewidth=0.02cm](0.8,-0.85)(1.6,-0.05)
\psline[linewidth=0.02cm](1.6,-0.05)(1.6,0.75)
\psline[linewidth=0.02cm](1.6,-0.05)(2.4,0.75)
\psline[linewidth=0.02cm](1.6,-0.05)(0.8,0.75)
\psline[linewidth=0.02cm](0.8,0.75)(1.4,1.55)
\psline[linewidth=0.02cm](0.8,1.55)(0.8,0.75)
\psline[linewidth=0.02cm](0.8,0.75)(0.2,1.55)
\psline[linewidth=0.02cm](0.8,-0.85)(0.0,-0.05)
\psline[linewidth=0.02cm](2.4,0.75)(2.4,1.55)
\psdots[dotsize=0.18](0.8,0.75)
\psdots[dotsize=0.18](2.4,0.75)
\end{pspicture} 
}\end{array}
\leadsto
\begin{array}{c}
\scalebox{1} 
{
\begin{pspicture}(0,-1.67)(2.42,1.67)
\rput(0.8,-0.85){$\mathcal O(2)$}
\rput(1.2,-0.45){$\otimes$}
\rput(1.6,-0.05){$V(3)$}
\rput(1.6,0.35){$\otimes$}
\rput(2.2,0.35){$\otimes$}
\rput(1.,0.35){$\otimes$}
\rput(1.6,0.75){$\mathcal O(0)$}
\rput(1.4,1.55){$V(0)$}
\rput(1.,1.15){$\otimes$}
\rput(0.5,0.75){$\mathcal O(3)$}
\rput(2.7,0.75){$\mathcal O(1)$}
\end{pspicture} 
}\end{array}\quad=\quad(\mathcal O, V)(T).$$

The arity $n$ component of the coproduct $\mathcal O\amalg\mathcal F(V)$ is
$$\left(\mathcal O\amalg\mathcal F(V)\right)(n)=\coprod_{T}(\mathcal O, V)(T),$$
where $T$ runs over the trees with $n$ leaves  of even level and no leaves of odd level. 

In order to explain how the composition laws are defined, we look at the case $\C V=\operatorname{Set}$ so as to work with labelled trees. An element in $(\mathcal O, V)(T)$ can be seen as a labelling of $T$,
$$
\begin{array}{c}
\scalebox{.7} 
{
\begin{pspicture}(0,-1.67)(2.42,1.67)
\psdots[dotsize=0.18](0.8,-0.85)
\psdots[dotsize=0.18](1.6,-0.05)
\psdots[dotsize=0.18](1.6,0.75)
\psdots[dotsize=0.18](1.4,1.55)
\psline[linewidth=0.02cm](0.8,-0.85)(0.8,-1.65)
\psline[linewidth=0.02cm](0.8,-0.85)(1.6,-0.05)
\psline[linewidth=0.02cm](1.6,-0.05)(1.6,0.75)
\psline[linewidth=0.02cm](1.6,-0.05)(2.4,0.75)
\psline[linewidth=0.02cm](1.6,-0.05)(0.8,0.75)
\psline[linewidth=0.02cm](0.8,0.75)(1.4,1.55)
\psline[linewidth=0.02cm](0.8,1.55)(0.8,0.75)
\psline[linewidth=0.02cm](0.8,0.75)(0.2,1.55)
\psline[linewidth=0.02cm](0.8,-0.85)(0.0,-0.05)
\psline[linewidth=0.02cm](2.4,0.75)(2.4,1.55)
\psdots[dotsize=0.18](0.8,0.75)
\psdots[dotsize=0.18](2.4,0.75)
\rput(1.2,-1){$x_{2}$}
\rput(.6,.5){$x_{3}$}
\rput(2,-.2){$y_{3}$}
\rput(1.9,.9){$x_{0}$}
\rput(1.8,1.6){$y_{0}$}
\rput(2.8,.8){$x_{1}$}
\end{pspicture} 
}
\end{array}
,\qquad
x_{i}\in\mathcal O(i),\qquad y_{j}\in V(j).
$$
The $i^{\text{th}}$ composition law in $\mathcal O\amalg\mathcal F(V)$ is defined as follows. We take two such labelled trees, graft the root of the second one into the $i^{\text{th}}$ leaf of the first one, and contract the newly created inner edge. All vertices keep their label except for the vertex resulting from the contraction. This vertex is formed by merging two vertices, $v$ and $w$, labelled with elements in $\mathcal O$, $x_{v}$ and $x_{w}$, respectively.  Assume the $i^{\text{th}}$ leaf of the first tree is the $j^{\text{th}}$ incomming edge of $v$, i.e.~the $j^{\text{th}}$ edge adjacent to $v$ situated above. Then the label of the shrinked edge is $x_{v}\circ_{j}x_{w}$. Let us see an example,
$$\begin{array}{c}
\scalebox{.7} 
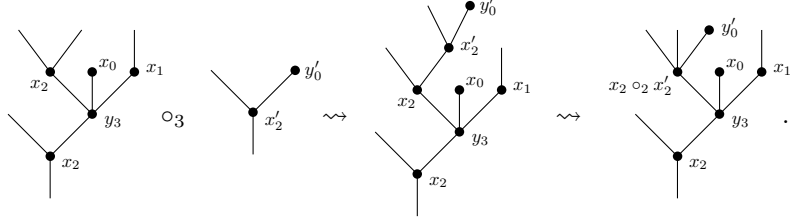
{
\begin{pspicture}(0,-1.67)(2.42,1.67)
\psdots[dotsize=0.18](0.8,-0.85)
\psdots[dotsize=0.18](1.6,-0.05)
\psdots[dotsize=0.18](1.6,0.75)
\psline[linewidth=0.02cm](0.8,-0.85)(0.8,-1.65)
\psline[linewidth=0.02cm](0.8,-0.85)(1.6,-0.05)
\psline[linewidth=0.02cm](1.6,-0.05)(1.6,0.75)
\psline[linewidth=0.02cm](1.6,-0.05)(2.4,0.75)
\psline[linewidth=0.02cm](1.6,-0.05)(0.8,0.75)
\psline[linewidth=0.02cm](0.8,0.75)(1.4,1.55)
\psline[linewidth=0.02cm](0.8,0.75)(0.2,1.55)
\psline[linewidth=0.02cm](0.8,-0.85)(0.0,-0.05)
\psline[linewidth=0.02cm](2.4,0.75)(2.4,1.55)
\psdots[dotsize=0.18](0.8,0.75)
\psdots[dotsize=0.18](2.4,0.75)
\rput(1.2,-1){$x_{2}$}
\rput(.6,.5){$x_{2}$}
\rput(2,-.2){$y_{3}$}
\rput(1.9,.9){$x_{0}$}
\rput(2.8,.8){$x_{1}$}
\end{pspicture} 
}
\end{array}
\circ_{3}
\begin{array}{c}
\scalebox{.7} 
{
\begin{pspicture}(0,-1.67)(1.6,0)
\psdots[dotsize=0.18](0.8,-0.85)
\psline[linewidth=0.02cm](0.8,-0.85)(0.8,-1.65)
\psline[linewidth=0.02cm](0.8,-0.85)(1.6,-0.05)
\psline[linewidth=0.02cm](0.8,-0.85)(0,-0.05)
\psdots[dotsize=0.18](1.6,-0.05)
\rput(1.2,-1){$x_{2}'$}
\rput(1.95,-0.05){$y_{0}'$}
\end{pspicture} 
}
\end{array}
\leadsto
\begin{array}{c}
\scalebox{.7} 
{
\begin{pspicture}(0,-1.67)(2.42,2.35)
\psdots[dotsize=0.18](0.8,-0.85)
\psdots[dotsize=0.18](1.6,-0.05)
\psdots[dotsize=0.18](1.6,0.75)
\psline[linewidth=0.02cm](0.8,-0.85)(0.8,-1.65)
\psline[linewidth=0.02cm](0.8,-0.85)(1.6,-0.05)
\psline[linewidth=0.02cm](1.6,-0.05)(1.6,0.75)
\psline[linewidth=0.02cm](1.6,-0.05)(2.4,0.75)
\psline[linewidth=0.02cm](1.6,-0.05)(0.8,0.75)
\psline[linewidth=0.02cm](1.4,1.55)(1.8,2.35)
\psline[linewidth=0.02cm](1,2.35)(1.4,1.55)
\psline[linewidth=0.02cm](0.8,.75)(0.2,1.55)
\psline[linewidth=0.02cm](0.8,-0.85)(0.0,-0.05)
\psline[linewidth=0.02cm](2.4,0.75)(2.4,1.55)
\psdots[dotsize=0.18](0.8,0.75)
\psdots[dotsize=0.18](2.4,0.75)
\rput(1.2,-1){$x_{2}$}
\rput(.6,.5){$x_{2}$}
\rput(2,-.2){$y_{3}$}
\rput(1.9,.9){$x_{0}$}
\rput(1.8,1.6){$x_{2}'$}
\rput(2.8,.8){$x_{1}$}
\psline[linewidth=0.02cm](0.8,0.75)(1.4,1.55)
\psdots[dotsize=0.18](1.4,1.55)
\psdots[dotsize=0.18](1.8,2.35)
\rput(2.15,2.35){$y_{0}'$}
\end{pspicture} 
}
\end{array}\quad\leadsto\quad
\begin{array}{c}
\scalebox{.7} 
{
\begin{pspicture}(0,-1.67)(2.42,1.67)
\psdots[dotsize=0.18](0.8,-0.85)
\psdots[dotsize=0.18](1.6,-0.05)
\psdots[dotsize=0.18](1.6,0.75)
\psdots[dotsize=0.18](1.4,1.55)
\psline[linewidth=0.02cm](0.8,-0.85)(0.8,-1.65)
\psline[linewidth=0.02cm](0.8,-0.85)(1.6,-0.05)
\psline[linewidth=0.02cm](1.6,-0.05)(1.6,0.75)
\psline[linewidth=0.02cm](1.6,-0.05)(2.4,0.75)
\psline[linewidth=0.02cm](1.6,-0.05)(0.8,0.75)
\psline[linewidth=0.02cm](0.8,0.75)(1.4,1.55)
\psline[linewidth=0.02cm](0.8,1.55)(0.8,0.75)
\psline[linewidth=0.02cm](0.8,0.75)(0.2,1.55)
\psline[linewidth=0.02cm](0.8,-0.85)(0.0,-0.05)
\psline[linewidth=0.02cm](2.4,0.75)(2.4,1.55)
\psdots[dotsize=0.18](0.8,0.75)
\psdots[dotsize=0.18](2.4,0.75)
\rput(1.2,-1){$x_{2}$}
\rput(.1,.5){$x_{2}\circ_{2}x_{2}'$}
\rput(2,-.2){$y_{3}$}
\rput(1.9,.9){$x_{0}$}
\rput(1.8,1.6){$y_{0}'$}
\rput(2.8,.8){$x_{1}$}
\end{pspicture} 
}
\end{array}.$$
The inclusion of the first factor $\mathcal O\rightarrow\mathcal O\amalg\mathcal F(V)$ sends $x\in \mathcal O(n)$ to
$$\begin{array}{c}
\overbrace{\scalebox{.7} 
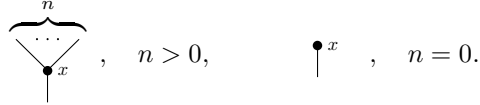
{
\begin{pspicture}(0,-0.63)(1.22,0.7)
\psdots[dotsize=0.18](0.6,-0.01)
\psline[linewidth=0.02cm](0.6,0)(0.6,-0.61)
\psline[linewidth=0.02cm](0.6,0)(0.0,0.59)
\psline[linewidth=0.02cm](0.6,0)(1.2,0.59)
\psdots[dotsize=0.04](0.4,0.59)
\psdots[dotsize=0.04](0.6,0.59)
\psdots[dotsize=0.04](0.8,0.59)
\rput(0.9,-0.01){$x$}
\end{pspicture}
}}^{n}\end{array},\quad n>0,\qquad
\begin{array}{c}
\scalebox{.7} 
{
\begin{pspicture}(0,-0.43)(1.22,0.2)
\psdots[dotsize=0.18](0.6,-0.01)
\psline[linewidth=0.02cm](0.6,0)(0.6,-0.61)
\rput(0.9,-0.01){$x$}
\end{pspicture}
}\end{array},\quad n=0.
$$
The inclusion of the second factor $\mathcal F(V)\rightarrow\mathcal O\amalg\mathcal F(V)$ sends $y\in V(n)$ to
$$\begin{array}{c}
\overbrace{\scalebox{.7} 
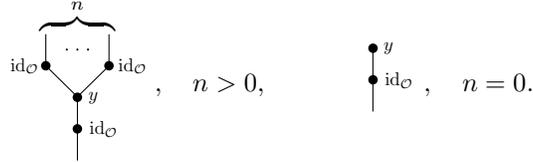
{
\begin{pspicture}(0,-1.21)(1.2,1.2)
\psdots[dotsize=0.18](0.6,-0.01)
\psline[linewidth=0.02cm](0.6,0)(0.6,-0.61)
\psline[linewidth=0.02cm](0.6,-1.21)(0.6,-0.61)
\psline[linewidth=0.02cm](0.6,0)(0.0,0.59)
\psline[linewidth=0.02cm](0.6,0)(1.2,0.59)
\psdots[dotsize=0.04](0.4,0.9)
\psdots[dotsize=0.04](0.6,0.9)
\psdots[dotsize=0.04](0.8,0.9)
\rput(.9,-0.01){$y$}
\psdots[dotsize=0.18](0.0,0.59)
\psdots[dotsize=0.18](1.2,0.59)
\psdots[dotsize=0.18](0.6,-0.61)
\rput(1.1,-0.61){$\id{\mathcal O}$}
\psline[linewidth=0.02cm](0.0,0.59)(0.0,1.19)
\psline[linewidth=0.02cm](1.2,0.59)(1.2,1.19)
\rput(-0.4,0.59){$\id{\mathcal O}$}
\rput(1.65,0.59){$\id{\mathcal O}$}
\end{pspicture}
}}^{n}\end{array}\quad,\quad n>0,\qquad
\begin{array}{c}
\scalebox{.7} 
{
\begin{pspicture}(0,-1.21)(1.2,0)
\psdots[dotsize=0.18](0.6,-0.01)
\psline[linewidth=0.02cm](0.6,0)(0.6,-0.61)
\psline[linewidth=0.02cm](0.6,-1.21)(0.6,-0.61)
\rput(.9,-0.01){$y$}
\psdots[dotsize=0.18](0.6,-0.61)
\rput(1.1,-0.61){$\id{\mathcal O}$}
\end{pspicture}
}\end{array},\quad n=0.$$
The labelled trees above represent  iterated compositions in $\mathcal O\amalg\mathcal F(V)$ in the sense of Remark \ref{treecomp}. 

The cases $\C V=\operatorname{Grd}$, $\operatorname{Mod}(\Bbbk)$, $\operatorname{Mod}(\Bbbk)^{\mathbb Z}$ and $\operatorname{Ch}(\Bbbk)$ are analogous. We leave the reader to formulate an `element free' description of composition laws as in \cite[\S5]{htnso}.
\end{rem}

\begin{exm}\label{sets}
The operad $\mathtt{Ass}^{\C V}$ admits a presentation with one arity $2$ generator and one arity $3$ relation, i.e.~it fits into a coequalizer
$$\xymatrix{\mathcal F(\unit[3])\ar@<.5ex>[r]^-{r_{1}}\ar@<-.5ex>[r]_-{r_{2}}&\mathcal F(\unit[2])\ar[r]^{g}&\mathtt{Ass}^{\C V}.}$$
Here, given an object $X$ in $\C V$ and $n\geq 0$, we denote by $X[n]$ the sequence consisting of $X$ concentrated in arity $n$ and the initial object $\varnothing$ elsewhere.
The morphism $g$ is induced by the identity $\unit=\mathtt{Ass}^{\C V}(2)$. Moreover, $\mathcal F(\unit[2])(2)=\unit$ 
and $r_{i}$ is induced by the morphism
$$\unit\cong\mathcal F(\unit[2])(2)\otimes\mathcal F(\unit[2])(2)\st{\circ_{i
}}\To\mathcal F(\unit[2])(3)
,\quad i=1,2.$$

For $\C V=\operatorname{Set}$ and $\operatorname{Mod}(\Bbbk)$, this translates into a 
generator $$\mu\in \mathtt{Ass}^{\C V}(2)$$ and a relation
$$\mu\circ_{1}\mu=\mu\circ_{2}\mu \in \mathtt{Ass}^{\C V}(3).$$
If  $\C V=\operatorname{Mod}(\Bbbk)^{\mathbb Z}$ we must specify that $\mu$ is in degree $0$, $\mu\in \mathtt{Ass}^{\C V}(2)_{0}$. For  $\C V=\operatorname{Ch}(\Bbbk)$, $\mu$ is in addition a cycle, $d(\mu)=0$. Moreover, if  $\C V=\operatorname{Grd}$, $\mu$ is an object.

We denote by $\mu^{n-1}$ the arity $n$ element  obtained by composing $n-1$ copies of $\mu$,  $n\geq 1$, e.g.
$$\mu^{n-1}=(\cdots((\mu\circ_{1}\mu)\circ_{1}\mu)\circ_{1}\cdots)\circ_{1}\mu.$$ 
Here the brackets and the subscripts do not really matter, because of the defining relation.

Although $\mathtt{Ass}^{\C V}$ is not a free operad, it is customary to depict $\mu^{n-1}$ as a corolla with $n$ leaves and no corks, $n\geq 2$,
$$\mu^{n-1}\quad=\begin{array}{c}
\overbrace{\scalebox{.7} 
{
\begin{pspicture}(0,-0.63)(1.22,0.63)
\psdots[dotsize=0.18](0.6,-0.01)
\psline[linewidth=0.02cm](0.6,0)(0.6,-0.61)
\psline[linewidth=0.02cm](0.6,0)(0.0,0.59)
\psline[linewidth=0.02cm](0.6,0)(1.2,0.59)
\psdots[dotsize=0.04](0.4,0.59)
\psdots[dotsize=0.04](0.6,0.59)
\psdots[dotsize=0.04](0.8,0.59)
\end{pspicture}
}}^n\\[10pt]\end{array}.$$
Hence, the composition laws in $\mathtt{Ass}^{\C V}$ are given by grafting and then contracting the newly created inner edge,
$$\mu^{p-1}\circ_{i}\mu^{q-1}\quad\leadsto\begin{array}{c}
\scalebox{.7} 
{
\begin{pspicture}(0,-2.0567188)(2.7190626,2.0767188)
\psline[linewidth=0.02cm](1.2,-0.83671874)(1.2,-2.0367188)
\psline[linewidth=0.02cm](1.2,-0.83671874)(0.0,0.36328125)
\psline[linewidth=0.02cm](1.2,-0.83671874)(2.6,0.36328125)
\psdots[dotsize=0.04](0.4,0.36328125)
\psdots[dotsize=0.04](0.6,0.36328125)
\psdots[dotsize=0.04](0.8,0.36328125)
\psdots[dotsize=0.04](2.2,0.36328125)
\psdots[dotsize=0.04](2.0,0.36328125)
\psdots[dotsize=0.04](1.8,0.36328125)
\psline[linewidth=0.02cm](1.4,0.36328125)(1.2,-0.83671874)
\psline[linewidth=0.02cm](1.4,0.36328125)(0.6,1.5632813)
\psline[linewidth=0.02cm](1.4,0.36328125)(2.2,1.5632813)
\psdots[dotsize=0.04](1.2,1.5632813)
\psdots[dotsize=0.04](1.4,1.5632813)
\psdots[dotsize=0.04](1.6,1.5632813)
\psdots[dotsize=0.18](1.2,-0.83671874)
\psdots[dotsize=0.18](1.4,0.36328125)
\usefont{T1}{ptm}{m}{n}
\rput(1.4,1.9){$\overbrace{\,\quad\qquad\qquad}^{\displaystyle q}$}
\usefont{T1}{ptm}{m}{n}
\rput(.45,.7){$\overbrace{\qquad\quad}^{\displaystyle i-1}$}
\usefont{T1}{ptm}{m}{n}
\rput(2.2,0.7){$\overbrace{\qquad\quad}^{\;\;\displaystyle p-i}$}
\end{pspicture} 
}\\[10pt]\end{array}
\leadsto
\begin{array}{c}
\overbrace{\scalebox{.7} 
{
\begin{pspicture}(0,-0.63)(1.22,0.63)
\psdots[dotsize=0.18](0.6,-0.01)
\psline[linewidth=0.02cm](0.6,0)(0.6,-0.61)
\psline[linewidth=0.02cm](0.6,0)(0.0,0.59)
\psline[linewidth=0.02cm](0.6,0)(1.2,0.59)
\psdots[dotsize=0.04](0.4,0.59)
\psdots[dotsize=0.04](0.6,0.59)
\psdots[dotsize=0.04](0.8,0.59)
\end{pspicture}
}}^{p+q-1}\\[10pt]\end{array}\;=\quad\mu^{p+q-2}.
$$

The operad $\mathtt{uAss}^{\C V}$ admits a presentation extending the presentation of $\mathtt{Ass}^{\C V}$ with one more generator in arity $0$ and two more relations in arity $1$,
$$\xymatrix@C=30pt{\mathcal F((\unit\amalg \unit)[1])\ar@<.5ex>[r]^-{r'_{1}}\ar@<-.5ex>[r]_-{r'_{2}}&\mathcal F(\unit[0])\amalg \mathtt{Ass}^{\C V}\ar[r]^-{(g',\phi^{\C V})}&\mathtt{uAss}^{\C V}.}$$
The morphism  $g'$ is induced by the identity $\unit=\mathtt{uAss}^{\C V}(0)$. We leave the reader to give an abstract description of the morphisms $r_{1}'$ and $r_{2}'$ defining the relations.

If $\C V$ is one of the examples considered above, the new generator is denoted by
$$u\in \mathtt{uAss}^{\C V}(0).$$ 
We must specify that $u$ is in degree $0$, a cycle, or an object, according to which $\C V$ we are working with. In all cases, the two relations are
$$\mu\circ_{1}u=\id{}=\mu\circ_{2}u  \in \mathtt{uAss}^{\C V}(1).$$
In terms of trees, $u$ us represented by the trivial corolla,
$$u\quad=\begin{array}{c}
\scalebox{.7} 
{
\begin{pspicture}(0.6,-0.5)(.6,0)
\psline[linewidth=0.02cm](0.6,0)(0.6,-0.61)
\psdots[dotsize=0.18,fillstyle=solid,dotstyle=o](0.6,-0.01)
\end{pspicture}
}\end{array}.$$
Here, the cork is depicted in white for reasons that will be clear in the proof of Lemma \ref{uuu} below.
The composition laws are given as above, including also the following case, $n>2$,
$$\mu^{n-1}\circ_{i}u\quad\leadsto\begin{array}{c}
\scalebox{.7} 
{
\begin{pspicture}(0,-2.0567188)(2.7190626,1.2)
\psline[linewidth=0.02cm](1.2,-0.83671874)(1.2,-2.0367188)
\psline[linewidth=0.02cm](1.2,-0.83671874)(0.0,0.36328125)
\psline[linewidth=0.02cm](1.2,-0.83671874)(2.6,0.36328125)
\psdots[dotsize=0.04](0.4,0.36328125)
\psdots[dotsize=0.04](0.6,0.36328125)
\psdots[dotsize=0.04](0.8,0.36328125)
\psdots[dotsize=0.04](2.2,0.36328125)
\psdots[dotsize=0.04](2.0,0.36328125)
\psdots[dotsize=0.04](1.8,0.36328125)
\psline[linewidth=0.02cm](1.4,0.36328125)(1.2,-0.83671874)
\psdots[dotsize=0.18](1.2,-0.83671874)
\psdots[dotsize=0.18,fillstyle=solid,dotstyle=o](1.4,0.36328125)
\usefont{T1}{ptm}{m}{n}
\rput(.45,.7){$\overbrace{\qquad\quad}^{\displaystyle i-1}$}
\usefont{T1}{ptm}{m}{n}
\rput(2.2,0.7){$\overbrace{\qquad\quad}^{\displaystyle n-i}$}
\end{pspicture} 
}\\[10pt]\end{array}
\leadsto
\begin{array}{c}
\overbrace{\scalebox{.7} 
{
\begin{pspicture}(0,-0.63)(1.22,0.63)
\psdots[dotsize=0.18](0.6,-0.01)
\psline[linewidth=0.02cm](0.6,0)(0.6,-0.61)
\psline[linewidth=0.02cm](0.6,0)(0.0,0.59)
\psline[linewidth=0.02cm](0.6,0)(1.2,0.59)
\psdots[dotsize=0.04](0.4,0.59)
\psdots[dotsize=0.04](0.6,0.59)
\psdots[dotsize=0.04](0.8,0.59)
\end{pspicture}
}}^{n-1}\\[10pt]\end{array}\;=\quad\mu^{n-2}.
$$
There are, however, two exceptions,
$$\mu\circ_{1}u\quad\leadsto\begin{array}{c}
\scalebox{.7} 
{
\begin{pspicture}(0,-0.63)(1.22,0.7)
\psdots[dotsize=0.18](0.6,-0.01)
\psline[linewidth=0.02cm](0.6,0)(0.6,-0.61)
\psline[linewidth=0.02cm](0.6,0)(0.0,0.59)
\psline[linewidth=0.02cm](0.6,0)(1.2,0.59)
\psdots[dotsize=0.18,fillstyle=solid,dotstyle=o](0,0.59)
\end{pspicture}
}\end{array}\leadsto\begin{array}{c}
\scalebox{.7} 
{
\begin{pspicture}(0.6,-0.5)(.6,0)
\psline[linewidth=0.02cm](0.6,0)(0.6,-0.61)
\end{pspicture}
}\end{array}\;=\quad\id{},\qquad\quad
\mu\circ_{2}u\quad\leadsto
\begin{array}{c}
\scalebox{.7} 
{
\begin{pspicture}(0,-0.63)(1.22,0.7)
\psdots[dotsize=0.18](0.6,-0.01)
\psline[linewidth=0.02cm](0.6,0)(0.6,-0.61)
\psline[linewidth=0.02cm](0.6,0)(0.0,0.59)
\psline[linewidth=0.02cm](0.6,0)(1.2,0.59)
\psdots[dotsize=0.18,fillstyle=solid,dotstyle=o](1.2,0.59)
\end{pspicture}
}\end{array}\leadsto\begin{array}{c}
\scalebox{.7} 
{
\begin{pspicture}(0.6,-0.5)(.6,0)
\psline[linewidth=0.02cm](0.6,0)(0.6,-0.61)
\end{pspicture}
}\end{array}\;=\quad\id{}.$$
The morphism $\phi^{\C V}$ is defined in terms of trees by the obvious inclusion. 
\end{exm}

We now prove that  this morphism is an epimorphism.

\begin{prop}\label{episet}
The morphism $\phi^{\operatorname{Set}}\colon \mathtt{Ass}^{\operatorname{Set}}\r \mathtt{uAss}^{\operatorname{Set}}$ in Example \ref{comparison} is an epimorphism in $\operad{\operatorname{Set}}$. 
\end{prop}

\begin{proof}

Consider two parallel morphisms in $\operad{\operatorname{Set}}$,
$$\xymatrix{\mathtt{uAss}^{\operatorname{Set}}
\ar@<.5ex>[r]^-{f}\ar@<-.5ex>[r]_-{g}
&\mathcal O,}$$
such that $f\phi^{\operatorname{Set}}=g\phi^{\operatorname{Set}}$. We must show that $f=g$. It is enough to prove that they coincide on the generators $\mu$ and $u$. They coincide on $\mu$ since it comes from $\phi^{\operatorname{Set}}$.
We now show that $f(u)=g(u)$ through a series of equations that hold by Remark \ref{examples},
\begin{align*}
(f(\mu)\circ_{1}f(u))\circ_{1}g(u)&=f(\mu\circ_{1}u)\circ_{1}g(u)\\
&=f(\id{\mathtt{uAss}^{\operatorname{Set}}})\circ_{1}g(u)\\
&=\id{\mathcal O}\circ_{1}g(u)\\
&=g(u),\\
(f(\mu)\circ_{1}f(u))\circ_{1}g(u)&=(f(\mu)\circ_{2}g(u))\circ_{1}f(u)\\
&=(g(\mu)\circ_{2}g(u))\circ_{1}f(u)\\
&=g(\mu\circ_{2}u)\circ_{1}f(u)\\
&=g(\id{\mathtt{uAss}^{\operatorname{Set}}})\circ_{1}f(u)\\
&=\id{\mathcal O}\circ_{1}f(u)\\
&=f(u).
\end{align*}
\end{proof}

We deduce that $\phi^{\C V}$ is an epimorphism in the following general situation.

\begin{prop}\label{epiall}
Let $\C{V}$ be a closed symmetric monoidal category with coproducts. 
The morphism $\phi^{\C V}\colon \mathtt{Ass}^{\C{V}}\r \mathtt{uAss}^{\C{V}}$ in Example \ref{comparison} is an epimorphism in $\operad{\C V}$. 
\end{prop}

\begin{proof}
The `underlying set' functor $\C{V}(\unit,-)\colon\C V\r\operatorname{Set}$ is part of a lax-lax symmetric monoidal adjoint pair
$$\xymatrix@C=40pt{
\operatorname{Set}\ar@<.5ex>[r]^-{-\otimes\unit}&\C V.\ar@<.5ex>[l]^-{\C{V}(\unit,-)}
}$$
The left adjoint sends a set $S$ to the coproduct of copies of the tensor unit indexed by this set $S\otimes\unit=\amalg_{s\in S}\unit$. This  adjoint pair induces an adjoint pair between categories of operads
$$\xymatrix@C=40pt{
\operad{\operatorname{Set}}\ar@<.5ex>[r]^-{-\otimes\unit}&\operad{\C V}.\ar@<.5ex>[l]^-{\C{V}(\unit,-)}
}$$
Notice that $\mathtt{Ass}^{\operatorname{Set}}\otimes \unit=\mathtt{Ass}^{\C{V}}$, $\mathtt{uAss}^{\operatorname{Set}}\otimes \unit=\mathtt{uAss}^{\C{V}}$, and $\phi^{\operatorname{Set}}\otimes\unit=\phi^{\C V}$. Hence, this proposition follows from the previous one, since left adjoints preserve epimorphisms.
\end{proof}

Proposition \ref{epiall} is  true even if $\C V$ does not have coproducts. The proof of Proposition \ref{episet} can be translated into diagrams in order to check this general case, Propositon \ref{epicureo}.

If $\C V$ has a suitable model structure, compatible with the monoidal structure, then the category of operads $\operad{\C V}$ carries an induced model structure. 

\begin{thm}[{\cite[Theorem 1.1]{htnso}}]\label{modelo}
Let $\C V$ be a cofibrantly generated closed symmetric monoidal model
category. Assume that $\C V$ satisfies the monoid axiom. Moreover, suppose that there are
sets of generating cofibrations $I$ and generating trivial cofibrations $J$ in $\C V$ with presentable
sources. Then the category $\operad{\C V}$ of  operads in $\C V$ is a cofibrantly
generated model category such that a morphism $f\colon\mathcal O \r\mathcal P$ in $\operad{\C V}$ is a weak
equivalence (resp. fibration) if and only if $f(n)\colon\mathcal O(n) \r\mathcal P(n)$ is a weak equivalence
(resp. fibration) in $\C V$ for all $n\geq0$.
\end{thm}

We refer the reader to \cite[\S4]{hmc} and \cite{ammmc} for the theory of symmetric monoidal model categories. 
All categories $\C V$ in this paper will satisfy the assumptions in this theorem, and this will be the only model structure considered on $\operad{\C V}$. 

\begin{rem}\label{generatriz}
Let us describe sets of generating (trivial) cofibrations in $\operad{\C V}$. The model structure in the previous theorem is transferred along the free operad adjunction in Remark \ref{free}. 
The category of sequences $\C V^{\mathbb N}$ is endowed with the product model structure.

Recall that given an object $X$ in $\C V$ and $n\geq 0$, we denote by $X[n]$ the sequence consisting of $X$ concentrated in arity $n$ and the initial object $\varnothing$ elsewhere. Given a morphism $f\colon X\r Y$ in $\C V$ we denote by $f[n]\colon X[n]\r Y[n]$ the morphism of sequences defined by $f$ in arity $n$ and the identity in $\varnothing$ elsewhere. For any set $S$ of morphisms in $\C V$, we write
$$S_{\mathbb N}=\bigcup_{n\geq 0}\{f[n]\,;\,f\in S\}.$$
The sets $I_{\mathbb N}$ and $J_{\mathbb N}$ are sets of generating cofibrations and generating trivial cofibrations in $\C V^{\mathbb N}$, respectively. Hence, $\mathcal F(I_{\mathbb N})$ and $\mathcal F(J_{\mathbb N})$ are sets of generating cofibrations and generating trivial cofibrations in $\operad{\C V}$.
\end{rem}
%
%
%

\begin{defn}\label{uiu}
Let $\C V$ be a symmetric monoidal model category satisfying the hypotheses of Theorem \ref{modelo}. 
An \emph{$A$-infinity operad} $\mathtt{A}^{\C V}_{\infty}$ in $\operad{\C V}$ is a cofibrant resolution of $\mathtt{Ass}^{\C V}$, $$\mathtt{A}^{\C V}_{\infty}\st{\sim}\To\mathtt{Ass}^{\C V}.$$

Similarly, a \emph{unital $A$-infinity operad} $\mathtt{uA}^{\C V}_{\infty}$ is a cofibrant resolution of $\mathtt{uAss}^{\C V}$  in $\operad{\C V}$, $$\mathtt{uA}^{\C V}_{\infty}\st{\sim}\To\mathtt{uAss}^{\C V}.$$

A \emph{$u$-infinity associative operad} $\mathtt{u}_{\infty}\mathtt{A}^{\C V}$ is the middle term of a factorization of $\phi^{\C V}$ as a cofibration $\bar\phi^{\C V}_{\infty}$ followed by a weak equivalence,
$$\mathtt{Ass}^{\C V}\st{\bar\phi^{\C V}_{\infty}}\into \mathtt{u}_{\infty}\mathtt{A}^{\C V}\st{\sim}\To \mathtt{uAss}^{\C V}.$$
\end{defn}

\begin{rem}\label{construcciones}
For each specific $\C V$, we may choose a cofibrant resolution of $\phi^{\C V}$,
$$\phi^{\C V}_{\infty}\colon \mathtt{A}^{\C V}_{\infty}\into \mathtt{uA}^{\C V}_{\infty}.$$
If $\C V$ satisfies the strong unit axiom \cite[Definition A.9]{htnso2}, e.g.~if the tensor unit is cofibrant, then 
we can define a $u$-infinity associative operad as the following push-out,
$$\xymatrix{\mathtt{A}^{\C V}_{\infty}\ar@{ >->}[r]^{\phi^{\C V}_{\infty}}\ar[d]_{\sim}\ar@{}[rd]|{\text{push}}& \mathtt{uA}^{\C V}_{\infty}\ar[d]^{\sim}\\
\mathtt{Ass}^{\C V}\ar@{ >->}[r]^{\bar\phi^{\C V}_{\infty}}&\mathtt{u}_{\infty}\mathtt{A}^{\C V}}$$
Here, the right vertical map is a weak equivalence by \cite[Corollary C.3 and Theorem C.7]{htnso2}. Hence, the map $\mathtt{u}_{\infty}\mathtt{A}^{\C V}\r \mathtt{uAss}^{\C V}$ induced by the universal property of the push-out is a weak equivalence by the 2-out-of-3 axiom.
\end{rem}

\begin{defn}\label{lol}
A \emph{$u$-infinity unital associative operad} $\mathtt{u}_{\infty}\mathtt{uA}^{\C V}$ is an operad fitting into a push-out square as follows,
$$\xymatrix{\mathtt{Ass}^{\C V}\ar@{ >->}[r]^{\bar\phi^{\C V}_{\infty}}\ar[d]_-{\phi^{\C V}}\ar@{}[rd]|{\text{push}}& \mathtt{u}_{\infty}\mathtt{A}^{\C V}\ar[d]^-{\psi^{\C V}}\\
\mathtt{uAss}^{\C V}\ar@{ >->}[r]_-{\varphi^{\C V}}&\mathtt{u}_{\infty}\mathtt{uA}^{\C V}}$$
\end{defn}

We will prove Theorem \ref{hepi} using the following lemma.

\begin{lem}\label{basico}
Let $\C V$ be a symmetric monoidal model category as in Theorem \ref{modelo}. Assume further that $\C V$ satisfies the strong unit axiom. The morphism $\phi^{\C V}\colon\mathtt{Ass}^{\C V}\r\mathtt{uAss}^{\C V}$ is a homotopy epimorphism in $\operad{\C V}$ if and only if $\varphi^{\C V}$ is a weak equivalence.
\end{lem}

\begin{proof}
The operads $\mathtt{Ass}^{\C V}$ and $\mathtt{uAss}^{\C V}$ belong to the full subcategory $\operadpc{\C V}\subset\operad{\C V}$ spanned by the operads $\mathcal O$ whose components $\mathcal O(n)$ are pseudo-cofibrant for all $n\geq 0$. Recall from \cite[Definition A.1]{htnso2} that an object $X$ in $\C V$ is \emph{pseudo-cofibrant} if the functor $X\otimes-$ preserves cofibrations. The tensor unit $\unit$ and the initial object $\varnothing$ are obviously  pseudo-cofibrant. The category $\operadpc{\C V}$ inherits from $\operad{\C V}$ the structure of a cofibration category, see \cite[Proposition C.8]{htnso2}. The operads $\mathtt{u}_{\infty}\mathtt{A}^{\C V}$ and $\mathtt{u}_{\infty}\mathtt{uA}^{\C V}$ are also in $\operadpc{\C V}$, see \cite[Corollary C.2]{htnso2}. Hence, this lemma follows from Remark \ref{casilp}. 
\end{proof}

The following lemma is useful to check that some symmetric monoidal categories carry a compatible model structure.

\begin{lem}\label{ppa}
Let $F\colon \C V\rightleftarrows\C W\colon G$ be a lax-lax symmetric monoidal adjunction between symmetric monoidal categories. Suppose that $\C V$ is a cofibrantly generated model category satisfying the push-out product axiom in \cite[Definition 3.1]{ammmc}. Assume further that $\C W$ possesses a transferred model structure along this adjunction, in the sense of \cite[Theorem 11.3.2]{hirschhorn}. Then $\C W$  also satisfies the push-out product axiom.
\end{lem}

\begin{proof}
Let $I$ and $J$ be sets of generating (trivial) cofibrations of $\C V$. Then $F(I)$ and $F(J)$ are sets of generating (trivial) cofibrations of $\C W$. Denote by $$f\odot g\colon U\otimes Y\bigcup_{U\otimes X}V\otimes X\To V\otimes Y$$ the push-out product of two morphisms $f\colon U\r V$ and $g\colon X\r Y$. In order to check the push-out product axiom for $\C W$, it is enough to prove that the sets $F(I)\odot F(I)$ and $F(I)\odot F(J)$ consist of cofibrations and trivial cofibrations in $\C W$, respectively, compare \cite[Corollary 4.2.5]{hmc}. The monoidal functor $F$ is strong, see \cite[Proposition 3.96]{mfsha}. It also preserves push-outs, since it is a left adjoint. Hence, $F$ preserves push-out products. In particular,
\begin{align*}
F(I)\odot F(I)&=F(I\odot I),&
F(I)\odot F(J)&=F(I\odot J).
\end{align*}
These sets consist of cofibrations and trivial cofibrations, respectively, since $\C V$ satisfies the push-out product axiom and $F$ preserves (trivial) cofibrations.
\end{proof}

\section{Main theorem for operads of groupoids}\label{grd}

Let $\operatorname{Gpd}$ be the model category of small groupoids. Morphisms are functors and weak equivalences are equivalences of categories.  A cofibration is a functor which is injective on objects. Fibrations are functors satisfying the isomorphism lifting property. Recall that $\varphi\colon G\r H$ has the \emph{isomorphism lifting property} if for any object $x$ in $G$ and any isomorphism $f\colon\varphi(x)\r y$ in $H$ there exists an isomorphism $f'\colon x\r x'$ in $G$ with $\varphi(f')=f$, in particular $\varphi(x')=y$. 

Trivial fibrations have a simple characterization.

\begin{lem}\label{tf}
A trivial fibration in $\operatorname{Gpd}$ is a fully-faithful functor surjective on objects.
\end{lem}

It is enough to notice that an equivalence of categories satisfies the isomorphism lifting property if and only if it is surjective on objects.

\begin{prop}\label{loes}
The category  $\operatorname{Gpd}$ with the cartesian product is a combinatorial closed symmetric monoidal model category satisfying the monoid axiom where all objects are cofibrant.
\end{prop}

\begin{proof}

The model structure on $\operatorname{Gpd}$ is transferred along the following adjoint pair
$$\xymatrix{\operatorname{Set}^{\Delta^{\op}}\ar@<.5ex>[r]^-{\Pi_{1}}&\operatorname{Grd}.\ar@<.5ex>[l]^-{\ner}}$$
Here, $\ner$ is the nerve functor and $\Pi_{1}$ is the fundamental groupoid functor. We regard $\operatorname{Set}^{\Delta^{\op}}$ as a symmetric monoidal model category with the usual model structure and the cartesian product monoidal structure, see \cite[Proposition 4.2.8]{hmc}. Since $\operatorname{Set}^{\Delta^{\op}}$ is cofibrantly generated, then so is $\operatorname{Gpd}$. Moreover, $\operatorname{Gpd}$ is locally presentable (this is an elementary fact from category theory). Hence, $\operatorname{Gpd}$ is a combinatorial model category.

The functor $\Pi_{1}$ is known to preserve products. Therefore, the push-out product axiom for $\operatorname{Grd}$ follows from Lemma \ref{ppa}. All objects are cofibrant by the very definition of cofibration. Hence, the monoid axiom follows from the push-out product axiom, see \cite[Remark 3.4]{ammmc}.
\end{proof}

The main result of this section is the following theorem.

\begin{thm}\label{hepiGpd}
The morphism $\phi^{\operatorname{Gpd}}\colon \mathtt{Ass}^{\operatorname{Gpd}}\r \mathtt{uAss}^{\operatorname{Gpd}}$ in Example \ref{comparison} is a homotopy epimorphism in $\operad{\operatorname{Gpd}}$. 
\end{thm}

This theorem follows from Lemma \ref{basico} above and Lemma \ref{uuu} below.

\begin{rem}\label{ijgrd}
The sets of generating (trivial) cofibrations of $\operatorname{Gpd}$ obtained by taking fundamental groupoids on the usual sets of generating (trivial) cofibrations of $\operatorname{Set}^{\Delta^{\op}}$ are too big. We can alternatively take
\begin{align*}
I&=\{\varnothing\into e,\;i\colon \{e,e'\}\into E,\;p\colon \mathbb{Z}\into e\},& 
J&=\{j\colon e\st{\sim}\into E\}.
\end{align*}
Here $e$ is the final groupoid, which consists of only one object and one morphism (the identity), $\mathbb Z$ is the groupoid with one object with automorphism group $\mathbb Z$, and $E$ is the groupoid with two isomorphic objects, $e\cong e'$, with trivial automorphism groups. The functor $i$ is the inclusion of the discrete subgroupoid formed by the two objects, and $j$ is the inclusion of an object. Recall that a groupoid $G$ is \emph{discrete} if the only morphisms in $G$ are the indentities.

Indeed, a functor satisfying the right lifting property with respect to $\varnothing\into e$, $i$, or $p$, is a functor surjective on objects, full, or faithful, respectively.
\end{rem}

\begin{rem}\label{contractible}
Limits are easier than colimits in the category of groupoids, at least easier than non-filtered colimits. However, colimits behave well on objects, in the sense that the set of objects of the colimit of a diagram of groupoids is the colimit of the diagram of object sets.
This follows from the fact that the `set of objects' functor from groupoids to sets has a right adjoint,
$$\xymatrix@C=40pt{\operatorname{Grd}\ar@<.5ex>[r]^-{\operatorname{Ob}}&\operatorname{Set}.\ar@<.5ex>[l]^-{\text{contractible}}}$$
The right adjoint, called `contractible groupoid' functor, sends the empty set to the empty groupoid, and any non-empty set $S$ to the contractible groupoid with object set $S$. Recall that a groupoid $G$ is \emph{contractible} if it is equivalent to $e$, i.e.~if it has a non-empty set of objects and there exists a unique isomorphism between any to objects of $G$. Hence, morphisms into contractible groupoids are usually denoted by simply indicating the source, the target, and the map between object sets.

The `contractible groupoid' and the `set of objects' functors preserve products, hence they induce an adjoint pair on operads,
$$\xymatrix@C=40pt{\operad{\operatorname{Grd}}\ar@<.5ex>[r]^-{\operatorname{Ob}}&\operad{\operatorname{Set}}.\ar@<.5ex>[l]^-{\text{contractible}}}$$
In particular, the `set of objects' functor  also preserves colimits at the level of operads. This fact is used in the proof of the following lemma.
\end{rem}

In this section we consider free operads of sets and free operads of groupoids. For the sake of simplicity, we omit the subscript from the free operad functor for sets $\mathcal F=\mathcal F_{\operatorname{Set}}$, but not for groupoids $\mathcal F_{\operatorname{Grd}}$, in order to avoid confusion. These free operad functors and the object set functor commute,
$\operatorname{Ob}\mathcal F_{\operatorname{Grd}}=\mathcal F\operatorname{Ob}$, since the latter preserves cartesian products, compare Remark \ref{free}.

\begin{lem}\label{cofibrationgrd}
A morphism is a cofibration in $\operad{\operatorname{Grd}}$  if and only if it is a retract of a morphism $f\colon \mathcal O\r\mathcal P$ such that $\operatorname{Ob}(f)\colon \operatorname{Ob}(\mathcal O)\r\operatorname{Ob}(\mathcal P)=\operatorname{Ob}(\mathcal O)\amalg \mathcal F(V)$ is an inclusion of a factor of a binary coproduct such that the other factor  is a free operad in $\operad{\operatorname{Set}}$.
\end{lem}

\begin{proof}
Any relative $\mathcal F_{\operatorname{Grd}}(I_{\mathbb N})$-cell complex is as $f$ in the statement. Indeed, on the one hand, the functors $i,p\in I$ are the identity on objects, hence a push-out along $\mathcal F_{\operatorname{Grd}}(i[n])$ or $\mathcal F_{\operatorname{Grd}}(p[n])$ is the identity on objects. On the other hand, a push-out along $\mathcal F_{\operatorname{Grd}}(\varnothing\into e[n])$ adds freely a new object in arity $n$.  Hence, the `only if' part follows.

The converse is also true, i.e.~any morphism as $f$ is a relative $\mathcal F_{\operatorname{Grd}}(I_{\mathbb N})$-cell complex, but this is complicated to show directly. In order to prove the `if' part, it is easier to check that $f$ in the statement satisfies the left lifting property with respect to trivial fibrations. We therefore consider a commutative diagram of solid arrows in $\operad{\operatorname{Grd}}$ as follows,
$$\xymatrix{\mathcal O\ar[r]^{g}\ar[d]_{f}&\mathcal Q\ar@{->>}[d]_{\sim}^{q}\\
\mathcal P\ar[r]_{h}\ar@{-->}[ru]^{l}&\mathcal R}$$
Here, $q$ is a trivial fibration. In order to obtain a lifting $l$, we first consider the diagram of objects
$$\xymatrix{\operatorname{Ob}(\mathcal O)\ar[r]^{\operatorname{Ob}(g)}\ar[d]_{\operatorname{Ob}(f)}&\operatorname{Ob}(\mathcal Q)\ar@{->>}[d]^{\operatorname{Ob}(q)}\\
\operatorname{Ob}(\mathcal O)\amalg \mathcal F(V)\ar[r]_-{\operatorname{Ob}(h)}\ar@{-->}[ru]^{l'}&\operatorname{Ob}(\mathcal R)}$$
Here, $\operatorname{Ob}(q)$ is levelwise surjective by Lemma \ref{tf}. Hence, it is easy to obtain a lifting $l'$ in $\operad{\operatorname{Set}}$. Define $l'$ as $\operatorname{Ob}(g)$ on the first factor. On the second factor, we choose preimages of the objects $h(V(n))$ along $q(n)$, $n\geq 0$, and extend to a morphism from the free operad $\mathcal F(V)$. Finally, since $q(n)$ is fully faithful, there is a unique functor  $l(n)\colon\mathcal{P}(n)\r\mathcal{Q}(n)$ given by $l'(n)$ on objects and such that $q(n)l(n)=h(n)$, $n\geq 0$. One can easily check that the sequence of functors $\{l(n)\}_{n\geq 0}$ is an operad morphism $l$ which also satisfies $lf=g$.
\end{proof}

\begin{cor}\label{cofibrantgrd}
An operad $\mathcal O$ in $\operad{\operatorname{Grd}}$ is cofibrant if and only if the operad  of object sets $\operatorname{Ob}(\mathcal O)$  is a retract of a free operad in $\operad{\operatorname{Set}}$.
\end{cor}

We can now easily define a $u$-infinity associative operad of groupoids.

%

\begin{defn}
The \emph{$u$-infinity associative operad $\mathtt{u}_{\infty}\mathtt{A}^{\operatorname{Grd}}$} in $\operad{\operatorname{Grd}}$ is the levelwise contractible operad with operad of objects $\operatorname{Ob}(\mathtt{u}_{\infty}\mathtt{A}^{\operatorname{Grd}})=\mathtt{Ass}^{\operatorname{Set}}\amalg \mathcal F(\{u\}[0])$.
\end{defn}

\begin{lem}\label{esuiu}
The operad $\mathtt{u}_{\infty}\mathtt{A}^{\operatorname{Grd}}$  in the previous definition is indeed a $u$-infinity associative operad in the sense of Definition \ref{uiu}. The morphism ${\bar\phi^{\operatorname{Grd}}_{\infty}}\colon \mathtt{Ass}^{\operatorname{Grd}}\into \mathtt{u}_{\infty}\mathtt{A}^{\operatorname{Grd}}$ is given on objects by the inclusion of the first factor of the coproduct.
\end{lem}

\begin{proof}
The morphism ${\bar\phi^{\operatorname{Grd}}_{\infty}}$ is a cofibration by Lemma \ref{cofibrationgrd}. Moreover, the unique morphism $\mathtt{u}_{\infty}\mathtt{A}^{\operatorname{Grd}}\r \mathtt{uAss}^{\operatorname{Grd}}$ is a weak equivalence since $\mathtt{u}_{\infty}\mathtt{A}^{\operatorname{Grd}}$ is levelwise contractible. The composition of these morphisms is $\phi^{\operatorname{Grd}}$ since $\mathtt{uAss}^{\operatorname{Grd}}$ is final in $\operad{\operatorname{Grd}}$.
\end{proof}

\begin{rem}\label{precisar}
We now describe the operad of objects of $\mathtt{u}_{\infty}\mathtt{A}^{\operatorname{Grd}}$ following Remarks \ref{binary} and \ref{sets}. The set $\operatorname{Ob}(\mathtt{u}_{\infty}\mathtt{A}^{\operatorname{Grd}})(n)$, for $n>1$, can be identified with the set of corollas with $n$ leaves, at least two branches, and possibly corks, e.g.
$$
\mu^{4}(\id{},u,u,\id{},\id{})\quad=\begin{array}{c}
\scalebox{.7} 
{
\begin{pspicture}(0,-1.05)(1.62,1.05)
\psline[linewidth=0.02cm](0.8,-0.03)(0.8,-1.03)
\psline[linewidth=0.02cm](0.8,-0.03)(0.0,0.97)
\psline[linewidth=0.02cm](0.4,0.97)(0.8,-0.03)
\psline[linewidth=0.02cm](0.8,0.97)(0.8,-0.03)
\psline[linewidth=0.02cm](0.8,-0.03)(1.2,0.97)
\psline[linewidth=0.02cm](0.8,-0.03)(1.6,0.97)
\psdots[dotsize=0.18](0.8,-0.03)
\psdots[dotsize=0.18](0.4,0.97)
\psdots[dotsize=0.18](0.8,0.97)
\end{pspicture} 
}
\end{array}\in\quad \operatorname{Ob}(\mathtt{u}_{\infty}\mathtt{A}^{\operatorname{Grd}}(3)),$$
A \emph{branch} of a tree is an edge adjacent to a leaf or a cork. For $n=0,1$, in addition to the corollas with $n$ leaves, at least two branches, and possibly corks,  we have
$$
u=\begin{array}{c}
\scalebox{.7} 
{
\begin{pspicture}(0.6,-0.61)(.6,0)
\psline[linewidth=0.02cm](0.6,0)(0.6,-0.61)
\psdots[dotsize=0.18](0.6,-0.01)
\end{pspicture}
}\end{array}\in \operatorname{Ob}(\mathtt{u}_{\infty}\mathtt{A}^{\operatorname{Grd}}(0)),\qquad\qquad
\id{}=\begin{array}{c}
\scalebox{.7} 
{
\begin{pspicture}(0.6,-0.61)(.6,0)
\psline[linewidth=0.02cm](0.6,0)(0.6,-0.61)
\end{pspicture}
}\end{array}\in \operatorname{Ob}(\mathtt{u}_{\infty}\mathtt{A}^{\operatorname{Grd}}(1)).
$$
By requiring at least two branches we are explicitly excluding the following two corollas,
$$
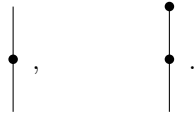
\begin{array}{c}
\scalebox{.7} 
{
\begin{pspicture}(0.8,-1.05)(.8,1.05)
\psline[linewidth=0.02cm](0.8,-0.03)(0.8,-1.03)
\psline[linewidth=0.02cm](0.8,0.97)(0.8,-0.03)
\psdots[dotsize=0.18](0.8,-0.03)
\end{pspicture} 
}
\end{array},\qquad\qquad
\begin{array}{c}
\scalebox{.7} 
{
\begin{pspicture}(0.8,-1.05)(.8,1.05)
\psline[linewidth=0.02cm](0.8,-0.03)(0.8,-1.03)
\psline[linewidth=0.02cm](0.8,0.97)(0.8,-0.03)
\psdots[dotsize=0.18](0.8,-0.03)
\psdots[dotsize=0.18](0.8,0.97)
\end{pspicture} 
}
\end{array}.
$$

The composition laws  (between trees different from $\id{}=|$) are given by grafting and then contracting the newly created inner edge,
$$
\begin{array}{c}
\scalebox{.7} 
{
\begin{pspicture}(0,-1.05)(1.62,1.05)
\psline[linewidth=0.02cm](0.8,-0.03)(0.8,-1.03)
\psline[linewidth=0.02cm](0.8,-0.03)(0.0,0.97)
\psline[linewidth=0.02cm](0.8,0.97)(0.8,-0.03)
\psline[linewidth=0.02cm](0.8,-0.03)(1.6,0.97)
\psdots[dotsize=0.18](0.8,-0.03)
\psdots[dotsize=0.18](0.8,0.97)
\end{pspicture}}
\end{array}
\circ_{2} 
\begin{array}{c}
\scalebox{.7} 
{
\begin{pspicture}(0,-1.05)(1.62,1.05)
\psline[linewidth=0.02cm](0.8,-0.03)(0.8,-1.03)
\psline[linewidth=0.02cm](0.8,-0.03)(0.0,0.97)
\psline[linewidth=0.02cm](0.4,0.97)(0.8,-0.03)
\psline[linewidth=0.02cm](0.8,0.97)(0.8,-0.03)
\psline[linewidth=0.02cm](0.8,-0.03)(1.2,0.97)
\psline[linewidth=0.02cm](0.8,-0.03)(1.6,0.97)
\psdots[dotsize=0.18](0.8,-0.03)
\psdots[dotsize=0.18](0.4,0.97)
\psdots[dotsize=0.18](0.8,0.97)
\end{pspicture} 
}
\end{array}
\leadsto
\begin{array}{c}
\scalebox{.7} 
{
\begin{pspicture}(0,-1.555)(2.41,1.5591111)
\psline[linewidth=0.02cm](0.8,-0.545)(0.8,-1.545)
\psline[linewidth=0.02cm](0.8,-0.545)(0.0,0.455)
\psline[linewidth=0.02cm](0.8,0.455)(0.8,-0.545)
\psline[linewidth=0.02cm](0.8,-0.545)(1.6,0.455)
\psdots[dotsize=0.18](0.8,-0.545)
\psdots[dotsize=0.18](0.8,0.455)
\psline[linewidth=0.02cm](1.6,0.455)(0.8,1.455)
\psline[linewidth=0.02cm](1.2,1.455)(1.6,0.455)
\psline[linewidth=0.02cm](1.6,1.455)(1.6,0.455)
\psline[linewidth=0.02cm](1.6,0.455)(2.0,1.455)
\psline[linewidth=0.02cm](1.6,0.455)(2.4,1.455)
\psdots[dotsize=0.18](1.6,0.455)
\psdots[dotsize=0.18](1.2,1.455)
\psdots[dotsize=0.18](1.6,1.455)
\end{pspicture} 
}
\end{array}
\leadsto
\begin{array}{c}
\scalebox{.7} 
{
\begin{pspicture}(0,-1.055)(2.4,1.0591111)
\psline[linewidth=0.02cm](1.19,-0.045)(1.19,-1.045)
\psline[linewidth=0.02cm](1.19,-0.045)(0.0,0.9391111)
\psline[linewidth=0.02cm](0.4,0.9391111)(1.19,-0.045)
\psdots[dotsize=0.18](1.19,-0.045)
\psdots[dotsize=0.18](0.39,0.955)
\psline[linewidth=0.02cm](1.19,-0.045)(0.8,0.9391111)
\psline[linewidth=0.02cm](1.2,0.9391111)(1.2,-0.06088889)
\psline[linewidth=0.02cm](1.59,0.955)(1.2,-0.06088889)
\psline[linewidth=0.02cm](1.2,-0.06088889)(1.99,0.955)
\psline[linewidth=0.02cm](1.2,-0.06088889)(2.39,0.955)
\psdots[dotsize=0.18](1.19,0.955)
\psdots[dotsize=0.18](1.59,0.955)
\end{pspicture} 
}
\end{array},$$
$$\mu^{2}(\id{},u,\id{})\circ_{2}\mu^{4}(\id{},u,u,\id{},\id{})=
\mu^{6}(\id{},u,\id{},u,u,\id{},\id{}),$$
except when the grafted tree is $u$. In that case, the new inner edge is not contracted,
$$
\mu^{2}(\id{},u,\id{})\circ_{2}u\quad =
\begin{array}{c}
\scalebox{.7} 
{
\begin{pspicture}(0,-1.05)(1.62,1.05)
\psline[linewidth=0.02cm](0.8,-0.03)(0.8,-1.03)
\psline[linewidth=0.02cm](0.8,-0.03)(0.0,0.97)
\psline[linewidth=0.02cm](0.8,0.97)(0.8,-0.03)
\psline[linewidth=0.02cm](0.8,-0.03)(1.6,0.97)
\psdots[dotsize=0.18](0.8,-0.03)
\psdots[dotsize=0.18](0.8,0.97)
\end{pspicture}}
\end{array}
\circ_{2} \begin{array}{c}
\scalebox{.7} 
{
\begin{pspicture}(0,-0.63)(1.22,0.63)
\psline[linewidth=0.02cm](0.6,0)(0.6,-0.61)
\psdots[dotsize=0.18](0.6,-0.01)
\end{pspicture}
}\\[10pt]\end{array}=\begin{array}{c}
\scalebox{.7} 
{
\begin{pspicture}(0,-1.05)(1.62,1.05)
\psline[linewidth=0.02cm](0.8,-0.03)(0.8,-1.03)
\psline[linewidth=0.02cm](0.8,-0.03)(0.0,0.97)
\psline[linewidth=0.02cm](0.8,0.97)(0.8,-0.03)
\psline[linewidth=0.02cm](0.8,-0.03)(1.6,0.97)
\psdots[dotsize=0.18](0.8,-0.03)
\psdots[dotsize=0.18](0.8,0.97)
\psdots[dotsize=0.18](1.6,0.97)
\end{pspicture}}
\end{array}=\quad\mu^{2}(\id{},u,u).$$
Compare Example \ref{sets}. 
\end{rem}

\begin{lem}\label{generam}
The operad $\mathtt{u}_{\infty}\mathtt{A}^{\operatorname{Grd}}$ is generated by the objects $\mu$ and $u$ and by the isomorphisms $\mu(u,\id{})\cong \id{}$ and $\mu(\id{},u)\cong \id{}$.
\end{lem}

\begin{proof}
The previous remark shows that any object in $\mathtt{u}_{\infty}\mathtt{A}^{\operatorname{Grd}}$ can be obtained from $\mu$ and $u$. We must show that the unique existing isomorphism between any two objects can be obtained from $\mu(u,\id{})\cong \id{}$ and $\mu(\id{},u)\cong \id{}$. It is enough to prove that we can get all morphisms with target $\mu^{n-1}$, the corolla with $n$ leaves and no corks, $n\geq2$, all morphisms with target $\id{}=|$, and all morphisms with target $u$. 

Starting with an  object of positive arity represented by a corolla as in the previous remark, the isomorphisms in the statement, that we can respectively denote
$$
\begin{array}{c}
\scalebox{.7} 
{
\begin{pspicture}(0,-0.63)(1.22,0.7)
\psdots[dotsize=0.18](0.6,-0.01)
\psline[linewidth=0.02cm](0.6,0)(0.6,-0.61)
\psline[linewidth=0.02cm](0.6,0)(0.0,0.59)
\psline[linewidth=0.02cm](0.6,0)(1.2,0.59)
\psdots[dotsize=0.18](0,0.59)
\end{pspicture}
}\end{array}\st{\lambda}\To\!\!\!\!\begin{array}{c}
\scalebox{.7} 
{
\begin{pspicture}(0.3,-0.63)(.8,0)
\psline[linewidth=0.02cm](0.6,0)(0.6,-0.61)
\end{pspicture}
}\end{array},
\qquad
\begin{array}{c}
\scalebox{.7} 
{
\begin{pspicture}(0,-0.63)(1.22,0.7)
\psdots[dotsize=0.18](0.6,-0.01)
\psline[linewidth=0.02cm](0.6,0)(0.6,-0.61)
\psline[linewidth=0.02cm](0.6,0)(0.0,0.59)
\psline[linewidth=0.02cm](0.6,0)(1.2,0.59)
\psdots[dotsize=0.18](1.2,0.59)
\end{pspicture}
}\end{array}\st{\rho}\To\!\!\!\!\begin{array}{c}
\scalebox{.7} 
{
\begin{pspicture}(0.3,-0.63)(.8,0)
\psline[linewidth=0.02cm](0.6,0)(0.6,-0.61)
\end{pspicture}
}\end{array},$$
allow to delete one cork at a time, ending up with a corolla with no corks or with $\id{}=|$, depending on the arity, e.g.~
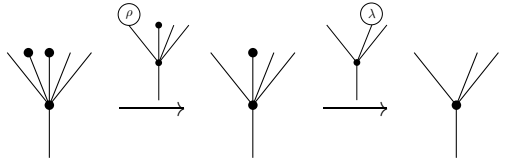
$$\xymatrix{\begin{array}{c}
\scalebox{.7} 
{
\begin{pspicture}(0,-1.05)(1.62,1.05)
\psline[linewidth=0.02cm](0.8,-0.03)(0.8,-1.03)
\psline[linewidth=0.02cm](0.8,-0.03)(0.0,0.97)
\psline[linewidth=0.02cm](0.4,0.97)(0.8,-0.03)
\psline[linewidth=0.02cm](0.8,0.97)(0.8,-0.03)
\psline[linewidth=0.02cm](0.8,-0.03)(1.2,0.97)
\psline[linewidth=0.02cm](0.8,-0.03)(1.6,0.97)
\psdots[dotsize=0.18](0.8,-0.03)
\psdots[dotsize=0.18](0.4,0.97)
\psdots[dotsize=0.18](0.8,0.97)
\end{pspicture} 
}
\end{array}
\ar[r]^{
\scalebox{.5} 
{
\begin{pspicture}(0,-1.4020555)(2.0,1.3920555)
\psline[linewidth=0.02cm](1.19,-0.39205554)(1.19,-1.3920555)
\psline[linewidth=0.02cm](1.19,-0.39205554)(0.39,0.6079444)
\psline[linewidth=0.02cm](1.19,0.6079444)(1.19,-0.39205554)
\psline[linewidth=0.02cm](1.19,-0.39205554)(1.59,0.6079444)
\psline[linewidth=0.02cm](1.19,-0.39205554)(1.99,0.6079444)
\psdots[dotsize=0.18](1.19,-0.39205554)
\psdots[dotsize=0.18](1.19,0.6079444)
\psellipse[linewidth=0.02,dimen=outer](0.4,0.89)(0.3,0.3)
\usefont{T1}{ptm}{m}{n}
\rput(0.4,.9){$\rho$}
\end{pspicture} 
}
}
&
\begin{array}{c}
\scalebox{.7} 
{
\begin{pspicture}(0,-1.05)(1.62,1.05)
\psline[linewidth=0.02cm](0.8,-0.03)(0.8,-1.03)
\psline[linewidth=0.02cm](0.8,-0.03)(0.0,0.97)
\psline[linewidth=0.02cm](0.8,0.97)(0.8,-0.03)
\psline[linewidth=0.02cm](0.8,-0.03)(1.2,0.97)
\psline[linewidth=0.02cm](0.8,-0.03)(1.6,0.97)
\psdots[dotsize=0.18](0.8,-0.03)
\psdots[dotsize=0.18](0.8,0.97)
\end{pspicture} 
}
\end{array}
\ar[r]^{
\scalebox{.5} 
{
\begin{pspicture}(0,-1.415)(1.6,1.405)
\psline[linewidth=0.02cm](0.8,-0.405)(0.8,-1.405)
\psline[linewidth=0.02cm](0.8,-0.405)(0.0,0.595)
\psline[linewidth=0.02cm](0.8,-0.405)(1.2,0.595)
\psline[linewidth=0.02cm](0.8,-0.405)(1.6,0.595)
\psdots[dotsize=0.18](0.8,-0.405)
\psellipse[linewidth=0.02,dimen=outer](1.19,.89)(0.3,0.3)
\usefont{T1}{ptm}{m}{n}
\rput(1.2,.9){$\displaystyle\lambda$}
\end{pspicture} 
}
}
&
\begin{array}{c}
\scalebox{.7} 
{
\begin{pspicture}(0,-1.05)(1.62,1.05)
\psline[linewidth=0.02cm](0.8,-0.03)(0.8,-1.03)
\psline[linewidth=0.02cm](0.8,-0.03)(0.0,0.97)
\psline[linewidth=0.02cm](0.8,-0.03)(1.2,0.97)
\psline[linewidth=0.02cm](0.8,-0.03)(1.6,0.97)
\psdots[dotsize=0.18](0.8,-0.03)
\end{pspicture} 
}
\end{array}.}$$

In arity $0$, we can use the isomorphism
$$
\lambda\circ_{1}u=\rho\circ_{1}u\colon \begin{array}{c}
\scalebox{.7} 
{
\begin{pspicture}(0,-0.63)(1.22,0.7)
\psdots[dotsize=0.18](0.6,-0.01)
\psline[linewidth=0.02cm](0.6,0)(0.6,-0.61)
\psline[linewidth=0.02cm](0.6,0)(0.0,0.59)
\psline[linewidth=0.02cm](0.6,0)(1.2,0.59)
\psdots[dotsize=0.18](0,0.59)
\psdots[dotsize=0.18](1.2,0.59)
\end{pspicture}
}\end{array}\To\!\!\!\!\begin{array}{c}
\scalebox{.7} 
{
\begin{pspicture}(0.3,-0.63)(.8,0)
\psline[linewidth=0.02cm](0.6,0)(0.6,-0.61)
\psdots[dotsize=0.18](0.6,-0.01)
\end{pspicture}
}\end{array},
$$
to reduce the number of corks, ending up with $u$. Hence, we are done.
\end{proof}

We now define an operad that we will later show to be a $u$-infinity unital associative operad  of groupoids, see Lemma \ref{uuu} below.

\begin{defn}
The operad  of groupoids $\mathcal{U}$ is defined as the  levelwise contractible operad with objects
$$
\operatorname{Ob}(\mathcal{U})=
\mathtt{uAss}^{\operatorname{Set}}\amalg \mathcal{F}(\{u'\}[0]).$$
\end{defn}

\begin{rem}\label{precosar}
Following Remarks \ref{binary} and \ref{precisar}, we here describe the operad of sets $\operatorname{Ob}(\mathcal{U})$. We can consider the morphism in $\operad{\operatorname{Grd}}$
$$\psi\colon \mathtt{u}_{\infty}\mathtt{A}^{\operatorname{Grd}}\To \mathcal{U}$$
given on objects by 
$$\operatorname{Ob}(\psi)=\phi^{\operatorname{Set}}\amalg(\text{iso. }u\mapsto u')\colon\mathtt{Ass}^{\operatorname{Set}}\amalg \mathcal{F}(\{u\}[0])\To \mathtt{uAss}^{\operatorname{Set}}\amalg \mathcal{F}(\{u'\}[0]).$$
The morphism $\psi$ is an isomorphism in positive arities, that we use as an identification. In arity $0$, $\psi$ in an inclusion of objects (and hence morphisms). We also identify $\mathtt{u}_{\infty}\mathtt{A}^{\operatorname{Grd}}(0)$ with its image in $\mathcal{U}(0)$ through $\psi(0)$. The extra object of  $\mathcal{U}(0)$ is $u$, which is represented by a trivial corolla with a white cork, as in Example \ref{sets}. Therefore,  black cork means $u'$ and white cork means $u$, 
$$
\mu^{4}(\id{},u',u',\id{},\id{})\quad=\begin{array}{c}
\scalebox{.7} 
{
\begin{pspicture}(0,-1.05)(1.62,1.05)
\psline[linewidth=0.02cm](0.8,-0.03)(0.8,-1.03)
\psline[linewidth=0.02cm](0.8,-0.03)(0.0,0.97)
\psline[linewidth=0.02cm](0.4,0.97)(0.8,-0.03)
\psline[linewidth=0.02cm](0.8,0.97)(0.8,-0.03)
\psline[linewidth=0.02cm](0.8,-0.03)(1.2,0.97)
\psline[linewidth=0.02cm](0.8,-0.03)(1.6,0.97)
\psdots[dotsize=0.18](0.8,-0.03)
\psdots[dotsize=0.18](0.4,0.97)
\psdots[dotsize=0.18](0.8,0.97)
\end{pspicture} 
}
\end{array},\qquad
u'\quad=\begin{array}{c}
\scalebox{.7} 
{
\begin{pspicture}(0.6,-0.5)(.6,0)
\psline[linewidth=0.02cm](0.6,0)(0.6,-0.61)
\psdots[dotsize=0.18](0.6,-0.01)
\end{pspicture}
}\end{array},\qquad
u\quad=\begin{array}{c}
\scalebox{.7} 
{
\begin{pspicture}(0.6,-0.5)(.6,0)
\psline[linewidth=0.02cm](0.6,0)(0.6,-0.61)
\psdots[dotsize=0.18,fillstyle=solid,dotstyle=o](0.6,-0.01)
\end{pspicture}
}\end{array}.
$$
The composition laws are defined in terms of trees as in Remark \ref{precisar} when $u$ is not involved. If $u$ appears, it is almost always given by grafting and then contracting the newly created inner edge, 
$$
\begin{array}{c}
\scalebox{.7} 
{
\begin{pspicture}(0,-1.05)(1.62,1.05)
\psline[linewidth=0.02cm](0.8,-0.03)(0.8,-1.03)
\psline[linewidth=0.02cm](0.8,-0.03)(0.0,0.97)
\psline[linewidth=0.02cm](0.4,0.97)(0.8,-0.03)
\psline[linewidth=0.02cm](0.8,0.97)(0.8,-0.03)
\psline[linewidth=0.02cm](0.8,-0.03)(1.2,0.97)
\psline[linewidth=0.02cm](0.8,-0.03)(1.6,0.97)
\psdots[dotsize=0.18](0.8,-0.03)
\psdots[dotsize=0.18](0.4,0.97)
\psdots[dotsize=0.18](0.8,0.97)
\end{pspicture} 
}
\end{array}\circ_{2}\begin{array}{c}
\scalebox{.7} 
{
\begin{pspicture}(0,-0.63)(1.22,0.63)
\psline[linewidth=0.02cm](0.6,0)(0.6,-0.61)
\psdots[dotsize=0.18,fillstyle=solid,dotstyle=o](0.6,-0.01)
\end{pspicture}
}\\[10pt]\end{array}\leadsto
\begin{array}{c}
\scalebox{.7} 
{
\begin{pspicture}(0,-1.05)(1.62,1.05)
\psline[linewidth=0.02cm](0.8,-0.03)(0.8,-1.03)
\psline[linewidth=0.02cm](0.8,-0.03)(0.0,0.97)
\psline[linewidth=0.02cm](0.4,0.97)(0.8,-0.03)
\psline[linewidth=0.02cm](0.8,0.97)(0.8,-0.03)
\psline[linewidth=0.02cm](0.8,-0.03)(1.2,0.97)
\psline[linewidth=0.02cm](0.8,-0.03)(1.6,0.97)
\psdots[dotsize=0.18](0.8,-0.03)
\psdots[dotsize=0.18](0.4,0.97)
\psdots[dotsize=0.18](0.8,0.97)
\psdots[dotsize=0.18,fillstyle=solid,dotstyle=o](1.2,0.97)
\end{pspicture} 
}
\end{array}\leadsto
\begin{array}{c}
\scalebox{.7} 
{
\begin{pspicture}(0,-1.05)(1.62,1.05)
\psline[linewidth=0.02cm](0.8,-0.03)(0.8,-1.03)
\psline[linewidth=0.02cm](0.8,-0.03)(0.0,0.97)
\psline[linewidth=0.02cm](0.4,0.97)(0.8,-0.03)
\psline[linewidth=0.02cm](0.8,0.97)(0.8,-0.03)
\psline[linewidth=0.02cm](0.8,-0.03)(1.6,0.97)
\psdots[dotsize=0.18](0.8,-0.03)
\psdots[dotsize=0.18](0.4,0.97)
\psdots[dotsize=0.18](0.8,0.97)
\end{pspicture} 
}
\end{array},
$$$$\mu^{4}(\id{},u',u',\id{},\id{})\circ_{2}u=\mu^{4}(\id{},u',u',u,\id{})
=\mu^{3}(\id{},u',u',\id{})
.
$$
There are four exceptions, the two exceptions in Example \ref{sets} and 
$$(\mu(\id{},u'))\circ_{1}u\;\leadsto\begin{array}{c}
\scalebox{.7} 
{
\begin{pspicture}(0,-0.63)(1.22,0.7)
\psdots[dotsize=0.18](0.6,-0.01)
\psline[linewidth=0.02cm](0.6,0)(0.6,-0.61)
\psline[linewidth=0.02cm](0.6,0)(0.0,0.59)
\psline[linewidth=0.02cm](0.6,0)(1.2,0.59)
\psdots[dotsize=0.18,fillstyle=solid,dotstyle=o](0,0.59)
\psdots[dotsize=0.18](1.2,0.59)
\end{pspicture}
}\end{array}\leadsto\begin{array}{c}
\scalebox{.7} 
{
\begin{pspicture}(0.6,-0.5)(.6,0)
\psline[linewidth=0.02cm](0.6,0)(0.6,-0.61)
\psdots[dotsize=0.18](0.6,-0.01)
\end{pspicture}
}\end{array}=\;u',\quad
(\mu(u',\id{}))\circ_{1}u\;\leadsto
\begin{array}{c}
\scalebox{.7} 
{
\begin{pspicture}(0,-0.63)(1.22,0.7)
\psdots[dotsize=0.18](0.6,-0.01)
\psline[linewidth=0.02cm](0.6,0)(0.6,-0.61)
\psline[linewidth=0.02cm](0.6,0)(0.0,0.59)
\psline[linewidth=0.02cm](0.6,0)(1.2,0.59)
\psdots[dotsize=0.18,fillstyle=solid,dotstyle=o](1.2,0.59)
\psdots[dotsize=0.18](0,0.59)
\end{pspicture}
}\end{array}\leadsto\begin{array}{c}
\scalebox{.7} 
{
\begin{pspicture}(0.6,-0.5)(.6,0)
\psline[linewidth=0.02cm](0.6,0)(0.6,-0.61)
\psdots[dotsize=0.18](0.6,-0.01)
\end{pspicture}
}\end{array}=\;u'.$$

\end{rem}

\begin{lem}\label{ooo}
The operad $\mathcal{U}$ fits into the following push-out diagram,
$$\xymatrix@C=40pt{\mathcal{F}_{\operatorname{Grd}}(e[0])\ar@{ >->}[r]^{\mathcal{F}_{\operatorname{Grd}}(j[0])}_{\sim}\ar[d]_{\zeta}\ar@{}[rd]|{\text{push}}& \mathcal{F}_{\operatorname{Grd}}(E[0])\ar[d]^{\zeta'}\\
\mathtt{uAss}^{\operatorname{Grd}}\ar@{ >->}[r]^-{\sim}_-{\varphi}&\mathcal{U}}$$
where $\varphi$ is given on objects by the inclusion of the first factor of the coproduct, $\zeta$ is defined by $\zeta(e)=u$, and $\zeta'$ is defined by $\zeta'(e)=u$ and $\zeta'(e')=u'$. 
\end{lem}

\begin{proof}
Denote by
$$\xymatrix@C=40pt{\mathcal{F}_{\operatorname{Grd}}(e[0])\ar@{ >->}[r]^{\mathcal{F}_{\operatorname{Grd}}(j[0])}_{\sim}\ar[d]_{\zeta}\ar@{}[rd]|{\text{push}}& \mathcal{F}_{\operatorname{Grd}}(E[0])\ar[d]^{\bar\zeta}\\
\mathtt{uAss}^{\operatorname{Grd}}\ar@{ >->}[r]^-{\sim}_-{\bar\jmath}&\mathcal P}$$
the push-out in $\operad{\operatorname{Grd}}$. The square in the statement is clearly commutative, so it induces a unique compatible morphism $\chi\colon \mathcal P\r \mathcal{U}$. We are going to show that $\chi$ is an isomorphism.

The square in the statement, on objects, is a push-out in $\operad{\operatorname{Set}}$, 
$$\xymatrix{\mathcal{F}(\{e\}[0])\ar[r]^-{\text{incl.}}\ar[d]_{\operatorname{Ob}(\zeta)}\ar@{}[rrd]|{\text{push}}& \mathcal{F}(\{e,e'\}[0])\ar@{=}[r]&\mathcal{F}(\{e\}[0])\amalg \mathcal{F}(\{e'\}[0])\ar[d]^{\operatorname{Ob}(\zeta')=\operatorname{Ob}(\zeta)\amalg (\text{iso. }e'\mapsto u')}\\
\mathtt{uAss}^{\operatorname{Set}}\ar[rr]_-{\text{incl.}}&&\mathtt{uAss}^{\operatorname{Set}}\amalg \mathcal{F}(\{u'\}[0])}$$
Hence, $\chi$ is bijective on objects. Therefore, in order to show that $\chi$ is an isomorphism it is enough to prove that $\mathcal P$ is levelwise contractible. This is obvious. Indeed $\mathcal{F}_{\operatorname{Grd}}(j[0])$ is a generating trivial cofibration,  so $\bar\jmath$ is a trivial cofibration, in particular a weak equivalence, i.e.~$\bar\jmath(n)\colon e=\mathtt{uAss}^{\operatorname{Grd}}(n)\r\mathcal P(n)$ is an equivalence of categories for all $n\geq 0$.
\end{proof}

\begin{lem}\label{uuu}
Consider the commutative square
$$\xymatrix{\mathtt{Ass}^{\operatorname{Grd}}\ar@{ >->}[r]^{\bar\phi^{\operatorname{Grd}}_{\infty}}\ar[d]_{\phi^{\operatorname{Grd}}}& \mathtt{u}_{\infty}\mathtt{A}^{\operatorname{Grd}}\ar[d]^{\psi}_{\sim}\\
\mathtt{uAss}^{\operatorname{Grd}}\ar@{ >->}[r]^-{\sim}_-{\varphi}&\mathcal{U}}$$
where $\varphi$ and $\psi$ were defined in Lemma \ref{ooo} and Remark \ref{precosar}, respectively. The morphisms $\varphi$ and $\psi$ are weak equivalences since their sources and their target are levelwise contractible by definition. We assert that the previous commutative square is a push-out in $\operad{\operatorname{Grd}}$.
\end{lem}

\begin{proof}
In order to warm up, the reader can easily check  that the square  in the statement is a push-out on objects. We tackle directly the statement.  We are going to prove that the square  satisfies the universal property of a push-out. With this purpose, we consider a commutative diagram of solid arrows in $\operad{\operatorname{Grd}}$,
$$\xymatrix{\mathtt{Ass}^{\operatorname{Grd}}\ar@{ >->}[r]^{\bar\phi^{\operatorname{Grd}}_{\infty}}\ar[d]_{\phi^{\operatorname{Grd}}}& \mathtt{u}_{\infty}\mathtt{A}^{\operatorname{Grd}}\ar[d]^{\psi}_{\sim}\ar@/^10pt/[rdd]^{g}\\
\mathtt{uAss}^{\operatorname{Grd}}\ar@{ >->}[r]^-{\sim}_-{\varphi}
\ar@/_10pt/[rrd]_{f}
&\mathcal{U}\ar@{-->}[rd]^{h}\\
&&\mathcal O}$$
where $f\phi^{\operatorname{Grd}}=g\bar\phi^{\operatorname{Grd}}_{\infty}$. We will show that there exists a unique morphism $h$ completing the diagram in a commutative way, i.e.~with two new commutative triangles, $f=h\varphi$ and $g=h\psi$.

The following equation holds in $\mathcal{U}$, 
$$\!\!\!\!\begin{array}{c}
\scalebox{.7} 
{
\begin{pspicture}(0.3,-0.63)(.8,.2)
\psline[linewidth=0.02cm](0.6,0)(0.6,-0.61)
\psdots[dotsize=0.18](0.6,-0.01)
\end{pspicture}
}\end{array}\!\!\!\!\To\!\!\!\!\begin{array}{c}
\scalebox{.7} 
{
\begin{pspicture}(0.3,-0.63)(.8,.2)
\psline[linewidth=0.02cm](0.6,0)(0.6,-0.61)
\psdots[dotsize=0.18,fillstyle=solid,dotstyle=o](0.6,-0.01)
\end{pspicture}
}\end{array}\!\!\!\!=
\left(\begin{array}{c}
\scalebox{.7} 
{
\begin{pspicture}(0.1,-0.63)(1.22,0.7)
\psdots[dotsize=0.18](0.6,-0.01)
\psline[linewidth=0.02cm](0.6,0)(0.6,-0.61)
\psline[linewidth=0.02cm](0.6,0)(0.0,0.59)
\psline[linewidth=0.02cm](0.6,0)(1.2,0.59)

\psdots[dotsize=0.18](1.2,0.59)
\end{pspicture}
}\end{array}\To\!\!\!\!\begin{array}{c}
\scalebox{.7} 
{
\begin{pspicture}(0.3,-0.63)(.8,.1)
\psline[linewidth=0.02cm](0.6,0)(0.6,-0.61)
\end{pspicture}
}\end{array}\!\!\!\!\right)\;\circ_{1}\!\!\!\!\begin{array}{c}
\scalebox{.7} 
{
\begin{pspicture}(0.4,-0.63)(.6,.2)
\psline[linewidth=0.02cm](0.6,0)(0.6,-0.61)
\psdots[dotsize=0.18,fillstyle=solid,dotstyle=o](0.6,-0.01)
\end{pspicture}
}\end{array}.
$$
Therefore, if $h$ existed, it should satisfy
$$h\left(\!\!\!\!\!\begin{array}{c}
\scalebox{.7} 
{
\begin{pspicture}(0.3,-0.63)(.8,.2)
\psline[linewidth=0.02cm](0.6,0)(0.6,-0.61)
\psdots[dotsize=0.18](0.6,-0.01)
\end{pspicture}
}\end{array}\!\!\!\!\To\!\!\!\!\begin{array}{c}
\scalebox{.7} 
{
\begin{pspicture}(0.3,-0.63)(.8,.2)
\psline[linewidth=0.02cm](0.6,0)(0.6,-0.61)
\psdots[dotsize=0.18,fillstyle=solid,dotstyle=o](0.6,-0.01)
\end{pspicture}
}\end{array}\!\!\!\!\right)=
g\left(\begin{array}{c}
\scalebox{.7} 
{
\begin{pspicture}(0.1,-0.63)(1.22,0.7)
\psdots[dotsize=0.18](0.6,-0.01)
\psline[linewidth=0.02cm](0.6,0)(0.6,-0.61)
\psline[linewidth=0.02cm](0.6,0)(0.0,0.59)
\psline[linewidth=0.02cm](0.6,0)(1.2,0.59)

\psdots[dotsize=0.18](1.2,0.59)
\end{pspicture}
}\end{array}\To\!\!\!\!\begin{array}{c}
\scalebox{.7} 
{
\begin{pspicture}(0.3,-0.63)(.8,.1)
\psline[linewidth=0.02cm](0.6,0)(0.6,-0.61)
\end{pspicture}
}\end{array}\!\!\!\!\right)\;\circ_{1}f\left(\!\!\!\!\begin{array}{c}
\scalebox{.7} 
{
\begin{pspicture}(0.4,-0.63)(.6,.2)
\psline[linewidth=0.02cm](0.6,0)(0.6,-0.61)
\psdots[dotsize=0.18,fillstyle=solid,dotstyle=o](0.6,-0.01)
\end{pspicture}
}\end{array}\!\right).
$$
By Lemma \ref{ooo}, there exists a unique  $h$ satisfying this equation and $f=h\varphi$ (notice that the equation implies $h(u)=f(u)$). Hence, it is only left to prove that $g=h\psi$. It is enough to show that this equation holds for the generators in Lemma \ref{generam}. This is obvious for $\mu$, since it comes from $\mathtt{Ass}^{\operatorname{Grd}}$.  For $u\in \mathtt{u}_{\infty}\mathtt{A}^{\operatorname{Grd}}(0)$, which is the black cork,
\begin{align*}
h\left(\!\!\!\!\!\begin{array}{c}
\scalebox{.7} 
{
\begin{pspicture}(0.3,-0.63)(.8,.2)
\psline[linewidth=0.02cm](0.6,0)(0.6,-0.61)
\psdots[dotsize=0.18](0.6,-0.01)
\end{pspicture}
}\end{array}\!\!\!\!\right)&=
g\left(\begin{array}{c}
\scalebox{.7} 
{
\begin{pspicture}(0.1,-0.63)(1.22,0.7)
\psdots[dotsize=0.18](0.6,-0.01)
\psline[linewidth=0.02cm](0.6,0)(0.6,-0.61)
\psline[linewidth=0.02cm](0.6,0)(0.0,0.59)
\psline[linewidth=0.02cm](0.6,0)(1.2,0.59)
\psdots[dotsize=0.18](1.2,0.59)
\end{pspicture}
}\end{array}\right)\;\circ_{1}f\left(\!\!\!\!\begin{array}{c}
\scalebox{.7} 
{
\begin{pspicture}(0.4,-0.63)(.6,.2)
\psline[linewidth=0.02cm](0.6,0)(0.6,-0.61)
\psdots[dotsize=0.18,fillstyle=solid,dotstyle=o](0.6,-0.01)
\end{pspicture}
}\end{array}\!\right)=
\left(g\left(\begin{array}{c}
\scalebox{.7} 
{
\begin{pspicture}(0.1,-0.63)(1.22,0.7)
\psdots[dotsize=0.18](0.6,-0.01)
\psline[linewidth=0.02cm](0.6,0)(0.6,-0.61)
\psline[linewidth=0.02cm](0.6,0)(0.0,0.59)
\psline[linewidth=0.02cm](0.6,0)(1.2,0.59)
\end{pspicture}
}\end{array}\right)\;\circ_{2}g\left(\!\!\!\!\begin{array}{c}
\scalebox{.7} 
{
\begin{pspicture}(0.4,-0.63)(.6,.2)
\psline[linewidth=0.02cm](0.6,0)(0.6,-0.61)
\psdots[dotsize=0.18](0.6,-0.01)
\end{pspicture}
}\end{array}\!\right)\right)
\;\circ_{1}f\left(\!\!\!\!\begin{array}{c}
\scalebox{.7} 
{
\begin{pspicture}(0.4,-0.63)(.6,.2)
\psline[linewidth=0.02cm](0.6,0)(0.6,-0.61)
\psdots[dotsize=0.18,fillstyle=solid,dotstyle=o](0.6,-0.01)
\end{pspicture}
}\end{array}\!\right)\\
&=
\left(f\left(\begin{array}{c}
\scalebox{.7} 
{
\begin{pspicture}(0.1,-0.63)(1.22,0.7)
\psdots[dotsize=0.18](0.6,-0.01)
\psline[linewidth=0.02cm](0.6,0)(0.6,-0.61)
\psline[linewidth=0.02cm](0.6,0)(0.0,0.59)
\psline[linewidth=0.02cm](0.6,0)(1.2,0.59)
\end{pspicture}
}\end{array}\right)\;\circ_{1}f\left(\!\!\!\!\begin{array}{c}
\scalebox{.7} 
{
\begin{pspicture}(0.4,-0.63)(.6,.2)
\psline[linewidth=0.02cm](0.6,0)(0.6,-0.61)
\psdots[dotsize=0.18,fillstyle=solid,dotstyle=o](0.6,-0.01)
\end{pspicture}
}\end{array}\!\right)\right)
\;\circ_{1}g\left(\!\!\!\!\begin{array}{c}
\scalebox{.7} 
{
\begin{pspicture}(0.4,-0.63)(.6,.2)
\psline[linewidth=0.02cm](0.6,0)(0.6,-0.61)
\psdots[dotsize=0.18](0.6,-0.01)
\end{pspicture}
}\end{array}\!\right)
=
f\left(\!\!\!\!\begin{array}{c}
\scalebox{.7} 
{
\begin{pspicture}(0.4,-0.63)(.6,.2)
\psline[linewidth=0.02cm](0.6,0)(0.6,-0.61)
\end{pspicture}
}\end{array}\!\right)
\;\circ_{1}g\left(\!\!\!\!\begin{array}{c}
\scalebox{.7} 
{
\begin{pspicture}(0.4,-0.63)(.6,.2)
\psline[linewidth=0.02cm](0.6,0)(0.6,-0.61)
\psdots[dotsize=0.18](0.6,-0.01)
\end{pspicture}
}\end{array}\!\right)\\
&=
\id{\mathcal O}\circ_{1}g\left(\!\!\!\!\begin{array}{c}
\scalebox{.7} 
{
\begin{pspicture}(0.4,-0.63)(.6,.2)
\psline[linewidth=0.02cm](0.6,0)(0.6,-0.61)
\psdots[dotsize=0.18](0.6,-0.01)
\end{pspicture}
}\end{array}\!\right)
=
g\left(\!\!\!\!\begin{array}{c}
\scalebox{.7} 
{
\begin{pspicture}(0.4,-0.63)(.6,.2)
\psline[linewidth=0.02cm](0.6,0)(0.6,-0.61)
\psdots[dotsize=0.18](0.6,-0.01)
\end{pspicture}
}\end{array}\!\right).
\end{align*}

The generating isomorphisms satisfy the following equations in $\mathcal{U}$, 
\begin{align*}
\begin{array}{c}
\scalebox{.7} 
{
\begin{pspicture}(0,-0.63)(1.22,0.7)
\psdots[dotsize=0.18](0.6,-0.01)
\psline[linewidth=0.02cm](0.6,0)(0.6,-0.61)
\psline[linewidth=0.02cm](0.6,0)(0.0,0.59)
\psline[linewidth=0.02cm](0.6,0)(1.2,0.59)
\psdots[dotsize=0.18](0,0.59)
\end{pspicture}
}\end{array}\st{\lambda}\To\!\!\!\!\begin{array}{c}
\scalebox{.7} 
{
\begin{pspicture}(0.3,-0.63)(.8,0)
\psline[linewidth=0.02cm](0.6,0)(0.6,-0.61)
\end{pspicture}
}\end{array}&=
\begin{array}{c}
\scalebox{.7} 
{
\begin{pspicture}(0,-0.63)(1.22,0.7)
\psdots[dotsize=0.18](0.6,-0.01)
\psline[linewidth=0.02cm](0.6,0)(0.6,-0.61)
\psline[linewidth=0.02cm](0.6,0)(0.0,0.59)
\psline[linewidth=0.02cm](0.6,0)(1.2,0.59)
\end{pspicture}
}\end{array}\circ_{1}
\left(\!\!\!\!\!\begin{array}{c}
\scalebox{.7} 
{
\begin{pspicture}(0.3,-0.63)(.8,.2)
\psline[linewidth=0.02cm](0.6,0)(0.6,-0.61)
\psdots[dotsize=0.18](0.6,-0.01)
\end{pspicture}
}\end{array}\!\!\!\!\To\!\!\!\!\begin{array}{c}
\scalebox{.7} 
{
\begin{pspicture}(0.3,-0.63)(.8,.2)
\psline[linewidth=0.02cm](0.6,0)(0.6,-0.61)
\psdots[dotsize=0.18,fillstyle=solid,dotstyle=o](0.6,-0.01)
\end{pspicture}
}\end{array}\!\!\!\!\right),\\
\begin{array}{c}
\scalebox{.7} 
{
\begin{pspicture}(0,-0.63)(1.22,0.7)
\psdots[dotsize=0.18](0.6,-0.01)
\psline[linewidth=0.02cm](0.6,0)(0.6,-0.61)
\psline[linewidth=0.02cm](0.6,0)(0.0,0.59)
\psline[linewidth=0.02cm](0.6,0)(1.2,0.59)
\psdots[dotsize=0.18](1.2,0.59)
\end{pspicture}
}\end{array}\st{\rho}\To\!\!\!\!\begin{array}{c}
\scalebox{.7} 
{
\begin{pspicture}(0.3,-0.63)(.8,0)
\psline[linewidth=0.02cm](0.6,0)(0.6,-0.61)
\end{pspicture}
}\end{array}&=
\begin{array}{c}
\scalebox{.7} 
{
\begin{pspicture}(0,-0.63)(1.22,0.7)
\psdots[dotsize=0.18](0.6,-0.01)
\psline[linewidth=0.02cm](0.6,0)(0.6,-0.61)
\psline[linewidth=0.02cm](0.6,0)(0.0,0.59)
\psline[linewidth=0.02cm](0.6,0)(1.2,0.59)
\end{pspicture}
}\end{array}\circ_{2}
\left(\!\!\!\!\!\begin{array}{c}
\scalebox{.7} 
{
\begin{pspicture}(0.3,-0.63)(.8,.2)
\psline[linewidth=0.02cm](0.6,0)(0.6,-0.61)
\psdots[dotsize=0.18](0.6,-0.01)
\end{pspicture}
}\end{array}\!\!\!\!\To\!\!\!\!\begin{array}{c}
\scalebox{.7} 
{
\begin{pspicture}(0.3,-0.63)(.8,.2)
\psline[linewidth=0.02cm](0.6,0)(0.6,-0.61)
\psdots[dotsize=0.18,fillstyle=solid,dotstyle=o](0.6,-0.01)
\end{pspicture}
}\end{array}\!\!\!\!\right).
\end{align*}
Hence,
\begin{align*}
h(\lambda)&=
h\left(\begin{array}{c}
\scalebox{.7} 
{
\begin{pspicture}(0,-0.63)(1.22,0.7)
\psdots[dotsize=0.18](0.6,-0.01)
\psline[linewidth=0.02cm](0.6,0)(0.6,-0.61)
\psline[linewidth=0.02cm](0.6,0)(0.0,0.59)
\psline[linewidth=0.02cm](0.6,0)(1.2,0.59)
\end{pspicture}
}\end{array}\right)\circ_{1}
h\left(\!\!\!\!\!\begin{array}{c}
\scalebox{.7} 
{
\begin{pspicture}(0.3,-0.63)(.8,.2)
\psline[linewidth=0.02cm](0.6,0)(0.6,-0.61)
\psdots[dotsize=0.18](0.6,-0.01)
\end{pspicture}
}\end{array}\!\!\!\!\To\!\!\!\!\begin{array}{c}
\scalebox{.7} 
{
\begin{pspicture}(0.3,-0.63)(.8,.2)
\psline[linewidth=0.02cm](0.6,0)(0.6,-0.61)
\psdots[dotsize=0.18,fillstyle=solid,dotstyle=o](0.6,-0.01)
\end{pspicture}
}\end{array}\!\!\!\!\right)\\
&=
g\left(\begin{array}{c}
\scalebox{.7} 
{
\begin{pspicture}(0,-0.63)(1.22,0.7)
\psdots[dotsize=0.18](0.6,-0.01)
\psline[linewidth=0.02cm](0.6,0)(0.6,-0.61)
\psline[linewidth=0.02cm](0.6,0)(0.0,0.59)
\psline[linewidth=0.02cm](0.6,0)(1.2,0.59)
\end{pspicture}
}\end{array}\right)\circ_{1}
\left(
g\left(\begin{array}{c}
\scalebox{.7} 
{
\begin{pspicture}(0.1,-0.63)(1.22,0.7)
\psdots[dotsize=0.18](0.6,-0.01)
\psline[linewidth=0.02cm](0.6,0)(0.6,-0.61)
\psline[linewidth=0.02cm](0.6,0)(0.0,0.59)
\psline[linewidth=0.02cm](0.6,0)(1.2,0.59)
\psdots[dotsize=0.18](1.2,0.59)
\end{pspicture}
}\end{array}\To\!\!\!\!\begin{array}{c}
\scalebox{.7} 
{
\begin{pspicture}(0.3,-0.63)(.8,.1)
\psline[linewidth=0.02cm](0.6,0)(0.6,-0.61)
\end{pspicture}
}\end{array}\!\!\!\!\right)\;\circ_{1}f\left(\!\!\!\!\begin{array}{c}
\scalebox{.7} 
{
\begin{pspicture}(0.4,-0.63)(.6,.2)
\psline[linewidth=0.02cm](0.6,0)(0.6,-0.61)
\psdots[dotsize=0.18,fillstyle=solid,dotstyle=o](0.6,-0.01)
\end{pspicture}
}\end{array}\!\right)
\right)
\\
&=
\left(g\left(\begin{array}{c}
\scalebox{.7} 
{
\begin{pspicture}(0,-0.63)(1.22,0.7)
\psdots[dotsize=0.18](0.6,-0.01)
\psline[linewidth=0.02cm](0.6,0)(0.6,-0.61)
\psline[linewidth=0.02cm](0.6,0)(0.0,0.59)
\psline[linewidth=0.02cm](0.6,0)(1.2,0.59)
\end{pspicture}
}\end{array}\right)\circ_{1}
g\left(\begin{array}{c}
\scalebox{.7} 
{
\begin{pspicture}(0.1,-0.63)(1.22,0.7)
\psdots[dotsize=0.18](0.6,-0.01)
\psline[linewidth=0.02cm](0.6,0)(0.6,-0.61)
\psline[linewidth=0.02cm](0.6,0)(0.0,0.59)
\psline[linewidth=0.02cm](0.6,0)(1.2,0.59)
\psdots[dotsize=0.18](1.2,0.59)
\end{pspicture}
}\end{array}\To\!\!\!\!\begin{array}{c}
\scalebox{.7} 
{
\begin{pspicture}(0.3,-0.63)(.8,.1)
\psline[linewidth=0.02cm](0.6,0)(0.6,-0.61)
\end{pspicture}
}\end{array}\!\!\!\!\right)\right)\;\circ_{1}f\left(\!\!\!\!\begin{array}{c}
\scalebox{.7} 
{
\begin{pspicture}(0.4,-0.63)(.6,.2)
\psline[linewidth=0.02cm](0.6,0)(0.6,-0.61)
\psdots[dotsize=0.18,fillstyle=solid,dotstyle=o](0.6,-0.01)
\end{pspicture}
}\end{array}\!\right)
\\
&=
g\left(\begin{array}{c}
\scalebox{.7} 
{
\begin{pspicture}(0.1,-0.63)(1.22,0.7)
\psdots[dotsize=0.18](0.6,-0.01)
\psline[linewidth=0.02cm](0.6,0)(0.6,-0.61)
\psline[linewidth=0.02cm](0.6,0)(0.0,0.59)
\psline[linewidth=0.02cm](0.6,0)(0.6,0.59)
\psline[linewidth=0.02cm](0.6,0)(1.2,0.59)
\psdots[dotsize=0.18](.6,0.59)
\end{pspicture}
}\end{array}\To\!\!\!\!\begin{array}{c}
\scalebox{.7} 
{
\begin{pspicture}(0,-0.63)(1.22,0.7)
\psdots[dotsize=0.18](0.6,-0.01)
\psline[linewidth=0.02cm](0.6,0)(0.6,-0.61)
\psline[linewidth=0.02cm](0.6,0)(0.0,0.59)
\psline[linewidth=0.02cm](0.6,0)(1.2,0.59)
\end{pspicture}
}\end{array}\right)\;\circ_{1}f\left(\!\!\!\!\begin{array}{c}
\scalebox{.7} 
{
\begin{pspicture}(0.4,-0.63)(.6,.2)
\psline[linewidth=0.02cm](0.6,0)(0.6,-0.61)
\psdots[dotsize=0.18,fillstyle=solid,dotstyle=o](0.6,-0.01)
\end{pspicture}
}\end{array}\!
\right)
\\
&=
\left(g\left(\begin{array}{c}
\scalebox{.7} 
{
\begin{pspicture}(0,-0.63)(1.22,0.7)
\psdots[dotsize=0.18](0.6,-0.01)
\psline[linewidth=0.02cm](0.6,0)(0.6,-0.61)
\psline[linewidth=0.02cm](0.6,0)(0.0,0.59)
\psline[linewidth=0.02cm](0.6,0)(1.2,0.59)
\end{pspicture}
}\end{array}\right)\circ_{2}
g\left(\begin{array}{c}
\scalebox{.7} 
{
\begin{pspicture}(0.1,-0.63)(1.22,0.7)
\psdots[dotsize=0.18](0.6,-0.01)
\psline[linewidth=0.02cm](0.6,0)(0.6,-0.61)
\psline[linewidth=0.02cm](0.6,0)(0.0,0.59)
\psline[linewidth=0.02cm](0.6,0)(1.2,0.59)
\psdots[dotsize=0.18](0,0.59)
\end{pspicture}
}\end{array}\To\!\!\!\!\begin{array}{c}
\scalebox{.7} 
{
\begin{pspicture}(0.3,-0.63)(.8,.1)
\psline[linewidth=0.02cm](0.6,0)(0.6,-0.61)
\end{pspicture}
}\end{array}\!\!\!\!\right)\right)\;\circ_{1}f\left(\!\!\!\!\begin{array}{c}
\scalebox{.7} 
{
\begin{pspicture}(0.4,-0.63)(.6,.2)
\psline[linewidth=0.02cm](0.6,0)(0.6,-0.61)
\psdots[dotsize=0.18,fillstyle=solid,dotstyle=o](0.6,-0.01)
\end{pspicture}
}\end{array}\!\right)\\
&=
\left(f\left(\begin{array}{c}
\scalebox{.7} 
{
\begin{pspicture}(0.1,-0.63)(1.22,0.7)
\psdots[dotsize=0.18](0.6,-0.01)
\psline[linewidth=0.02cm](0.6,0)(0.6,-0.61)
\psline[linewidth=0.02cm](0.6,0)(0.0,0.59)
\psline[linewidth=0.02cm](0.6,0)(1.2,0.59)
\end{pspicture}
}\end{array}\right)\;\circ_{1}f\left(\!\!\!\!\begin{array}{c}
\scalebox{.7} 
{
\begin{pspicture}(0.4,-0.63)(.6,.2)
\psline[linewidth=0.02cm](0.6,0)(0.6,-0.61)
\psdots[dotsize=0.18,fillstyle=solid,dotstyle=o](0.6,-0.01)
\end{pspicture}
}\end{array}\!\right)\right)
\circ_{1}
g\left(\begin{array}{c}
\scalebox{.7} 
{
\begin{pspicture}(0.1,-0.63)(1.22,0.7)
\psdots[dotsize=0.18](0.6,-0.01)
\psline[linewidth=0.02cm](0.6,0)(0.6,-0.61)
\psline[linewidth=0.02cm](0.6,0)(0.0,0.59)
\psline[linewidth=0.02cm](0.6,0)(1.2,0.59)
\psdots[dotsize=0.18](0,0.59)
\end{pspicture}
}\end{array}\To\!\!\!\!\begin{array}{c}
\scalebox{.7} 
{
\begin{pspicture}(0.3,-0.63)(.8,.1)
\psline[linewidth=0.02cm](0.6,0)(0.6,-0.61)
\end{pspicture}
}\end{array}\!\!\!\!\right)\\
&=
g\left(\begin{array}{c}
\scalebox{.7} 
{
\begin{pspicture}(0.1,-0.63)(1.22,0.7)
\psdots[dotsize=0.18](0.6,-0.01)
\psline[linewidth=0.02cm](0.6,0)(0.6,-0.61)
\psline[linewidth=0.02cm](0.6,0)(0.0,0.59)
\psline[linewidth=0.02cm](0.6,0)(1.2,0.59)
\psdots[dotsize=0.18](0,0.59)
\end{pspicture}
}\end{array}\To\!\!\!\!\begin{array}{c}
\scalebox{.7} 
{
\begin{pspicture}(0.3,-0.63)(.8,.1)
\psline[linewidth=0.02cm](0.6,0)(0.6,-0.61)
\end{pspicture}
}\end{array}\!\!\!\!\right)=g(\lambda),
\end{align*}
\begin{align*}
h(\rho)&=
h\left(\begin{array}{c}
\scalebox{.7} 
{
\begin{pspicture}(0,-0.63)(1.22,0.7)
\psdots[dotsize=0.18](0.6,-0.01)
\psline[linewidth=0.02cm](0.6,0)(0.6,-0.61)
\psline[linewidth=0.02cm](0.6,0)(0.0,0.59)
\psline[linewidth=0.02cm](0.6,0)(1.2,0.59)
\end{pspicture}
}\end{array}\right)\circ_{2}
h\left(\!\!\!\!\!\begin{array}{c}
\scalebox{.7} 
{
\begin{pspicture}(0.3,-0.63)(.8,.2)
\psline[linewidth=0.02cm](0.6,0)(0.6,-0.61)
\psdots[dotsize=0.18](0.6,-0.01)
\end{pspicture}
}\end{array}\!\!\!\!\To\!\!\!\!\begin{array}{c}
\scalebox{.7} 
{
\begin{pspicture}(0.3,-0.63)(.8,.2)
\psline[linewidth=0.02cm](0.6,0)(0.6,-0.61)
\psdots[dotsize=0.18,fillstyle=solid,dotstyle=o](0.6,-0.01)
\end{pspicture}
}\end{array}\!\!\!\!\right)\\
&=
g\left(\begin{array}{c}
\scalebox{.7} 
{
\begin{pspicture}(0,-0.63)(1.22,0.7)
\psdots[dotsize=0.18](0.6,-0.01)
\psline[linewidth=0.02cm](0.6,0)(0.6,-0.61)
\psline[linewidth=0.02cm](0.6,0)(0.0,0.59)
\psline[linewidth=0.02cm](0.6,0)(1.2,0.59)
\end{pspicture}
}\end{array}\right)\circ_{2}
\left(
g\left(\begin{array}{c}
\scalebox{.7} 
{
\begin{pspicture}(0.1,-0.63)(1.22,0.7)
\psdots[dotsize=0.18](0.6,-0.01)
\psline[linewidth=0.02cm](0.6,0)(0.6,-0.61)
\psline[linewidth=0.02cm](0.6,0)(0.0,0.59)
\psline[linewidth=0.02cm](0.6,0)(1.2,0.59)
\psdots[dotsize=0.18](1.2,0.59)
\end{pspicture}
}\end{array}\To\!\!\!\!\begin{array}{c}
\scalebox{.7} 
{
\begin{pspicture}(0.3,-0.63)(.8,.1)
\psline[linewidth=0.02cm](0.6,0)(0.6,-0.61)
\end{pspicture}
}\end{array}\!\!\!\!\right)\;\circ_{1}f\left(\!\!\!\!\begin{array}{c}
\scalebox{.7} 
{
\begin{pspicture}(0.4,-0.63)(.6,.2)
\psline[linewidth=0.02cm](0.6,0)(0.6,-0.61)
\psdots[dotsize=0.18,fillstyle=solid,dotstyle=o](0.6,-0.01)
\end{pspicture}
}\end{array}\!\right)
\right)
\\
&=
\left(g\left(\begin{array}{c}
\scalebox{.7} 
{
\begin{pspicture}(0,-0.63)(1.22,0.7)
\psdots[dotsize=0.18](0.6,-0.01)
\psline[linewidth=0.02cm](0.6,0)(0.6,-0.61)
\psline[linewidth=0.02cm](0.6,0)(0.0,0.59)
\psline[linewidth=0.02cm](0.6,0)(1.2,0.59)
\end{pspicture}
}\end{array}\right)\circ_{2}
g\left(\begin{array}{c}
\scalebox{.7} 
{
\begin{pspicture}(0.1,-0.63)(1.22,0.7)
\psdots[dotsize=0.18](0.6,-0.01)
\psline[linewidth=0.02cm](0.6,0)(0.6,-0.61)
\psline[linewidth=0.02cm](0.6,0)(0.0,0.59)
\psline[linewidth=0.02cm](0.6,0)(1.2,0.59)
\psdots[dotsize=0.18](1.2,0.59)
\end{pspicture}
}\end{array}\To\!\!\!\!\begin{array}{c}
\scalebox{.7} 
{
\begin{pspicture}(0.3,-0.63)(.8,.1)
\psline[linewidth=0.02cm](0.6,0)(0.6,-0.61)
\end{pspicture}
}\end{array}\!\!\!\!\right)\right)\;\circ_{2}f\left(\!\!\!\!\begin{array}{c}
\scalebox{.7} 
{
\begin{pspicture}(0.4,-0.63)(.6,.2)
\psline[linewidth=0.02cm](0.6,0)(0.6,-0.61)
\psdots[dotsize=0.18,fillstyle=solid,dotstyle=o](0.6,-0.01)
\end{pspicture}
}\end{array}\!\right)
\\
&=
g\left(\begin{array}{c}
\scalebox{.7} 
{
\begin{pspicture}(0.1,-0.63)(1.22,0.7)
\psdots[dotsize=0.18](0.6,-0.01)
\psline[linewidth=0.02cm](0.6,0)(0.6,-0.61)
\psline[linewidth=0.02cm](0.6,0)(0.0,0.59)
\psline[linewidth=0.02cm](0.6,0)(0.6,0.59)
\psline[linewidth=0.02cm](0.6,0)(1.2,0.59)
\psdots[dotsize=0.18](1.2,0.59)
\end{pspicture}
}\end{array}\To\!\!\!\!\begin{array}{c}
\scalebox{.7} 
{
\begin{pspicture}(0,-0.63)(1.22,0.7)
\psdots[dotsize=0.18](0.6,-0.01)
\psline[linewidth=0.02cm](0.6,0)(0.6,-0.61)
\psline[linewidth=0.02cm](0.6,0)(0.0,0.59)
\psline[linewidth=0.02cm](0.6,0)(1.2,0.59)
\end{pspicture}
}\end{array}\right)\;\circ_{2}f\left(\!\!\!\!\begin{array}{c}
\scalebox{.7} 
{
\begin{pspicture}(0.4,-0.63)(.6,.2)
\psline[linewidth=0.02cm](0.6,0)(0.6,-0.61)
\psdots[dotsize=0.18,fillstyle=solid,dotstyle=o](0.6,-0.01)
\end{pspicture}
}\end{array}\!
\right)
\\
&=
\left(
g\left(\begin{array}{c}
\scalebox{.7} 
{
\begin{pspicture}(0.1,-0.63)(1.22,0.7)
\psdots[dotsize=0.18](0.6,-0.01)
\psline[linewidth=0.02cm](0.6,0)(0.6,-0.61)
\psline[linewidth=0.02cm](0.6,0)(0.0,0.59)
\psline[linewidth=0.02cm](0.6,0)(1.2,0.59)
\psdots[dotsize=0.18](1.2,0.59)
\end{pspicture}
}\end{array}\To\!\!\!\!\begin{array}{c}
\scalebox{.7} 
{
\begin{pspicture}(0.3,-0.63)(.8,.1)
\psline[linewidth=0.02cm](0.6,0)(0.6,-0.61)
\end{pspicture}
}\end{array}\!\!\!\!\right)\circ_{1}g\left(\begin{array}{c}
\scalebox{.7} 
{
\begin{pspicture}(0,-0.63)(1.22,0.7)
\psdots[dotsize=0.18](0.6,-0.01)
\psline[linewidth=0.02cm](0.6,0)(0.6,-0.61)
\psline[linewidth=0.02cm](0.6,0)(0.0,0.59)
\psline[linewidth=0.02cm](0.6,0)(1.2,0.59)
\end{pspicture}
}\end{array}\right)\right)\;\circ_{2}f\left(\!\!\!\!\begin{array}{c}
\scalebox{.7} 
{
\begin{pspicture}(0.4,-0.63)(.6,.2)
\psline[linewidth=0.02cm](0.6,0)(0.6,-0.61)
\psdots[dotsize=0.18,fillstyle=solid,dotstyle=o](0.6,-0.01)
\end{pspicture}
}\end{array}\!\right)\\
&=
g\left(\begin{array}{c}
\scalebox{.7} 
{
\begin{pspicture}(0.1,-0.63)(1.22,0.7)
\psdots[dotsize=0.18](0.6,-0.01)
\psline[linewidth=0.02cm](0.6,0)(0.6,-0.61)
\psline[linewidth=0.02cm](0.6,0)(0.0,0.59)
\psline[linewidth=0.02cm](0.6,0)(1.2,0.59)
\psdots[dotsize=0.18](1.2,0.59)
\end{pspicture}
}\end{array}\To\!\!\!\!\begin{array}{c}
\scalebox{.7} 
{
\begin{pspicture}(0.3,-0.63)(.8,.1)
\psline[linewidth=0.02cm](0.6,0)(0.6,-0.61)
\end{pspicture}
}\end{array}\!\!\!\!\right)\circ_{1}\left(f\left(\begin{array}{c}
\scalebox{.7} 
{
\begin{pspicture}(0,-0.63)(1.22,0.7)
\psdots[dotsize=0.18](0.6,-0.01)
\psline[linewidth=0.02cm](0.6,0)(0.6,-0.61)
\psline[linewidth=0.02cm](0.6,0)(0.0,0.59)
\psline[linewidth=0.02cm](0.6,0)(1.2,0.59)
\end{pspicture}
}\end{array}\right)\;\circ_{2}f\left(\!\!\!\!\begin{array}{c}
\scalebox{.7} 
{
\begin{pspicture}(0.4,-0.63)(.6,.2)
\psline[linewidth=0.02cm](0.6,0)(0.6,-0.61)
\psdots[dotsize=0.18,fillstyle=solid,dotstyle=o](0.6,-0.01)
\end{pspicture}
}\end{array}\!\right)\right)\\
&=
g\left(\begin{array}{c}
\scalebox{.7} 
{
\begin{pspicture}(0.1,-0.63)(1.22,0.7)
\psdots[dotsize=0.18](0.6,-0.01)
\psline[linewidth=0.02cm](0.6,0)(0.6,-0.61)
\psline[linewidth=0.02cm](0.6,0)(0.0,0.59)
\psline[linewidth=0.02cm](0.6,0)(1.2,0.59)
\psdots[dotsize=0.18](1.2,0.59)
\end{pspicture}
}\end{array}\To\!\!\!\!\begin{array}{c}
\scalebox{.7} 
{
\begin{pspicture}(0.3,-0.63)(.8,.1)
\psline[linewidth=0.02cm](0.6,0)(0.6,-0.61)
\end{pspicture}
}\end{array}\!\!\!\!\right)=g(\rho).
\end{align*}
This concludes the proof.
\end{proof}

\section{Main theorem for DG-operads}\label{DG}

Let $\operatorname{Ch}(\Bbbk)$ be the category of DG-modules over a ground commutative ring $\Bbbk$.   Weak equivalences in $\operatorname{Ch}(\Bbbk)$ are quasi-isomorphisms and fibrations are levelwise surjective maps. In this way, $\operatorname{Ch}(\Bbbk)$ with the usual tensor product becomes a combinatorial closed symmetric monoidal model category with cofibrant tensor unit and satisfying the monoid axiom, compare \cite[Proposition 4.2.13]{hmc}. Hence, $\operad{\operatorname{Ch}(\Bbbk)}$ has the model structure in Theorem \ref{modelo}. The main result of this section is the following theorem.

\begin{thm}\label{hepiDG}
The morphism $\phi^{\operatorname{Ch}(\Bbbk)}\colon \mathtt{Ass}^{\operatorname{Ch}(\Bbbk)}\r \mathtt{uAss}^{\operatorname{Ch}(\Bbbk)}$ in Example \ref{comparison} is a homotopy epimorphism in $\operad{\operatorname{Ch}(\Bbbk)}$. 
\end{thm}

This theorem follows from Lemma \ref{basico} above and Corollary \ref{ultimo} below.

Operads in ${\operatorname{Mod}(\Bbbk)^{\mathbb Z}}$ and ${\operatorname{Ch}(\Bbbk)}$ are called \emph{graded operads} and \emph{DG-operads}, respectively.  In this section we consider free graded operads and free DG-operads. For the sake of simplicity, we omit the subscript from the free graded operad functor $\mathcal F=\mathcal F_{\operatorname{Mod}(\Bbbk)^{\mathbb Z}}$, but not from the free DG-operad functor $\mathcal F_{\operatorname{Ch}(\Bbbk)}$, in order to avoid confusion. The underlying graded operad of a free DG-operad $\mathcal F_{\operatorname{Ch}(\Bbbk)}(V)$ on a sequence of DG-modules $V$ is the free graded operad $\mathcal F(V)$ on the underlying sequence of graded modules. This follows from the fact that the forgetful functor $\operatorname{Ch}(\Bbbk)\r \operatorname{Mod}(\Bbbk)^{\mathbb Z}$ strictly preserves tensor products, compare Remark \ref{free}. Moreover, this forgetful functor is a left adjoint since it is a colimit preserving functor between locally 
presentable categories, see \cite[Theorem 3.3.4]{borceux1} and \cite[Theorem 1.58]{adamekrosicky} (it is actually easy to construct an explicit right adjoint), hence it induces a functor between operad categories $\operad{\operatorname{Ch}(\Bbbk)}\r \operad{\operatorname{Mod}(\Bbbk)^{\mathbb Z}}$ which is also a left adjoint, in particular it preserves all colimits.

We denote by $S^n$ the module $\Bbbk$ regarded as a DG-module concentrated in degree~$n$,
$$S^n=\{\cdots\r0\r \hspace{-12pt}\mathop{\Bbbk}\limits_{\text{degree }n}\hspace{-12pt}\r0\r\cdots\}.$$
We denote  its mapping cone by $D^{n+1}$
$$D^{n+1}=\{\cdots\r0\r\Bbbk\st{1_{\Bbbk}}\r \hspace{-12pt}\mathop{\Bbbk}\limits_{\text{degree }n}\hspace{-12pt}\r0\r\cdots\}.$$
Let $f_n\colon S^n\r D^{n+1}$ be the morphism given by the identity in degree $n$. The standard sets of generating (trivial) cofibrations in $\operatorname{Ch}(\Bbbk)$ are
$$I=\{f_n\}_{n\in\mathbb Z},\qquad J=\{0\r D^{n+1}\}_{n\in\mathbb Z}.$$
Hence $\mathcal F_{\operatorname{Ch}(\Bbbk)}(I_{\mathbb N})$ and $\mathcal F_{\operatorname{Ch}(\Bbbk)}(J_{\mathbb N})$ are sets of generating (trivial) cofibrations in $\operad{\operatorname{Ch}(\Bbbk)}$, see Remark \ref{generatriz}.

We now analyze the structure of relative $\mathcal F_{\operatorname{Ch}(\Bbbk)}(I_{\mathbb N})$-cell complexes. The category of \emph{graded sets} is simply $\operatorname{Set}^{\mathbb Z}$. A \emph{sequence of graded sets} $S$, i.e.~an object in $(\operatorname{Set}^{\mathbb Z})^{\mathbb N}$, can be regarded as a plain set equipped with a map $S\r\mathbb N\times\mathbb Z$ sending each element to the pair given by its arity and its degree. The sequence of free graded modules $\Bbbk\cdot S$ is obtained by taking free graded module aritywise.

\begin{prop}\label{estructura}
A morphism of DG-operads $i\colon\mathcal O\r\mathcal P$ is a relative $\mathcal F_{\operatorname{Ch}(\Bbbk)}(I_{\mathbb N})$-cell complex if and only if there is a sequence of graded subsets $S\subset\mathcal P$ 
equipped with a continuous filtration $\{S_{\beta}\}_{\beta\leq\alpha}$, $\alpha$ an ordinal, $S_{0}=\varnothing$, $S_{\alpha}=S$, 
such that $i$ and  the inclusion $S\subset\mathcal P$  induce an isomorphism of graded operads $\mathcal O\amalg\mathcal F(\Bbbk\cdot S)=\mathcal P$ and, for any $\beta<\alpha$,  
$$d(S_{\beta+1})\subset \mathcal O\amalg\mathcal F(\Bbbk\cdot S_{\beta}).$$
\end{prop}

This proposition is a consequence of the following characterization of push-outs of generating cofibrations in $\operad{\operatorname{Ch}(\Bbbk)}$.

\begin{lem}\label{unosolo}
Given a DG-operad $\mathcal P$ and a cycle $y\in\mathcal P(m)_{n}$, $d(y)=0$, there is a unique DG-operad structure $\mathcal Q$ on $\mathcal P\amalg\mathcal F(\Bbbk\cdot\{x\})$, where $x$ has arity $m$ and degree $n+1$, such that the inclusion of the first factor $j\colon \mathcal P\r \mathcal P\amalg\mathcal F(\Bbbk\cdot\{x\})$ is a morphism of DG-operads $j\colon\mathcal P\r\mathcal Q$ and $d(x)=y$. Moreover, such a $j$ is exactly the same as a push-out of $\mathcal F_{\operatorname{Ch}(\Bbbk)}(f_{n}[m])$.
\end{lem}

\begin{proof}
A morphism $h\colon S^{n}\r X$ in $\operatorname{Ch}(\Bbbk)$ is determined by the choice of a cycle $h(1)\in X_{n}$, hence a cycle $y\in\mathcal P(m)_{n}$ is the same as a morphism $g\colon\mathcal F_{\operatorname{Ch}(\Bbbk)}(S^{n}[m])\r\mathcal P$. As a graded module $D^{n+1}=S^{n}\amalg S^{n+1}$ and $f_{n}$ is the inclusion of the first factor, therefore the push-out 
$$\xymatrix@C=50pt{
\mathcal F_{\operatorname{Ch}(\Bbbk)}(S^n[m])
\ar@{>->}[r]^-{\mathcal F_{\operatorname{Ch}(\Bbbk)}(f_n[m])}\ar[d]_g\ar@{}[rd]|{\text{push}}&
\mathcal F_{\operatorname{Ch}(\Bbbk)}(D^{n+1}[m])\ar[d]^{\bar g}\\
\mathcal P\ar@{>->}[r]^-j&\mathcal Q
}$$
is $\mathcal Q=\mathcal P\amalg \mathcal F(S^{n+1}[m])$ as a graded operad. Notice also that $S^{n+1}[m]=\Bbbk\cdot\{x\}$ for $x$ of arity $m$ and degree $n+1$.  As a DG-operad, $\mathcal Q$  is characterized by the fact that $j$ is a morphism in $\operad{\operatorname{Ch}(\Bbbk)}$ and that the morphism $D^{n+1}\r \mathcal Q(m)$ defined by $\bar g$, which sends the generators of degree $n$ and $n+1$ to $y$ and $x$, respectively, preserves the differential. The last condition is equivalent to $d(x)=y$.
\end{proof}

Lemma \ref{unosolo} shows the apparently stronger statement that in Proposition \ref{estructura} we can always take $S_{\beta+1}\setminus S_{\beta}$ to be a singleton, $\beta<\alpha$. Both statements are equivalent since we can refine the filtration by putting well orderings on the sets $S_{\beta+1}\setminus S_{\beta}$, $\beta<\alpha$.

We wish to stress that $\mathcal Q$ in the statement of Lemma \ref{unosolo} is the graded operad $\mathcal P\amalg\mathcal F(\Bbbk\cdot\{x\})$ equipped with the differential defined therein.

We now present 
a  $u$-infinity associative operad in $\operatorname{Ch}(\Bbbk)$, in the sense of Definition \ref{uiu}. 

\begin{defn}\label{rafiki}
The \emph{$u$-infinity associative DG-operad} $\mathtt{u}_{\infty}\mathtt{A}^{\operatorname{Ch}(\Bbbk)}$ is built on the graded operad with generators
\begin{align*}
\mu&\in \mathtt{u}_{\infty}\mathtt{A}^{\operatorname{Ch}(\Bbbk)}(2)_{0}, &
\nu_n^{S}&\in \mathtt{u}_{\infty}\mathtt{A}^{\operatorname{Ch}(\Bbbk)}(n-|S|)_{n-2+|S|},
\end{align*}
where $n>0$, $\varnothing\neq S\subset\{1,\dots,n\}$, and $|S|$ is the cardinal of $S$, 
and relations
$$\mu\circ_1\mu=\mu\circ_2\mu.$$
The differential is given by 
\begin{align*}
d(\mu)&=0,&
d(\nu_2^{\{1\}})&=\mu\circ_1\nu_1^{\{1\}}-\id{},\\
d(\nu_1^{\{1\}})&=0,&d(\nu_2^{\{2\}})&=\mu\circ_2\nu_1^{\{1\}}-\id{},
\end{align*}
and if $n>1$, $(n,|S|)\neq (2,1)$,
\begin{align*}
d(\nu_n^S)={}&(-1)^{n}\mu\circ_1\nu_{n-1}^{S}\quad\text{unless $n$ is the last element of }S\\
&+\mu\circ_2\nu_{n-1}^{S-1}\quad\text{unless $1$ is the first element of }S\\
&+\hspace{-20pt}\sum_{
\begin{array}{c}\\[-17pt]
\scriptstyle 1\leq v\leq |S|+1\\[-5pt]
\scriptstyle l_{v-1}< i+v-1< l_{v}-1
\end{array}
}
\hspace{-20pt}
(-1)^{i+v-1}\nu_{n-1}^{S_{v}\cup(S_{v}'-1)}\circ_i\mu\\
&+\sum_{
\begin{array}{c}\\[-17pt]
\scriptstyle p+q=n+1 \\[-5pt]
\scriptstyle 1\leq i\leq p-|S_1|\\[-5pt]
\scriptstyle S_1\circ_iS_2=S\\[-5pt]
\scriptstyle S_1,S_2\neq \varnothing
\end{array}
}(-1)^{q(p-|S_1|)+(q-1)(i+r-1)+|S_2|(r-1)}\nu_p^{S_1}\circ_i\nu_q^{S_2}.
\end{align*}
Here we denote $S+m=\{s+m\,;\,s\in S\}$, 
 $S=\{l_{1},\dots,l_{|S|}\}\subset\{1,\dots,n\}$,  $l_0=0$, $l_{|S|+1}=n+1$,
\begin{align*}
S_{v}&=\{l_{1},\dots,l_{v-1}\},&S'_{v}=&S\setminus S_{v}=\{l_{v},\dots,l_{|S|}\},&1\leq v\leq |S|+1.
\end{align*}
Moreover, if
$$S_1=\{j_1,\dots,j_s\}\subset\{1,\dots,p\},\qquad
S_2=\{k_1,\dots,k_t\}\subset\{1,\dots,q\},$$
and the $i^{\text{th}}$ element of the complement of $S_1$ lies between $j_{r-1}$ and $j_r$, then the subset $S_1\circ_i S_2\subset\{1,\dots,n\}$ is
$$S_1\circ_i S_2=\{j_1,\dots,j_{r-1},k_1+i+r-2,\dots,k_t+i+r-2,j_r+q-1,\dots, j_s+q-1\}.$$
We understand that $r=1$ if $i<j_1$, and $r=s+1$ if the $i^{\text{th}}$ element of the complement of $S_1$ is bigger than $j_s$. 

Notice that in the last line of the generic definition of $d(\nu_n^S)$, $|S|=|S_1|+|S_2|$ and $|S_1|,|S_2|>0$, so $|S_1|,|S_2|<|S|$.

The factorization
$$\mathtt{Ass}^{\operatorname{Ch}(\Bbbk)}\st{\bar \phi^{\operatorname{Ch}(\Bbbk)}_\infty}\into \mathtt{u}_{\infty}\mathtt{A}^{\operatorname{Ch}(\Bbbk)}\st{\sim}\To \mathtt{uAss}^{\operatorname{Ch}(\Bbbk)}$$
is given as follows: $\bar \phi^{\operatorname{Ch}(\Bbbk)}_\infty$ is the obvious inclusion, $\bar \phi^{\operatorname{Ch}(\Bbbk)}_\infty(\mu)=\mu$, and the weak equivalence (actually a trivial fibration in this case) is defined by
\begin{align*}
\mu&\mapsto\mu,&
\nu_1^{\{1\}}&\mapsto u,&
\nu_n^S&\mapsto 0,\quad n>1.
\end{align*}
Notice that $\bar \phi^{\operatorname{Ch}(\Bbbk)}_\infty$ is a relative $\mathcal F_{\operatorname{Ch}(\Bbbk)}(I_{\mathbb N})$-cell complex, see Remark \ref{itis} below.
\end{defn}

\begin{rem}\label{elu}
The DG-operad $\mathtt{u}_{\infty}\mathtt{A}^{\operatorname{Ch}(\Bbbk)}$ has been obtained by applying the method in Remark \ref{construcciones}. Indeed, a cofibrant resolution $\phi_{\infty}^{\operatorname{Ch}(\Bbbk)}\colon \mathtt{A}_{\infty}^{\operatorname{Ch}(\Bbbk)}\into \mathtt{uA}_{\infty}^{\operatorname{Ch}(\Bbbk)}$ of $\phi^{\operatorname{Ch}(\Bbbk)}$ has been constructed in \cite{uass}. Here $\mathtt{A}_{\infty}^{\operatorname{Ch}(\Bbbk)}$ is Stasheff's \emph{$A$-infinity DG-operad} \cite{hahs}, given by the cellular homology of associahedra $K_n$, $\mathtt{A}_{\infty}^{\operatorname{Ch}(\Bbbk)}(n)=C_*(K_n,\Bbbk)$, $n\geq 0$. The \emph{unital $A$-infinity DG-operad} $\mathtt{uA}_{\infty}^{\operatorname{Ch}(\Bbbk)}$ is similarly given by the cellular homology of the unital associahedra $K^u_n$ introduced in \cite{uass}, $\mathtt{uA}_{\infty}^{\operatorname{Ch}(\Bbbk)}(n)=C_*(K^u_n,\Bbbk)$, and the morphism $\phi_{\infty}^{\operatorname{Ch}(\Bbbk)}$ is induced by the inclusion of 
associahedra into unital associahedra $K_n\subset K^u_n$, $n\geq 0$. 

Recall that, as a graded operad, $\mathtt{A}_{\infty}^{\operatorname{Ch}(\Bbbk)}$ is free on generators $\mu_n$ of arity $n$ and degree $n-2$, $n\geq 2$. Analogously, $\mathtt{uA}_{\infty}^{\operatorname{Ch}(\Bbbk)}$ is free as a graded operad on generators $\mu_n^S$, $n\geq 1$, $S\subset\{1,\dots, n\}$, $(n,S)\neq (1,\varnothing)$, of arity $n-|S|$ and degree $n+|S|-2$. We refer the reader to \cite[\S5]{uass} for the definition of the differentials in these DG-operads. The morphism $\phi_{\infty}^{\operatorname{Ch}(\Bbbk)}$ is given by $\mu_n\mapsto\mu_n^{\varnothing}$, and it is actually a relative $\mathcal F_{\operatorname{Ch}(\Bbbk)}(I_{\mathbb N})$-cell complex by Proposition \ref{estructura} since
$d(\mu_n^S)\in\mathcal F(\Bbbk\cdot\{\mu_{n'}^{S'}\,;\,n'<n\text{ or }|S'|<|S|\})$. The vertical weak equivalences in the commutative square
$$\xymatrix{\mathtt{A}^{\operatorname{Ch}(\Bbbk)}_{\infty}\ar@{ >->}[r]^{\phi^{\operatorname{Ch}(\Bbbk)}_{\infty}}\ar[d]_{\sim}& \mathtt{uA}^{\operatorname{Ch}(\Bbbk)}_{\infty}\ar[d]^{\sim}\\
\mathtt{Ass}^{\operatorname{Ch}(\Bbbk)}\ar[r]^{\phi^{\operatorname{Ch}(\Bbbk)}}&\mathtt{uAss}^{\operatorname{Ch}(\Bbbk)}}$$
are respectively given by
\begin{align*}
\mu_2&\mapsto\mu,
&
\mu_2^{\varnothing}&\mapsto\mu,\\
\mu_n&\mapsto 0, \quad n>2,&
\mu_1^{\{1\}}&\mapsto u,\\
&&\mu_n^S&\mapsto 0\quad\text{otherwise}.
\end{align*}

Consider the square $$\xymatrix{\mathtt{A}^{\operatorname{Ch}(\Bbbk)}_{\infty}\ar@{ >->}[r]^{\phi^{\operatorname{Ch}(\Bbbk)}_{\infty}}\ar[d]_{\sim}& \mathtt{uA}^{\operatorname{Ch}(\Bbbk)}_{\infty}\ar[d]\\
\mathtt{Ass}^{\operatorname{Ch}(\Bbbk)}\ar@{ >->}[r]^{\bar\phi^{\operatorname{Ch}(\Bbbk)}_{\infty}}&\mathtt{u}_{\infty}\mathtt{A}^{\operatorname{Ch}(\Bbbk)}}$$
where the left vertical arrow is the previous weak equivalence and the right vertical arrow is given by 
\begin{align*}
\mu_2^{\varnothing}&\mapsto\mu,\\
\mu_n^{\varnothing}&\mapsto 0,\quad n>2,\\
\mu_n^{S}&\mapsto\nu_n^S\quad\text{otherwise}.
\end{align*}
The differential in $\mathtt{u}_{\infty}\mathtt{A}^{\operatorname{Ch}(\Bbbk)}$ has been defined so that this is indeed a morphism of DG-operads. This square clearly commutes. Moreover, it is a push-out by the very definition of the underlying graded operads. In particular we deduce that the right vertical arrow is a weak equivalence, see Remark \ref{construcciones}.

We should finally remark that algebras over $\mathtt{uA}_{\infty}^{\operatorname{Ch}(\Bbbk)}$ were 
first considered in \cite{fooo1,fooo2}.
\end{rem}

\begin{rem}\label{uuu2}
The \emph{$u$-infinity unital associative DG-operad}   $\mathtt{u}_{\infty}\mathtt{uA}^{\operatorname{Ch}(\Bbbk)}$ in the sense of Definition \ref{lol} 
admits a graded operad presentation extending the presentation of $\mathtt{u}_{\infty}\mathtt{A}^{\operatorname{Ch}(\Bbbk)}$ above along $\psi^{\operatorname{Ch}(\Bbbk)}$ with one more generator,
$$u\in\mathtt{u}_{\infty}\mathtt{uA}^{\operatorname{Ch}(\Bbbk)}(0)_0,$$
and two more relations,
$$\mu\circ_1u=\id{}=\mu\circ_2u.$$
The differential is defined as in $\mathtt{u}_{\infty}\mathtt{A}^{\operatorname{Ch}(\Bbbk)}$ together with
$$d(u)=0.$$
The morphism $\varphi^{\operatorname{Ch}(\Bbbk)}$ is the inclusion given by $\varphi^{\operatorname{Ch}(\Bbbk)}(\mu)=\mu$ and $\varphi^{\operatorname{Ch}(\Bbbk)}(u)=u$.
\end{rem}


We now aim at proving that the inclusion  $\varphi^{\operatorname{Ch}(\Bbbk)}\colon \mathtt{uAss}^{\operatorname{Ch}(\Bbbk)}\subset \mathtt{u}_{\infty}\mathtt{uA}^{\operatorname{Ch}(\Bbbk)}$ is a weak equivalence. With this purpose, we define a countable filtration of $\mathtt{u}_{\infty}\mathtt{uA}^{\operatorname{Ch}(\Bbbk)}$ starting with $\mathtt{uAss}^{\operatorname{Ch}(\Bbbk)}$.

\begin{defn}
For $m\geq 0$, let $\mathtt{u}_{m}\mathtt{uA}^{\operatorname{Ch}(\Bbbk)}\subset \mathtt{u}_{\infty}\mathtt{uA}^{\operatorname{Ch}(\Bbbk)}$  be the sub-DG-operad spanned by $\mu$, $u$, and the $\nu_n^S$ with $|S|\leq m$. In particular $\mathtt{u}_{0}\mathtt{uA}^{\operatorname{Ch}(\Bbbk)}=
\mathtt{uAss}^{\operatorname{Ch}(\Bbbk)}$ and 
$$\mathtt{u}_{\infty}\mathtt{uA}^{\operatorname{Ch}(\Bbbk)}=\bigcup_{m\geq0} \mathtt{u}_{m}\mathtt{uA}^{\operatorname{Ch}(\Bbbk)}.$$
\end{defn}

\begin{rem}\label{itis}
The inclusion $\mathtt{u}_{m-1}\mathtt{uA}^{\operatorname{Ch}(\Bbbk)}\subset \mathtt{u}_{m}\mathtt{uA}^{\operatorname{Ch}(\Bbbk)}$, $m> 0$, is a relative $\mathcal F_{\operatorname{Ch}(\Bbbk)}(I_{\mathbb N})$-cell complex by Proposition \ref{estructura}. Indeed, as graded operads $$\mathtt{u}_{m}\mathtt{uA}^{\operatorname{Ch}(\Bbbk)}=\mathtt{u}_{m-1}\mathtt{uA}^{\operatorname{Ch}(\Bbbk)}\amalg
\mathcal F(\Bbbk\cdot\{\nu_{n}^{S}\,;\,|S|=m\}_{n\geq m}),$$
and the differential satisfies
$$d(\{\nu_{n}^{S}\,;\,|S|=m\}_{l+1\geq n\geq m})\subset \mathtt{u}_{m-1}\mathtt{uA}^{\operatorname{Ch}(\Bbbk)}\amalg
\mathcal F(\Bbbk\cdot\{\nu_{n}^{S}\,;\,|S|=m\}_{l\geq n\geq m})$$
for any $l\geq m-1$. In particular $\bar\phi_\infty^{\operatorname{Ch}(\Bbbk)}$ is a relative $\mathcal F_{\operatorname{Ch}(\Bbbk)}(I_{\mathbb N})$-cell complex.
\end{rem}


We now show that each step of this filtration is a weak equivalence. 

In the proof of the following lemma we use modules over a graded operad $\mathcal O$, as introduced in \cite[Definition 2.13]{cmmc} under the name of linear modules. We use the alternative description in \cite[Definition 1.4]{mfo}. An \emph{$\mathcal O$-module} is a sequence of graded modules $M$ equipped with composition laws, $1\leq i\leq p$, $q\geq 0$,
$$\circ_i\colon \mathcal{O}(p)\otimes M(q)\To M(p+q-1),
\qquad 
\circ_i\colon M(p)\otimes \mathcal O(q)\To M(p+q-1),$$
satisfying equations $(1')$, $(2)$, $(3)$ and $(4)$ in Remark \ref{examples} when allowing one of the variables $a$, $b$ or $c$ to be in $M$. They form a graded abelian category. The graded operad $\mathcal O$ is itself an $\mathcal O$-module and any graded operad morphism $\mathcal O\r\mathcal P$ endows $\mathcal P$ with an $\mathcal O$-module  structure. A similar notion exists for prop(erad)s \cite[\S3.1]{dtrpI}, called infinitesimal bimodule. Bimodule is probably a more appropriate name for this structure, but we wish to follow \cite{cmmc,mfo}.

If $\mathcal P=\mathcal O\amalg\mathcal F(V)$ and $\mathcal O\r\mathcal P$ is the inclusion of the first factor, the sub-$\mathcal O$-module of $\mathcal P$ generated by the operad identity $\id{}$ is $\mathcal O$. We denote by $\mathcal P_1\subset\mathcal P$ the  sub-$\mathcal O$-module generated by the sequence of graded modules $V$.  
In the coproduct decomposition of the sequequence of graded modules underlying $\mathcal P$, see Remark \ref{binary}, $\mathcal P_1$ corresponds to the factors indexed by trees $T$ with exactly one inner vertex of even level (necessarily $2$). This description shows that $\mathcal P_1$ is freely generated by $V$ as an $\mathcal O$-module, compare the example after \cite[Proposition 18]{dtrpII}. The coproduct of $\mathcal O$-modules $\mathcal O\amalg\mathcal P_1\subset\mathcal P$ is called the \emph{linear part} of $\mathcal P$.

If $\mathcal P'=\mathcal O\amalg\mathcal F(V')$ and $\Phi\colon \mathcal P\r\mathcal P'$ is a morphism of graded operads  under $\mathcal O$ satisfying $\Phi(V)\subset \mathcal O\amalg\mathcal P_1'$ we say that $\Phi$ is \emph{linear}. In this case $\Phi$ (co)restricts to a degree $0$ morphism of $\mathcal O$-modules $\Phi\colon \mathcal O\amalg\mathcal P_1\r\mathcal O\amalg\mathcal P_1'$ which (co)restricts to the identity between the first factors.

Suppose now that $\mathcal P$ above is equipped with a DG-operad structure. We say that the differential of $\mathcal P$ is \emph{linear} if $d(\mathcal O)=0$ and $d(V)\subset\mathcal O\amalg\mathcal P_1$. In this case $d$ (co)restricts to a degree $-1$ morphism of $\mathcal O$-modules $d\colon\mathcal O\amalg\mathcal P_1\r \mathcal O\amalg\mathcal P_1$ which vanishes on the first factor.

\begin{lem}\label{gordo}
The inclusion $\mathtt{u}_{m-1}\mathtt{uA}^{\operatorname{Ch}(\Bbbk)}\subset \mathtt{u}_{m}\mathtt{uA}^{\operatorname{Ch}(\Bbbk)}$ is a weak equivalence for any $m> 0$.
\end{lem}


\begin{proof}
We proceed by induction on $m>0$. Assume that the previous inclusions $$\mathtt{uAss}^{\operatorname{Ch}(\Bbbk)}=\mathtt{u}_{0}\mathtt{uA}^{\operatorname{Ch}(\Bbbk)}\subset\cdots\subset \mathtt{u}_{m-1}\mathtt{uA}^{\operatorname{Ch}(\Bbbk)}$$ are weak equivalences. There is a retraction 
$r\colon \mathtt{u}_{m-1}\mathtt{uA}^{\operatorname{Ch}(\Bbbk)}\r \mathtt{uAss}^{\operatorname{Ch}(\Bbbk)}$ of DG-operads defined for $m\geq 2$ by
\begin{align*}
r(\nu_1^{\{1\}})&=u,&
r(\nu_n^{S})&=0,\quad (n,|S|)\neq (1,1).
\end{align*}
These formulas clearly define a retraction of graded operads, hence we only have to check compatibility with differentials. It is easy to check compatibility for $\nu_1^{\{1\}}$, $\nu_2^{\{1\}}$ and $\nu_2^{\{2\}}$. For the rest of $\nu_n^S$,
compatibility is the same as $rd(\nu_n^{S})=0$. This equation follows  since 
all summads in $d(\nu_n^{S})$ contain a $\nu_{n'}^{S'}\neq \nu_1^{\{1\}}$, and $r(\nu_{n'}^{S'})=0$.

By the 2-out-of-3 axiom, $r$ is a weak equivalence. Consider the following push-out
$$\xymatrix{
\mathtt{u}_{m-1}\mathtt{uA}^{\operatorname{Ch}(\Bbbk)}\ar[d]^\sim_{r}\ar@{}[rd]|{\text{push}}\ar@{>->}[r]&
\mathtt{u}_{m}\mathtt{uA}^{\operatorname{Ch}(\Bbbk)}\ar[d]^{\bar r}_\sim\\
\mathtt{uAss}^{\operatorname{Ch}(\Bbbk)}\ar@{>->}[r]&
\mathcal{P}
}$$
Here $\bar r$ is a weak equivalence by \cite[Corollary C.2 and Theorem C.7]{htnso2}. If we show that the lower cofibration is a weak equivalence then the upper one will also be a weak equivalence by he 2-out-of-3 axiom.

Let $\mathcal Q$ be the DG-operad
$$\mathcal Q= \mathtt{uAss}^{\operatorname{Ch}(\Bbbk)}\amalg\coprod_{n,S}\mathcal F_{\operatorname{Ch}(\Bbbk)}(D^{n+m-1}[n-m]).$$
Here $n\geq 1$ and $S$ runs over the subsets $S\subset\{1,\dots,m\}$ with $m$ elements. 
The inclusion of the first factor $\mathtt{uAss}^{\operatorname{Ch}(\Bbbk)}\st{\sim}\into \mathcal Q$ is a relative $\mathcal F_{\operatorname{Ch}(\Bbbk)}(J_{\mathbb N})$-cell complex, in particular a trivial cofibration. We are going to construct morphisms of DG-operads $\Phi$ and $\Psi$ under $ \mathtt{uAss}^{\operatorname{Ch}(\Bbbk)}$,
$$\xymatrix@R=10pt{
&\mathcal P
\ar
@{<-}[dd]_\Phi\\
\mathtt{uAss}^{\operatorname{Ch}(\Bbbk)}\ar@{>->}[ru]\ar@{>->}[rd]_\sim&\\
&\mathcal Q
}\qquad\qquad\qquad
\xymatrix@R=10pt{
&\mathcal P\ar
[dd]^\Psi
\\
\mathtt{uAss}^{\operatorname{Ch}(\Bbbk)}\ar@{>->}[ru]\ar@{>->}[rd]_\sim&\\
&\mathcal Q
}$$
such that $\Phi\Psi=1_{\mathcal P}$. Assume we have done it. The second commutative triangle shows that $\Psi$ must be surjective in homology. Moreover, $\Psi$ is also injective in homology since $\Phi\Psi=1_{\mathcal P}$. Therefore $\Psi$ is a weak equivalence, and also $\mathtt{uAss}^{\operatorname{Ch}(\Bbbk)}\into \mathcal P$ by the 2-out-of-3 property applied to the second triangle. This will finish the proof.

The DG-operad $\mathcal{P}$ is presented as a graded operad by the generators
\begin{align*}
u&\in \mathcal{P}(0)_{0}, &
\mu&\in \mathcal{P}(2)_{0}, &
\nu_n^{S}&\in \mathcal{P}(n-m)_{n-2+m},
\end{align*}
where $n>0$ and $S=\{l_1,\dots, l_m\}\subset\{1,\dots,n\}$ is a subset with $m$ elements,
and relations $\mu\circ_1\mu=\mu\circ_2\mu$ and $\mu\circ_1u=\id{}=\mu\circ_2u$. 
The differential is given by 
\begin{align*}
d(u)&=0,&
d(\mu)&=0;
\end{align*}
if $(n,m)\neq (2,1), (1,1)$,
\begin{align*}
d(\nu_n^S)={}&(-1)^{n}\mu\circ_1\nu_{n-1}^{S}\quad\text{unless }l_m=n\\
&+\mu\circ_2\nu_{n-1}^{S-1}\quad\text{unless }l_1=1\\
&+\hspace{-20pt}\sum_{
\begin{array}{c}\\[-17pt]
\scriptstyle 1\leq v\leq m+1\\[-5pt]
\scriptstyle l_{v-1}< i+v-1< l_{v}-1
\end{array}
}
\hspace{-20pt}
(-1)^{i+v-1}\nu_{n-1}^{S_{v}\cup(S_{v}'-1)}\circ_i\mu;
\end{align*}
and if $m=1$ also
\begin{align*}
d(\nu_1^{\{1\}})&=0,&
d(\nu_2^{\{1\}})&=\mu\circ_1\nu_1^{\{1\}}-\id{},&
d(\nu_2^{\{2\}})&=\mu\circ_2\nu_1^{\{1\}}-\id{}.
\end{align*}
Notice that $\mathcal P$, whose underlying graded operad is $\mathtt{uAss}^{\operatorname{Ch}(\Bbbk)}\amalg\mathcal F(\Bbbk\cdot\{\nu_n^S\}_{n,S})$, has a linear differential regarded as an operad under $\mathtt{uAss}^{\operatorname{Ch}(\Bbbk)}$. 

As a graded operad $\mathcal Q=\mathtt{uAss}^{\operatorname{Ch}(\Bbbk)}\amalg\mathcal F(\Bbbk\cdot\{\sigma_n^S,d(\sigma_n^S)\}_{n,S})$. Here $\sigma_n^S$ is the top generator of the copy of $D^{n+m-1}$ indexed by $n$ and $S$, so it has degree $n+m-1$ and arity $n-m$. The differential of $\mathcal Q$ is linear too. 

The morphism of DG-poperads $\Phi$ under $\mathtt{uAss}^{\operatorname{Ch}(\Bbbk)}$ is defined by 
\begin{align*}
\Phi(\sigma_n^S)={}&(-1)^{l_1+1}\nu_{n+1}^{S+1}\circ_{l_1}u.
\end{align*}
Notice that $\Phi$ is linear.

The definition of $\Psi$ is more complicated. Let $h\colon\mathtt{uAss}^{\operatorname{Ch}(\Bbbk)}\amalg\mathcal P_1\r\mathtt{uAss}^{\operatorname{Ch}(\Bbbk)}\amalg \mathcal Q_1$ be the degree $+1$ morphism  of $\mathtt{uAss}^{\operatorname{Ch}(\Bbbk)}$-modules defined by
\begin{align*}
h(\nu_n^S)&=\sigma_n^S,&
h(\id{})&=0.
\end{align*}
Moreover, let $p\colon\mathtt{uAss}^{\operatorname{Ch}(\Bbbk)}\amalg \mathcal P_1\r\mathtt{uAss}^{\operatorname{Ch}(\Bbbk)}\amalg \mathcal Q_1$ be the degree $0$ morphism given~by 
\begin{align*}
p(\id{})&=\id{},\\
p(\nu_n^S)&=0\quad\text{if }n>1,\\
p(\nu_1^{\{1\}})&=u\quad\text{if }m=1.
\end{align*}
We define $\Psi$ as a morphism of graded operads under $\mathtt{uAss}^{\operatorname{Ch}(\Bbbk)}$ by
\begin{align*}
\Psi(\nu_n^S)=dh(\nu_n^S)+hd(\nu_n^S)+p(\nu_n^S)={}&d(\sigma_n^S)+hd(\nu_n^S)\\
&+u\text{ if }(n,m)=(1,1).
\end{align*}
Let us check that $\Psi$ is compatible with differentials. Notice that $\Psi$ is linear by definition and the formula 
$$\Psi=dh+hd+p$$
holds on the $\mathtt{uAss}^{\operatorname{Ch}(\Bbbk)}$-module generators $\id{}$ and $\nu_n^S$ of $\mathtt{uAss}^{\operatorname{Ch}(\Bbbk)}\amalg\mathcal P_1$, hence it holds
after (co)restricting to linear parts, i.e.~when we evaluate each side at an element in $\mathtt{uAss}^{\operatorname{Ch}(\Bbbk)}\amalg\mathcal P_1$ we obtain an equality in $\mathtt{uAss}^{\operatorname{Ch}(\Bbbk)}\amalg\mathcal Q_1$. Then,
\begin{align*}
\Psi d(\nu_n^S)=dhd(\nu_n^{S})+hd^2(\nu_n^S)+pd(\nu_n^S)={}&dhd(\nu_n^{S}),\\
d\Psi(\nu_n^S)=d^2h(\nu_n^S)+dhd(\nu_n^S)+dp(\nu_n^S)={}&dhd(\nu_n^S).
\end{align*}
Here we use that $pd(\nu_n^S)=0=dp(\nu_n^S)$ for any $\nu_n^S$. Indeed, $p(\nu_n^S)=0$ if $(n,m)\neq (1,1)$ and $dp(\nu_1^{\{1\}})=d(u)=0$, so $dp(\nu_n^S)=0$ in any case. Moreover, as pointed out above, if $\nu_n^S\neq \nu_1^{\{1\}}, \nu_2^{\{1\}},\nu_2^{\{2\}}$ each summand in $d(\nu_n^S)$ contains a certain $\nu_{n'}^{S'}\neq \nu_1^{\{1\}}$, so $p(\nu_{n'}^{S'})=0$ and $pd(\nu_S)=0$. Furthermore,
\begin{align*}
pd(\nu_1^{\{1\}})&=p(0)=0,\\
pd(\nu_2^{\{1\}})&=p(\mu\circ_1\nu_1^{\{1\}}-\id{})=\mu\circ_1p(\nu_1^{\{1\}})-\id{}=\mu\circ_1u-\id{}=\id{}-\id{}=0,
\end{align*}
and similarly $pd(\nu_2^{\{2\}})=0$.

We finally prove that $\Phi\Psi=1_{\mathcal P}$. It is enough to check the equation on generators $\nu_n^S$ since $\Phi$ and $\Psi$ are morphisms of DG-operads under $\mathtt{uAss}^{\operatorname{Ch}(\Bbbk)}$. We start with the general case $(n,m)\neq (1,1), (2,1)$:
\begin{align*}
\Phi\Psi(\nu_n^S)={}&\Phi(d(\sigma_n^S)+hd(\nu_n^S))\\
={}&d\Phi(\sigma_n^S)+\Phi hd(\nu_n^S)\\
={}&(-1)^{l_1+1}d(\nu_{n+1}^{S+1}\circ_{l_1}u)\\
&+\overbrace{(-1)^{n+l_{1}+1}\mu\circ_1(\nu_{n}^{S+1}\circ_{l_{1}}u)}^{(\text{a})\text{ unless }l_m=n}+
\overbrace{(-1)^{l_{1}}\mu\circ_2(\nu_{n}^{S}\circ_{l_{1}-1}u)}^{(\text{b})\text{ unless }l_1=1}\\
&+\sum_{0< i<l_{1}-1}
\overbrace{(-1)^{i+l_{1}}(\nu_{n}^{S}\circ_{l_{1}-1}u)\circ_i\mu}^{(\text{c}_{i})}\\
&+\hspace{-20pt}\sum_{
\begin{array}{c}\\[-17pt]
\scriptstyle 1< v\leq m+1\\[-5pt]
\scriptstyle l_{v-1}< i+v-1< l_{v}-1
\end{array}
}\hspace{-20pt}\overbrace{(-1)^{i+v+l_{1}}(\nu_{n}^{(S_{v}+1)\cup S'_{v}}\circ_{l_{1}}u)\circ_i\mu}^{(\text{d}_{i})},\\
(-1)^{l_1+1}d(\nu_{n+1}^{S+1}\circ_{l_1}u)={}&\overbrace{(-1)^{n+l_{1}}
(\mu\circ_1\nu_{n}^{S+1})\circ_{l_{1}}u}^{-\text{(a)}\text{ unless }l_m=n}+
\overbrace{(-1)^{l_{1}+1}
(\mu\circ_2\nu_{n}^{S})\circ_{l_{1}}u}^{(\text{b}')}\\
&+\sum_{
0< i< l_{1}
}\overbrace{(-1)^{i+l_{1}+1}(\nu_{n}^{S}\circ_{i}\mu)\circ_{l_{1}}u}^{-(\text{c}_{i})}
\\
&+\hspace{-20pt}\sum_{
\begin{array}{c}\\[-17pt]
\scriptstyle 1< v\leq m+1\\[-5pt]
\scriptstyle l_{v-1}+1< i+v-1< l_{v}
\end{array}
}\hspace{-20pt}\overbrace{(-1)^{i+v+l_{1}}(\nu_{n}^{(S_{v}+1)\cup S'_{v}}\circ_i\mu)\circ_{l_{1}}u}^{(\text{d}'_{i})}.
\end{align*}
The summands $(\text{a})$  and $-(\text{a})$  either do not occur (if $l_m=n$) or cancel. Moreover, 
$(\text{d}_{i})=-(\text{d}_{i+1}')$, so all the $(\text{d}_{i})$ and the $(\text{d}'_{i})$ cancel. 
If $l_{1}=1$, there are no $(\text{c}_{i})$, 
$(\text{b})$ does not occur, and 
$$(\text{b}')=(\mu\circ_2\nu_{n}^{S})\circ_{1}u
=(\mu\circ_{1}u)\circ_1\nu_{n}^{S}=\id{}\circ_1\nu_{n}^{S}=\nu_{n}^{S}.$$
If $l_{1}>1$ then $(\text{b})=-(\text{b}')$ and all the $(\text{c}_{i})$ cancel 
except for the last one: 
$$-(\text{c}_{l_{1}-1})=(\nu_{n}^{S}\circ_{l_{1}-1}\mu)\circ_{l_{1}}u
=\nu_{n}^{S}\circ_{l_{1}-1}(\mu\circ_{2}u)
=\nu_{n}^{S}\circ_{l_{1}-1}\id{}=\nu_{n}^{S}.$$
Therefore $\Phi\Psi(\nu_n^S)=\nu_n^S$.

Let us finally check the special cases $(n,m)=(1,1),(2,1)$:
\begin{align*}
\Phi\Psi(\nu_1^{\{1\}})={}&\Phi(d(\sigma_1^{\{1\}})+hd(\nu_1^{\{1\}})+u)\\
={}&d\Phi(\sigma_1^{\{1\}})+u\\
={}&d(\nu_2^{\{2\}}\circ_1u)+u\\
={}&(\mu\circ_2\nu_1^{\{1\}}-\id{})\circ_1u+u\\
={}&(\mu\circ_1u)\circ_1\nu_1^{\{1\}}-u+u\\
={}&\id{}\circ_1\nu_1^{\{1\}}\\
={}&\nu_1^{\{1\}},\\
\Phi\Psi(\nu_2^{\{1\}})={}&\Phi(d(\sigma_2^{\{1\}})+hd(\nu_2^{\{1\}}))\\
={}&d\Phi(\sigma_2^{\{1\}})+\Phi hd(\nu_2^{\{1\}})\\
={}&d(\nu_3^{\{2\}}\circ_1u)+\Phi h(\mu\circ_1\nu_1^{\{1\}}-\id{})\\
={}&(-\mu\circ_1\nu_2^{\{2\}}+\mu\circ_2\nu_2^{\{1\}})\circ_1u+\mu\circ_1(\nu_2^{\{2\}}\circ_1u)\\
={}&(\mu\circ_1 u)\circ_1\nu_2^{\{1\}}\\
={}&\nu_2^{\{1\}},\\
\Phi\Psi(\nu_2^{\{2\}})={}&\Phi(d(\sigma_2^{\{2\}})+hd(\nu_2^{\{2\}}))\\
={}&d\Phi(\sigma_2^{\{2\}})+\Phi hd(\nu_2^{\{2\}})\\
={}&-d(\nu_3^{\{3\}}\circ_2u)+\Phi h(\mu\circ_2\nu_1^{\{1\}}-\id{})\\
={}&-(\mu\circ_2\nu_2^{\{2\}}-\nu_2^{\{2\}}\circ_1\mu)\circ_2u+\mu\circ_2(\nu_2^{\{2\}}\circ_1u)\\
={}&\nu_2^{\{2\}}\circ_1(\mu\circ_2u)\\
={}&\nu_2^{\{2\}}.
\end{align*}
\end{proof}

\begin{cor}\label{ultimo}
The morphism $\varphi^{\operatorname{Ch}(\Bbbk)}$ in Remark \ref{uuu2} is a trivial cofibration.
\end{cor}

\begin{proof}
By Remark \ref{itis} and Lemma \ref{gordo}, $\varphi^{\operatorname{Ch}(\Bbbk)}$ is a transfinite (countable) composition of trivial cofibrations. Hence $\varphi^{\operatorname{Ch}(\Bbbk)}$ is itself a trivial cofibration.
\end{proof}

\begin{rem}\label{laim1}
There are friendly  characterizations of the image of the injective map
$$\pi_0(\phi^{\operatorname{Ch}(\Bbbk)})^*\colon\pi_0\rmap_{\operad{\operatorname{Ch}(\Bbbk)}}(\mathtt{uAss}^{\operatorname{Ch}(\Bbbk)}, \mathcal O)
\hookrightarrow
\pi_0\rmap_{\operad{\operatorname{Ch}(\Bbbk)}}( \mathtt{Ass}^{\operatorname{Ch}(\Bbbk)}, \mathcal O)$$
when $\mathcal O=\mathtt{End}_{\operatorname{Ch}(\Bbbk)}(X)$ is the endomorphism operad of a cofibrant complex $X$. Since all complexes are fibrant, these sets are the sets of homotopy classes of maps from the (unital) $A$-infinity DG-operad to $\mathtt{End}_{\operatorname{Ch}(\Bbbk)}(X)$, i.e.~homotopy classes of (unital) $A$-infinity structures on $X$. 

Recall from Remark \ref{elu} the description of the (unital) $A$-fininity DG-operad $(\mathtt{u})\mathtt{A}_\infty^{\operatorname{Ch}(\Bbbk)}$ and the morphism $\phi_\infty^{\operatorname{Ch}(\Bbbk)}\colon \mathtt{A}_\infty^{\operatorname{Ch}(\Bbbk)}\into\mathtt{uA}_\infty^{\operatorname{Ch}(\Bbbk)}$. An \emph{$A$-infinity structure} on $X$ is given by graded morphisms $m_n\colon X^{\otimes n}\r X$ of degree $n-2$ satisfying certain equations. A \emph{unital $A$-infinity structure} is similarly defined by graded morphisms $m_n^S\colon X^{\otimes n-|S|}\r X$ of degree $n-2+|S|$. The underlying $A$-infinity structure is given by the morphisms $m_n=m_n^\varnothing\colon X^{\otimes n}\r X$. A unital $A$-infinity structure  is \emph{strict} if $m_n^S=0$ for $S\neq \varnothing$ and $n>1$. 

Following the terminology in \cite[Definition 5.2.3]{lurieha}, we say that an $A$-infinity structure is \emph{quasi-unital} if there exists a cycle $v\in X_0$, $d(v)=0$, such that the chain maps $m_2(v,-), m_2(-,v)\colon X\r X$ are chain homotopic to the identity. The underlying $A$-infinity structure of a unital $A$-infinity structure is quasi-unital. We can take $v=m_{1}^{\{1\}}\in X_{0}$ and the chain homotopies $m_{2}^{\{1\}},m_{2}^{\{2\}}\colon X\r X$.

If $\Bbbk$ is a field we can suppose that $X$ has trivial differential. Then quasi-unital means that $m_2$ has a unit. It is well known that such a quasi-unital $A$-infinity structure is quasi-isomorphic, in the $A$-infinity sense, to a strictly unital $A$-infinity structure, via a quasi-isomorphism whose linear term is the identity \cite[\S3.2.1]{slaic}. One can check that two $A$-infinity structures which are quasi-isomorphic in this way represent the same element in $\pi_0\rmap_{\operad{\operatorname{Ch}(\Bbbk)}}( \mathtt{Ass}^{\operatorname{Ch}(\Bbbk)}, \mathtt{End}_{\operatorname{Ch}(\Bbbk)}(X))$. Therefore, the image of $\pi_0(\phi^{\operatorname{Ch}(\Bbbk)})^*$ is formed exactly by the homotopy classes of quasi-unital
$A$-infinity structures. A deeper result of Lyubashenko and Manzyuk \cite[Theorem 3.7]{uainfcat}  shows that this statement is actually true over any commutative ring $\Bbbk$.

We do not think that a similar result holds for any target operad $\mathcal O$. Our opinion is based in the following facts.

The generalization of a quasi-unital $A$-infinity structure can be straightforwadly defined as follows. A morphism $\xi\colon \mathtt{A}_\infty^{\operatorname{Ch}(\Bbbk)}\r\mathcal O$ is \emph{quasi-unital} if there exists a cycle $v\in\mathcal O(0)_0$, $d(v)=0$, such that the following three  homology classes coincide
$$[\id{\mathcal{O}}]=[\xi(\mu_2)\circ_1v]=[\xi(\mu_2)\circ_2v]\in H_0(\mathcal{O}(1)).$$
If $\xi$ extends to $\mathtt{uA}_\infty^{\operatorname{Ch}(\Bbbk)}$ by a morphism $\bar\xi\colon \mathtt{uA}_\infty^{\operatorname{Ch}(\Bbbk)}\r\mathcal O$  then $\xi$ is quasi-unital. We can take $v=\bar \xi(\mu_1^{\{1\}})$ since for $j=1,2$,
\begin{align*}
d(\bar\xi(\mu_2^{\{j\}}))=\bar\xi(d(\mu_2^{\{j\}}))&
=\bar\xi(\mu_2^\varnothing\circ_j\mu_1^{\{1\}}-\id{\mathtt{uA}_\infty^{\operatorname{Ch}(\Bbbk)}})
=\xi(\mu_2)\circ_j\bar\xi(\mu_1^{\{1\}})-\id{\mathcal O}.
\end{align*}
Actually, it is enough that $\xi$ extends to the suboperad $\mathcal P\subset  \mathtt{uA}_\infty^{\operatorname{Ch}(\Bbbk)}$ spanned by $\mu_{n}^{\varnothing}$, $n\geq 2$, $\mu_1^{\{1\}}$, $\mu_2^{\{1\}}$, and $\mu_2^{\{2\}}$. In particular, the inclusion $\mathtt{A}_\infty^{\operatorname{Ch}(\Bbbk)}\subset\mathcal P$ is quasi-unital.


The image of the injective map $\pi_0(\phi^{\operatorname{Ch}(\Bbbk)})^*$ consists of homotoy classes with a quasi-unital representative. 
Suppose that, conversely,  all homotoy classes with a quasi-unital representative where in the image for any target operad $\mathcal O$. Taking $\mathcal O=\mathcal P$  we would obtain a morphism $\bar\xi\colon \mathtt{uA}_\infty^{\operatorname{Ch}(\Bbbk)}\r\mathcal P$ whose restriction to $\mathtt{A}_\infty^{\operatorname{Ch}(\Bbbk)}$ would be homotopic to the inclusion. 
Below we show that the composition of $\bar\xi$ followed by the inclusion $\mathcal P\subset  \mathtt{uA}_\infty^{\operatorname{Ch}(\Bbbk)}$ would be a homotopy automorphism of $\mathtt{uA}_\infty^{\operatorname{Ch}(\Bbbk)}$, so $\mathtt{uA}_\infty^{\operatorname{Ch}(\Bbbk)}$ would be a homotopy retract of $\mathcal P$. We think this is very unlikely to happen since $\mathcal P$ seems too small in terms of both size and coherence  for the quasi-unit. Nevertheless, we have been unable to reach a contradiction. 
\end{rem}

\begin{prop}
Any endomorphism of $\mathtt{uA}_\infty^{\operatorname{Ch}(\Bbbk)}$  is a homotopy automorphism. 
\end{prop}

\begin{proof}
The monoid of homotopy classes of maps from $\mathtt{uA}_\infty^{\operatorname{Ch}(\Bbbk)}$ to itself 
coincides with the endomorphism monoid of $\mathtt{uAss}^{\operatorname{Ch}(\Bbbk)}$, since the homology operad of $\mathtt{uA}_\infty^{\operatorname{Ch}(\Bbbk)}$ is $\mathtt{uAss}^{\operatorname{Ch}(\Bbbk)}$, which is concentrated in degree $0$. Recall the presentation of $\mathtt{uAss}^{\operatorname{Ch}(\Bbbk)}$ in Example \ref{sets}. Any morphism $\varphi\colon\mathtt{uAss}^{\operatorname{Ch}(\Bbbk)}\r\mathtt{uAss}^{\operatorname{Ch}(\Bbbk)}$ is determined by the image of the generators, which by arity and degree reasons must be of the form
\begin{align*}
\varphi(\mu)&=\alpha\cdot\mu,&\varphi(u)&=\beta\cdot u,&\alpha,\beta&\in\Bbbk.
\end{align*}
The relation $\id{}=\mu\circ_{2}u$ implies
\begin{align*}
\id{}&=\varphi(\id{})\\
&=\varphi(\mu\circ_{2}u)\\
&=\varphi(\mu)\circ_{2}\varphi(u)\\
&=(\alpha\cdot\mu)\circ_{2}(\beta\cdot u)\\
&=\alpha\cdot\beta\cdot(\mu\circ_{2} u)\\
&=\alpha\cdot\beta\cdot\id{}.
\end{align*}
Therefore $\alpha\cdot\beta=1$, i.e.~$\alpha\in\Bbbk^{\times}$ and $\beta=\alpha^{-1}$. One can conversely check that for any $\alpha\in\Bbbk^{\times}$, $\mu\mapsto\alpha\mu$ and $u\mapsto\alpha^{-1}u$ define an automorphism of $\mathtt{uAss}^{\operatorname{Ch}(\Bbbk)}$. Hence,
$$\pi_0\rmap_{\operad{\operatorname{Ch}(\Bbbk)}}(\mathtt{uA}_\infty^{\operatorname{Ch}(\Bbbk)},\mathtt{uA}_\infty^{\operatorname{Ch}(\Bbbk)})=\aut_{\operad{\operatorname{Ch}(\Bbbk)}}
(\mathtt{uAss}^{\operatorname{Ch}(\Bbbk)})\cong\Bbbk^\times.$$
\end{proof}

\section{Transference}\label{transavia}

In this section we prove the main result of this paper, Theorem \ref{hepi}. The following proposition is a direct consequence of \cite[Proposition 4.1, Corollary C.5, and
Theorem E.2
]{htnso2} and Corollary \ref{trans1}. It allows to transfer our previous main results to a wide class of symmetric monoidal model categories. 

\begin{prop}\label{trans2}
Let $F\colon\C{V}\rightleftarrows\C{W}\colon G$ be a weak symmetric monoidal Quillen adjunction between symmetric monoidal model categories as in Theorem \ref{modelo}. Suppose that $\C{V}$ and $\C{W}$ satisfy the strong unit axiom and $F\dashv G$ satisfies the pseudo-cofibrant axiom and the $\unit$-cofibrant axiom. The following statements hold:
\begin{enumerate}
 \item If $\phi^{\C{V}}\colon \mathtt{Ass}^{\C{V}}\r \mathtt{uAss}^{\C{V}}$ is a homotopy epimorphism in $\operad{\C V}$ then the map $\phi^{\C{W}}\colon \mathtt{Ass}^{\C{W}}\r \mathtt{uAss}^{\C{W}}$ is a homotopy epimorphism in $\operad{\C W}$.

\item Suppose in addition that $\mathbb L F\colon \ho\C V\r\ho\C W$ reflects isomorphisms, e.g.~if $F\dashv G$ is a Quillen equivalence. In this case the converse of $(1)$ also holds.
\end{enumerate}
\end{prop}

The strong unit axiom, the pseudo-cofibrant axiom, and the $\unit$-cofibrant axiom were introduced in \cite[Definitions A.9 and B.6]{htnso2}. The strong unit axiom holds in all symmetric monoidal model categories with cofibrant unit. The pseudo-cofibrant axiom and the $\unit$-cofibrant axiom hold in all Quillen pairs $F\dashv G$  where the source of $F$ has a cofibrant tensor unit.


We now prove the main theorem for the category $\operatorname{Ch}(\Bbbk)_{\geq 0}$ of non-negative chain complexes. Weak equivalences and the monoidal structure are defined as in $\operatorname{Ch}(\Bbbk)$, fibrations are chain maps which are surjective in positive degrees, compare \cite[\S4.1]{emmc}.

\begin{thm}\label{hepiDG+}
The morphism $\phi^{\operatorname{Ch}(\Bbbk)_{\geq 0}}\colon \mathtt{Ass}^{\operatorname{Ch}(\Bbbk)_{\geq 0}}\r \mathtt{uAss}^{\operatorname{Ch}(\Bbbk)_{\geq 0}}$ in Example \ref{comparison} is a homotopy epimorphism in $\operad{\operatorname{Ch}(\Bbbk)_{\geq 0}}$. 
\end{thm}

\begin{proof}
Consider the  symmetric monoidal adjoint pair
$$\xymatrix@C=40pt{\operatorname{Ch}(\Bbbk)_{\geq 0}\ar@<.5ex>[r]^-{\text{inclusion}}&\operatorname{Ch}(\Bbbk).\ar@<.5ex>[l]^-{t_{\geq 0}}}$$
The right adjoint $t_{\geq 0}$ is the truncation functor. It is defined by the fact that the counit $t_{\geq 0}(X)\r X$ is the identity in positive degrees and the inclusion $\ker [d\colon X_{0}\r X_{-1}]\subset X_{0}$ in degree $0$. Clearly, $t_{\geq 0}$ preserves weak equivalences and fibrations, so this is a Quillen pair. The inclusion functor reflects weak equivalences, hence its left derived functor $\ho\operatorname{Ch}(\Bbbk)_{\geq 0}\r \ho\operatorname{Ch}(\Bbbk)$ reflects isomorphisms. The categories $\operatorname{Ch}(\Bbbk)_{\geq 0}$ and $\operatorname{Ch}(\Bbbk)$ have cofibrant tensor units. These observations show that our adjunction satisfies the assumptions of Proposition \ref{trans2} (2). Therefore the result follows from Theorem \ref{hepiDG}.
\end{proof}

The following theorem is the main result for simplicial $\Bbbk$-modules with the symmetric monoidal model structure considered in \cite[\S4.1]{emmc}.

\begin{thm}\label{hepiskmod}
The morphism $\phi^{\operatorname{Mod}(\Bbbk)^{\Delta^{\op}}}\colon \mathtt{Ass}^{\operatorname{Mod}(\Bbbk)^{\Delta^{\op}}}\r \mathtt{uAss}^{\operatorname{Mod}(\Bbbk)^{\Delta^{\op}}}$ in Example \ref{comparison} is a homotopy epimorphism in $\operad{\operatorname{Mod}(\Bbbk)^{\Delta^{\op}}}$. 
\end{thm}

\begin{proof}
This follows from Theorem \ref{hepiDG+} and Proposition \ref{trans2} (1) applied to the Dold--Kan equivalence 
$\operatorname{Ch}(\Bbbk)_{\geq 0}\rightleftarrows\operatorname{Mod}(\Bbbk)^{\Delta^{\op}}$, which is a weak symmetric monoidal Quillen equivalence, see \cite[\S4.2]{emmc}. Here, both categories have cofibrant tensor unit.
\end{proof}

We can now prove the main result for simplicial sets, with the usual cartesian symmetric monoidal model structure, see \cite[Proposition 4.2.8]{hmc}.

\begin{thm}\label{hepisset}
The morphism $\phi^{\operatorname{Set}^{\Delta^\op}}\colon \mathtt{Ass}^{\operatorname{Set}^{\Delta^\op}}\r \mathtt{uAss}^{\operatorname{Set}^{\Delta^\op}}$ in Example \ref{comparison} is a homotopy epimorphism in $\operad{\operatorname{Set}^{\Delta^\op}}$. 
\end{thm}

\begin{proof}
We have to show that $\varphi^{\operatorname{Set}^{\Delta^\op}}$ in Definition \ref{lol} is a weak equivalence, see Lemma \ref{basico}, or equivalently, that $\mathtt{u}_{\infty}\mathtt{uA}^{\operatorname{Set}^{\Delta^\op}}(n)$  is contractible for all $n\geq 0$.

Consider the following two Quillen pairs,
$$\xymatrix{{\operatorname{Set}^{\Delta^{\op}}}\ar@<.5ex>[r]^-{\Pi_{1}}&\operatorname{Grd},\ar@<.5ex>[l]^-{\ner}}
\qquad\qquad\qquad
\xymatrix{{\operatorname{Set}^{\Delta^{\op}}}\ar@<.5ex>[r]^-{\mathbb{Z}\cdot-}&
\operatorname{Mod}(\mathbb{Z})^{\Delta^{\op}}.
\ar@<.5ex>[l]^-{\text{forget}}}$$
The first Quillen pair was already considered in the proof of Proposition \ref{loes}. It is a symmetric monoidal Quillen pair in the sense of \cite{hmc}. The second Quillen pair, induced by the free abelian group functor,  is also symmetric monoidal. These four functors happen to preserve weak equivalences, so they coincide with their derived functors. These adjoint pairs induce Quillen pairs between operad categories, see \cite[Proposition 4.1]{htnso2}. Applying $\Pi_{1}$ and $\mathbb Z\cdot-$ to the push-out square in Definition \ref{lol} for $\C V=\operatorname{Set}^{\Delta^\op}$, we obtain push-out diagrams in $\operad{\operatorname{Grd}}$ and $\operad{\operatorname{Mod}(\mathbb{Z})^{\Delta^{\op}}}$, respectively, 
$$\xymatrix@C=20pt{
\mathtt{Ass}^{\operatorname{Grd}}\ar@{=}[r]\ar[d]_{\phi^{\operatorname{Grd}}}&
\Pi_{1}\mathtt{Ass}^{\operatorname{Set}^{\Delta^\op}}\ar@{ >->}[rr]^{\Pi_{1}\bar\phi^{\operatorname{Set}^{\Delta^\op}}_{\infty}}\ar[d]_{\Pi_{1}\phi^{\operatorname{Set}^{\Delta^\op}}}\ar@{}[rrd]|{\text{push}}&&
\Pi_{1}\mathtt{u}_{\infty}\mathtt{A}^{\operatorname{Set}^{\Delta^\op}}\ar[d]^{\Pi_{1}\psi^{\operatorname{Set}^{\Delta^\op}}}\\
\mathtt{uAss}^{\operatorname{Grd}}\ar@{=}[r]&
\Pi_{1}\mathtt{uAss}^{\operatorname{Set}^{\Delta^\op}}\ar@{ >->}[rr]_{\Pi_{1}\varphi^{\operatorname{Set}^{\Delta^\op}}}&&\Pi_{1}\mathtt{u}_{\infty}\mathtt{uA}^{\operatorname{Set}^{\Delta^\op}}}$$ 
$$\xymatrix@C=20pt{
\mathtt{Ass}^{\operatorname{Mod}(\mathbb{Z})^{\Delta^{\op}}}\ar@{=}[r]\ar[d]_{\phi^{\operatorname{Mod}(\mathbb{Z})^{\Delta^{\op}}}}&
\mathbb{Z}\cdot \mathtt{Ass}^{\operatorname{Set}^{\Delta^\op}} \ar@{ >->}[rr]^{\mathbb{Z}\cdot \bar\phi^{\operatorname{Set}^{\Delta^\op}}_{\infty} }\ar[d]_{\mathbb{Z}\cdot \phi^{\operatorname{Set}^{\Delta^\op}} }\ar@{}[rrd]|{\text{push}}&&
\mathbb{Z}\cdot \mathtt{u}_{\infty}\mathtt{A}^{\operatorname{Set}^{\Delta^\op}} \ar[d]^{\mathbb{Z}\cdot \psi^{\operatorname{Set}^{\Delta^\op}} }\\
\mathtt{uAss}^{\operatorname{Mod}(\mathbb{Z})^{\Delta^{\op}}}\ar@{=}[r]&
\mathbb{Z}\cdot \mathtt{uAss}^{\operatorname{Set}^{\Delta^\op}} \ar@{ >->}[rr]_{\mathbb{Z}\cdot \varphi^{\operatorname{Set}^{\Delta^\op}} }&&\mathbb{Z}\cdot \mathtt{u}_{\infty}\mathtt{uA}^{\operatorname{Set}^{\Delta^\op}} }$$
Here, $\Pi_{1}\bar\phi_{\infty}^{\operatorname{Set}^{\Delta^\op}}$ and $\mathbb{Z}\cdot \bar\phi^{\operatorname{Set}^{\Delta^\op}}_{\infty}$ are models for $\bar\phi_{\infty}^{\operatorname{Grd}}$ and $\bar\phi_{\infty}^{\operatorname{Mod}(\mathbb{Z})^{\Delta^{\op}}}$, respectively, see \cite[Theorem 1.7]{htnso2}. Hence,
$\Pi_{1}\varphi^{\operatorname{Set}^{\Delta^\op}}$ and $\mathbb{Z}\cdot \varphi^{\operatorname{Set}^{\Delta^\op}}$
are weak equivalences by Theorems \ref{hepiGpd} and \ref{hepiskmod} and Lemma \ref{basico}. In particular, $\mathtt{u}_{\infty}\mathtt{uA}^{\operatorname{Set}^{\Delta^\op}}(n)$  is simply connected and has the homology of a point for all $n\geq 0$, therefore it is  contractible.
\end{proof}

Let us finally prove our main theorem.

\begin{proof}[Proof of Theorem \ref{hepi}]
It is enough to check that we can apply Proposition \ref{trans2} (1) to the structure symmetric monoidal Quillen pair $F\dashv G$ of the simplicial or complicial monoidal model category $\C V$. We are assuming that $\C V$ satisfies the strong unit axiom. The categories $\operatorname{Set}^{\Delta^\op}$ and $\operatorname{Ch}(\Bbbk)$ have cofibrant tensor units. Hence, they satisfy the strong unit axiom and $F\dashv G$ satisfies the pseudo-cofibrant axiom and the $\unit$-cofibrant axiom.
\end{proof}

\begin{rem}\label{laim2}
Once Theorem \ref{hepi} is proved, it is reasonable to wonder whether there is a friendly identification of the image of $\pi_0(\phi^{\C V})^*$ as in Remark \ref{laim1} for any $\C V$ satisfying the hypotheses of that theorem and $\mathcal O=\mathtt{End}_{\C V}(X)$ the endomorphism operad of a fibrant-cofibrant object $X$. 

Suppose for simplicity that the tensor unit $\unit$ is cofibrant. In this case, it is possible to define quasi-unital $A$-infinity algebras as follows. Let us consider a cofibrant resolution $\mathtt{A}^{\C V}_\infty\st{\sim}\onto\mathtt{Ass}^{\C V}$ which is a trivial fibration. Since $\unit$ is cofibrant, the trivial fibration $\mathtt{A}^{\C V}_\infty(2)\st{\sim}\onto\mathtt{Ass}^{\C V}(2)=\unit$ is a retraction which admits a section $\tilde g\colon \unit\r \mathtt{A}^{\C V}_\infty(2)$.  Given an $A$-infinity structure on $X$,
we define $m_{2}\colon X\otimes X\r X$ as the composite
$$X\otimes X\cong \unit\otimes X\otimes X\st{\tilde g\otimes \id{}}\To \mathtt{A}^{\C V}_\infty(2)\otimes X\otimes X\To X,$$
where the last morphism is part of the $A$-infinity structure. 
We say that an $A$-infinity structure is \emph{quasi-unital} if there exists a morphism $v\colon \unit\r X$ such that the maps $m_2(v\otimes X),m_2(X\otimes v)\colon X\r X$ are homotopic to the identity.

It looks like if \cite[Theorem 5.2.3.5]{lurieha} implied a positive answer for all $\C V$ satisfying also the hypotheses in \cite[Theorem 4.1.4.4]{lurieha}, e.g.~chain complexes and simplicial sets, but not topological spaces. However, \cite[Theorem 5.2.3.5]{lurieha} is not about moduli spaces of algebra structures, but about 
(generalizations of) Dwyer--Kan simplicial localizations of categories of algebras. The connection between these spaces was established by Rezk \cite{rezkphd} for symmetric operads and $\C{V}=\operatorname{Set}^{\Delta^{\op}}$ or $\operatorname{Mod}(\Bbbk)^{\Delta^{\op}}$. In \cite{manso}
we prove the analogous result in the non-symmetric context for any $\C V$ as in Theorem \ref{modelo}. With that result at hand, we will be able to answer positively the question raised here \cite[Remark 6.8]{manso}.  
\end{rem}


\providecommand{\bysame}{\leavevmode\hbox to3em{\hrulefill}\thinspace}
\providecommand{\MR}{\relax\ifhmode\unskip\space\fi MR }
\providecommand{\MRhref}[2]{%
  \href{http://www.ams.org/mathscinet-getitem?mr=#1}{#2}
}
\providecommand{\href}[2]{#2}

\end{document}